\documentclass[reqno,11pt]{amsart}
\usepackage{amsmath,amssymb,mathrsfs,bm,amsthm,amsfonts,comment,csquotes}
\usepackage[inline]{enumitem} 
\usepackage[usenames,dvipsnames]{xcolor}
\usepackage{hyperref}
\hypersetup{%
	colorlinks=true, linkcolor=blue,
	citecolor=ForestGreen
}
\usepackage[paper=letterpaper,margin=1in]{geometry}

\newtheorem{theorem}{Theorem}[section]
\newtheorem{lemma}[theorem]{Lemma}
\newtheorem{proposition}[theorem]{Proposition}
\newtheorem{corollary}[theorem]{Corollary}
\theoremstyle{definition}
\newtheorem{definition}[theorem]{Definition}
\newtheorem{example}{Example}
\numberwithin{example}{section}
\newtheorem{remark}{Remark}
\numberwithin{remark}{section}
\numberwithin{equation}{section}
\allowdisplaybreaks
\usepackage{acronym}

\renewcommand{\Pr}{\mathbb{P}}	
\newcommand{\Ex}{\mathbb{E}}	
\newcommand{\ind}{\mathbf{1}}	

\newcommand{\R}{\mathbb{R}} 

\newcommand{\aut}{Aut}

\newcommand{\ah}{Aut(H)}
\newcommand{\vr}{\operatorname{Var}}
\newcommand{\e}{\varepsilon}

\newcommand{\s}{\bm{s}}

\newcommand{\ceil}[1]{\lceil #1 \rceil} 
\renewcommand{\bar}{\overline}
\renewcommand{\hat}{\widehat}

\usepackage{graphicx}

\title[Motif Estimation via Subgraph Sampling]{Motif Estimation via Subgraph Sampling: The Fourth Moment Phenomenon}

\author[B.\ B.\ Bhattacharya]{Bhaswar B.~Bhattacharya}
\address{B.\ B.\ Bhattacharya,
	Department of Statistics, University of Pennsylvania,
	\newline\hphantom{\quad \ \ B.\ B.\ Bhattacharya}
	3730 Walnut Street, Philadelphia, PA 19104 USA.
}
\email{bhaswar@wharton.upenn.edu}
\author[S.\ Das]{Sayan Das}
\address{S.\ Das,
	Department of Mathematics, Columbia University,
	\newline\hphantom{\quad \ \ S. Das}
	2990 Broadway, New York, NY 10027 USA
	}
\email{sayan.das@columbia.edu}
\author[S.\ Mukherjee]{Sumit Mukherjee\textsuperscript{*}}\thanks{\textsuperscript{*}Research partially supported by NSF grant DMS-1712037}
\address{S.\ Mukherjee,
	Department of Statistics, Columbia University,
	\newline\hphantom{\quad \ \ S. Mukherjee}
	1255 Amsterdam Avenue, New York, NY 10027 USA.
}
\email{sm3949@columbia.edu}

\begin{document}
\begin{abstract} 
Network sampling is an indispensable tool for understanding features of large complex networks where it is practically impossible to search over the entire graph. In this paper we develop a framework for statistical inference for counting network motifs, such as edges, triangles, and wedges, in the widely used subgraph sampling model, where each vertex is sampled independently, and the subgraph induced by the sampled vertices is observed. We derive necessary and sufficient conditions for the consistency and the asymptotic normality of the natural Horvitz-Thompson (HT) estimator, which can be used for constructing confidence intervals and hypothesis testing for the motif counts based on the sampled graph. In particular, we show that the asymptotic normality of the HT estimator exhibits an interesting fourth-moment phenomenon, which asserts that the HT estimator (appropriately centered and rescaled) converges in distribution to the standard normal whenever its fourth-moment converges to 3 (the fourth-moment of the standard normal distribution). As a consequence, we derive the exact thresholds for consistency and asymptotic normality of the HT estimator in various natural graph ensembles, such as sparse graphs with bounded degree, Erd\H os-R\'enyi random graphs, random regular graphs, and dense graphons. 
\end{abstract}


\maketitle

\section{Introduction}

One of the main challenges in network analysis is that the observed network is often a  sample from a much larger (parent) network. This is generally due to the massive size of the network or the inability to access parts of the network, making it practically impossible to search/query over the entire graph. The central statistical question in such studies is to estimate global features of the parent network, that accounts for the bias and variability  induced by the sampling paradigm. The study of network sampling began with the results of Frank \cite{frank1977estimation,frank1978estimation} and Capobianco \cite{capobianco1972estimating}, where methods for estimating features such as connected components and graph totals were studied (see \cite{frank_2005} for a more recent survey of these results). Network sampling has since then emerged as an essential tool for estimating features of large complex networks, with applications in social networks \cite{handcock2010modeling,leskovec2006sampling,yang2015defining}, protein interaction networks \cite{rual2005towards,uetz2000comprehensive}, internet and communication networks \cite{govindan2000heuristics}, and socio-economic networks \cite{apicella2012social,bandiera2006social} (see \cite{crane_network_book,kolaczyk2009statistical,kolaczyk2017topics} for a detailed discussion of different network sampling techniques and their applications).

Counting motifs (patterns of subgraphs) \cite{milo2002network,prvzulj2004modeling} in a large network, which  encode important structural information about the geometry of the network, is an important statistical and computational problem. In this direction, various sublinear time algorithms based on edge and degree queries have been proposed for testing and estimating properties such as the average degree \cite{feige2006sums,goldreich2008approximating}, triangles \cite{triangletripartite,eden2017approximately}, stars \cite{aliakbarpour2018sublinear}, general subgraph counts \cite{gonen2011counting}, and expansion properties \cite{goldreich2011testing}. These results are, however, all based on certain adaptive queries which are unrealistic in applications where the goal is to estimate features of the network based on a single sampled graph \cite{apicella2012social,network_sampling}. In this framework, estimating features such as the degree distribution \cite{zhang2015estimating}, the number of connected components \cite{klusowski2018connected}, and the number of motifs \cite{klusowski2018counting}, have been studied recently, under various sampling schemes and structural assumptions on the parent graph.

In this paper we consider the problem of motif estimation, that is, counting the number of copies of a fixed graph $H = (V(H), E(H))$ (for example, edges, triangles, and wedges) in a large parent graph $G_n$ in the most popular and commonly used  subgraph sampling model, where each vertex of $G_n$ is sampled  independently with probability $p_n \in (0, 1)$ and the subgraph induced by these sampled vertices is observed. Here, the natural Horvitz-Thompson (HT) estimator obtained by weighting the number of copies of $H$ in the observed network by $p_n^{-|V(H)|}$ (the inverse probability of sampling a subset of size $|V(H)|$ in the graph $G_n$) is unbiased for the  true motif count. 
Very recently, Klusowski and Yu \cite{klusowski2018counting} showed that the HT estimator (for induced subgraph counts) is minimax rate optimal in the subgraph sampling model for classes of graphs with maximum degree constraints. Given this result, it becomes imperative to develop a framework for statistical inference for the motif counts in the subgraph sampling model. In this paper we derive precise conditions for the consistency and the asymptotic normality of the HT estimator, which can be used for constructing confidence intervals and hypothesis testing for the motif counts in the subgraph sampling model. The results give a complete characterization of the asymptotics of the HT estimator, thus providing a mathematical framework for evaluating its performance in different examples. We begin by formally describing the subgraph sampling model and the motif estimation problem in Section \ref{sec:subgraph_sampling}.  A summary of the results obtained is given in Section \ref{sec:theorems}.

\subsection{The Subgraph Sampling Model} 
\label{sec:subgraph_sampling}

Suppose $G_n=(V(G_n), E(G_n))$ is a simple, labeled, and undirected graph with vertex set $V(G_n)=\{1,2,\ldots,|V(G_n)|\}$ and edge set $E(G_n)$. We denote by $A(G_n)=((a_{ij}))_{i, j \in V(G_n)}$  the adjacency matrix of $G_n$, that is, $a_{ij} =1$ whenever there is an edge between $(i, j)$ and zero otherwise. In the {\it subgraph sampling model} each vertex of the graph $G_n$ is sampled independently with probability $p_n \in (0, 1)$, and we observe the subgraph induced by the sampled vertices. The parameter $p_n$ is referred to as the {\it sampling ratio} of the graph $G_n$. In the survey sampling literature this sampling scheme is also referred to as the Poisson sampling plan (see Till\'e \cite{tille2006sampling} and the references therein). The sampling scheme is illustrated in Figure \ref{fig:edgehistogram}, where the population graph and the vertices sampled (colored in red) are shown in the left and the observed graph is shown in the right. 

\begin{figure*}[h]
\centering
\begin{minipage}[c]{1.0\textwidth}
\centering
\includegraphics[width=5.25in]
    {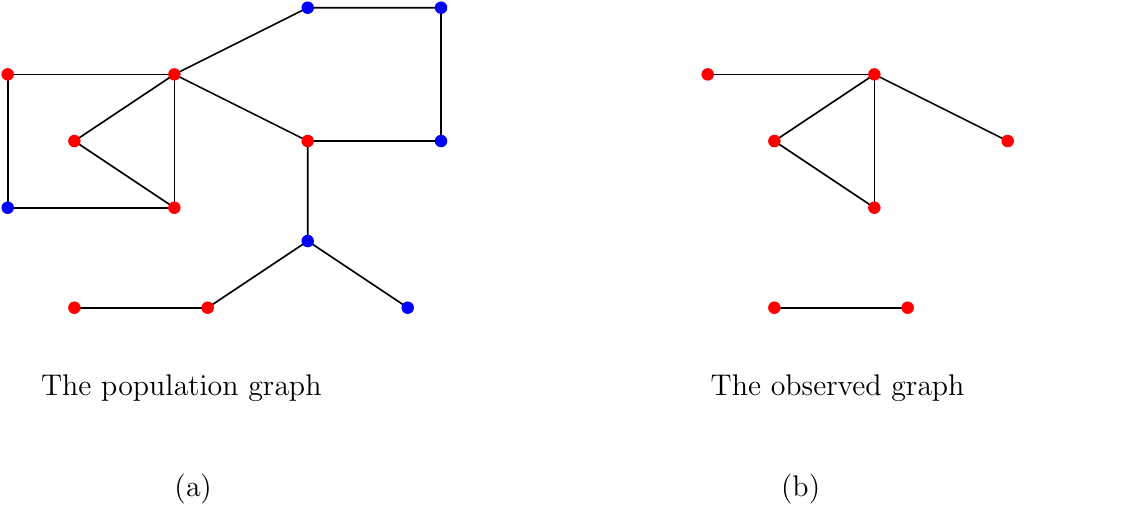}\\
\end{minipage} 
\caption{\small{The subgraph sampling scheme:  (a) The population graph and the vertices sampled (colored in red), and (b) the observed graph.}}
\label{fig:edgehistogram}
\end{figure*}

Having observed this sampled subgraph, our goal is to estimate the number of copies of a fixed connected graph $H=(V(H), E(H))$ in the parent graph $G_n$. Formally, the number of copies of $H$ in $G_n$ is given by 
\begin{align}\label{def:count}
N(H, G_n):=\frac1{|\ah|}\sum_{\s\in V(G_n)_{|V(H)|}} \prod_{(i,j)\in E(H)} a_{s_is_j} ,
\end{align}
\begin{itemize}
\item[--] $ V(G_n)_{|V(H)|}$ is the set of all $|V(H)|$-tuples ${\bm s}=(s_1,\ldots, s_{|V(H)|})\in  V(G_n)^{|V(H)|}$ with distinct indices.\footnote{For a set $S$, the set $S^N$ denotes the $N$-fold cartesian product $S\times S \times \ldots \times S$.} Thus, the cardinality of $V(G_n)_{|V(H)|}$ is $\frac{|V(G_n)|!}{(|V(G_n)|-|V(H)|)!}$. 

\item[--] $Aut(H)$ is the {\it automorphism group} of $H$, that is, the number permutations $\sigma$ of the vertex set $V(H)$ such that $(x, y) \in E(H)$ if and only if $(\sigma(x), \sigma(y)) \in E(H)$. 
\end{itemize}

Let $X_v$ be the indicator of the event that the vertex $v\in V(G_n)$ is sampled under subgraph sampling model. Note that $\{X_v\}_{v\in V(G_n)}$ is a collection of i.i.d.~$\operatorname{Ber}(p_n)$ variables.  
For $\s \in V(G_n)_{|V(H)|}$, denote $$X_{\s}:=X_{s_1} X_{s_2} \ldots X_{s_{|V(H)|}} := \prod_{u=1}^{|V(H)|} X_{s_u} \quad \text{ and } \quad M_H(\s):=\prod_{(i,j)\in E(H)}a_{s_is_j}.$$ Then the number of copies of $H$ in the sampled subgraph is given by:
\begin{align} \label{def:stat}
T(H, G_n) := \frac1{|\ah|}\sum_{\bm s\in V(G_n)_{|V(H)|}} M_H(\s) X_{\s} ,
\end{align}  
Note that $\Ex[T(H, G_n)] =p_n^{|V(H)|} N(H, G_n)$, hence 
\begin{align}\label{eq:estimate_H}
\hat N(H, G_n):=\frac{1}{p_n^{|V(H)|}}T(H, G_n). 
\end{align} 
is a natural unbiased estimator for the parameter $N(H, G_n)$. This is referred to in the literature as the Horvitz-Thompson (HT) estimator of the motif count $N(H, G_n)$ \cite{klusowski2018counting}, since it uses inverse probability weighting to achieve unbiasedness \cite{sampling_estimator}.

\subsection{Summary of Results} 
\label{sec:theorems}

In this paper, we develop a framework for statistical inference for the motif counts using the HT estimator in the subgraph sampling model. The following is a summary of the results obtained:

\begin{itemize}

\item To begin with, we establish a necessary and sufficient condition for the consistency of the HT estimator, that is, conditions under which $\hat N(H, G_n)/N(H, G_n)$ converges to 1 in probability. To this end, we introduce the notion of local count function, which counts the number of copies of $H$ incident on a fixed subset of vertices, and show that the precise condition for the consistency of the HT estimator is to ensure that subsets of vertices with `high' local counts have asymptotically negligible contribution to the total count $N(H, G_n)$ (Theorem \ref{thm:consistency}).  

\item To derive the asymptotic normality of the HT estimator we consider the rescaled statistic 
\begin{align}\label{eq:ZHGn_I}
Z(H, G_n):= \frac{\hat N(H, G_n)- N(H, G_n)}{\sqrt{\vr[\hat N(H, G_n)]}}. 
\end{align}
Using the Stein's method for normal approximation, we derive an explicit rate of convergence (in the Wasserstein's distance) between $Z(H, G_n)$ and the standard normal distribution. As a consequence, we show that $Z(H, G_n) \stackrel{D} \rightarrow N(0, 1)$, whenever the fourth-moment $\Ex[Z(H, G_n)] \rightarrow 3$ (the fourth-moment of $N(0, 1)$) (see Theorem \ref{thm:wass} for details). This is an example of the celebrated {\it fourth-moment phenomenon}, which initially appeared in the asymptotics of multiple stochastic integrals (Wiener chaos) in the seminal papers \cite{NoPe09,nualart2005central} and has, since then, emerged as the driving condition for the asymptotic normality of various non-linear functions of random fields \cite{book_np}. In the present context of motif estimation, we show that the asymptotic normality of $Z(H, G_n)$ is a consequence of a more general central limit theorem (CLT) for random multilinear forms in Bernoulli variables, a result which might be of independent interest (Theorem \ref{thm:wass1}). 

\item Next, we discuss how the CLT for $Z(H, G_n)$ can be used to compute a confidence interval for the motif count $N(H, G_n)$. Towards this, we provide an unbiased estimate of the variance of $Z(H, G_n)$ that is consistent whenever the CLT error term for $Z(H, G_n)$ goes to zero, which can be used to construct an asymptotically valid confidence interval for $N(H, G_n)$ (Proposition \ref{ppn:variance_estimation}). 

\item We then derive a necessary and sufficient condition for the asymptotic normality of $Z(H, G_n)$. For this we need to weaken the fourth-moment condition $\Ex[Z(H, G_n)^4]\rightarrow 3$, which, although sufficient, is not always necessary for the asymptotic normality of $Z(H, G_n)$. In particular, there are graph sequences for which $Z(H, G_n) \stackrel{D} \rightarrow N(0, 1)$, even though the fourth-moment condition fails (Example \ref{example:4_normal}). Instead, we show that the  asymptotic normality of $Z(H, G_n)$ is characterized by a {\it truncated fourth-moment} condition. More precisely, $Z(H, G_n)$ converges in distribution to $N(0, 1)$ if and only if the second and fourth moments of an appropriate truncation of $Z(H, G_n)$, based on the local count functions, converges to 1 and  3, respectively (Theorem \ref{thm:normality}).

\item As a consequence of the above results, we derive the exact thresholds for consistency and asymptotic normality of the HT estimator in various natural graph ensembles, such as sparse graphs with bounded degree (Proposition \ref{bdd-deg}), Erd\H os-R\'enyi random graphs (Theorem \ref{thm:er-consistency}), random regular graphs (Corollary \ref{regulargraph}), and graphons (Proposition \ref{dense}). In each of these cases there is a threshold (which depends on the graph parameters) such that if the sampling ratio $p_n$ is much larger than this threshold then the HT estimator is consistent and asymptotically normal, whereas if $p_n$ is of the same order as the threshold, the HT estimator is neither consistent nor asymptotic normal. In particular, for the Erd\H os-R\'enyi graph, the threshold for consistency and asymptotic normality depends on the well-known balancedness coefficient of the graph $H$ (Definition \ref{defn:m}), and is related to the threshold for the occurrence of $H$ is the sampled random graph. 

\end{itemize}

These results provide a comprehensive characterization of the asymptotics of the HT estimator for the motif counts in the subgraph sampling model, which can be used to validate its performance in various applications. The formal statements of the results and their various consequences are given below in Section \ref{sec:main}.

\subsection{Asymptotic Notations}\label{sec:asymptotic_notation} Throughout we will use the following standard asymptotic notations. For two positive sequences $\{a_n\}_{n\geq 1}$ and $\{b_n\}_{n\geq 1}$, $a_n = O(b_n)$ means $a_n \leq C_1 b_n$, $a_n = \Omega(b_n)$ means $a_n \geq C_2 b_n$, and $a_n = \Theta(b_n)$ means $C_2 b_n \leq a_n \leq C_1 b_n$, for all $n$ large enough and positive constants $C_1, C_2$. Similarly, $a_n \lesssim b_n$ means $a_n=O(b_n)$, and $a_n \gtrsim b_n$ means  $a_n=\Omega(b_n)$, and subscripts in the above notation,  for example $\lesssim_\square$ or $\gtrsim_\square$,  denote that the hidden constants may depend on the subscripted parameters. Moreover, $a_n \ll b_n$  means $a_n = o(b_n)$, and $a_n \gg b_n$ means $b_n=o(a_n)$. Finally, for a sequence of random variables $\{X_n\}_{n \geq 1}$ and a positive sequence $\{a_n\}_{n \geq 1}$, the notation $X_n = O_P(a_n)$ means $X_n / a_n$ is stochastically bounded, that is, $\lim_{M \rightarrow \infty} \lim_{n \rightarrow \infty}\Pr( | X_n / a_n | \leq M ) = 1$, and $X_n = \Theta_P(a_n)$ will mean  $X_n = O_P(a_n)$ and $\lim_{\delta \rightarrow 0} \lim_{n \rightarrow \infty}\Pr( | X_n / a_n | \geq \delta ) = 1$.

\section{Statements of the Main Results}
\label{sec:main}

In this section we state our main results. Throughout we will assume that there exists $\kappa \in (0, 1)$ such that 
\begin{align}\label{eq:boundp}
p_n\le 1-\kappa,
\end{align}
for all $n \geq 1$. This is to rule out the degenerate case when we observe nearly the whole graph, in which case the estimation problem becomes trivial. The rest of this section is organized as follows: The necessary and sufficient condition for the consistency of the HT estimator is discussed in Section \ref{sec:consistency}.  The precise conditions for the asymptotic normality of the HT estimator and construction of confidence intervals are given in Section \ref{sec:normal_ZHGn}. Finally, in Section \ref{sec:graphs} we compute the thresholds for consistency and asymptotic normality for various graph ensembles.

\subsection{Consistency of the HT Estimator} 
\label{sec:consistency} 

In this section we obtain the precise conditions for consistency of the HT estimator 
$\hat N(H, G_n)$, for any fixed connected motif $H$ and any sequence of graphs $\{G_n\}_{n \geq 1}$, such that $N(H, G_n) > 0$ for all $n \geq 1$. To state our results precisely, we need a few definitions. For an ordered tuple $\bm s \in V(G_n)_{|V(H)|}$ with distinct entries, denote by $\bar{\bm s}$ the (unordered) set formed by the entries of $\bm s$ (for example, if $\bm s=(4, 2, 5)$, then $\bar{\bm s}=\{2, 4, 5\}$). For any 
non-empty set $A \subset V(G_n)$ with $1\le |A| \leq |V(H)|$, define the {\it local count function} of $H$ on the set $A$ as follows: 
\begin{align}\label{def:tj}
t_H(A):=\frac1{|\ah|}\sum_{\s \in V(G_n)_{|V(H)|} : \bar{\s}\supseteq A } M_H(\s),
\end{align}
where the sum is over of all ordered $\s\in V(G_n)_{|V(H)|}$ such that the set $\bar \s$  contains all the elements of $A$. In other words, $t_H(A)$ counts the number of copies of $H$ in $G_n$ that passes through a given set $A$ of distinct vertices. 

\begin{example}\label{example:local_A}
To help parse the above definition, we compute $t_H(A)$ in a few examples. For this fix vertices $u, v, w \in V(G_n)$. 
\begin{enumerate}
\item[--] If $H=K_2$ is an edge, then 
\begin{align*}
t_{K_2}(\{v\})=&\frac{1}{2}\sum_{u\in V(G_n)} \{ a_{uv}  + a_{v u}  \} = \sum_{u \in V(G_n)} a_{u v}  , 
\end{align*} 
is the degree of vertex $v$ in $G_n$. On the other hand, $t_{K_2}(\{u,v\})=\frac{a_{uv}  + a_{vu} }{2} = a_{uv} $.

\item[--] If $H=K_{1,2}$ is a 2-star (wedge), then 
\begin{align*}
t_{K_{1, 2}}  (\{v\})  =&\sum_{\substack{1 \leq u_1<u_2 \leq |V(G_n)| \\ u_1,u_2 \neq v}} (a_{vu_1}  a_{u_1u_2}  + a_{u_2v}  a_{vu_1}  +a_{u_1u_2}  a_{u_2v} ), \\
t_{K_{1, 2}} (\{u,v\}) =&\sum_{\substack{1 \leq w \leq |V(G_n)| \\ w \neq u,v}} (a_{vu}  a_{uw}  + a_{wv}  a_{vu}  + a_{uw}  a_{wv} )\\ 
t_{K_{1, 2}}(\{u,v,w\})=&a_{vu} a_{uw} +a_{wv} a_{vu} +a_{uw} a_{wv} .
\end{align*}

\item[--] If $H=K_{3}$ is a triangle, then 
\begin{align*}
t_{K_3}(\{v\})=\sum_{\substack{1 \leq u_1<u_2  \leq |V(G_n)| \\ u_1, u_2\neq v}} a_{vu_1} a_{u_1u_2} a_{vu_2} ,\quad
 t_{K_3}(\{u,v\})=\sum_{\substack{1 \leq w \leq |V(G_n)| \\ w \neq u,v}} a_{vu}  a_{uw}  a_{vw}  , 
 \end{align*}
counts the number of triangles in $G_n$ which passes through the vertex $v$, and the edge $(u,v)$ respectively.
Finally, $t_{K_3}(\{u,v,w\})=a_{vu}  a_{uw} a_{vw}$. 	
\end{enumerate} 
\end{example}

Our first result gives a necessary and sufficient condition for the consistency of the HT estimator $\hat N(H, G_n)$ (recall \eqref{eq:estimate_H}). Note that, since the parameter being estimated $N(H, G_n)$ can grow to infinity with $n$, consistency is defined in terms of the ratio of the estimator to the true parameter converging to 1. More formally, given a sequence of graphs $\{G_n\}_{n \geq 1}$ the HT estimator $\hat N(H, G_n)$ is said to be  {\it consistent} for the true motif count $N(H, G_n)$, if $$\frac{\hat N(H, G_n)}{N(H, G_n)} \stackrel{P}{\to} 1,$$ as $n \rightarrow \infty$.

\begin{theorem} \label{thm:consistency}
Suppose $G_n=(V(G_n),E(G_n))$ is a sequence of graphs, with $|V(G_n)| \rightarrow \infty$ as $n \rightarrow \infty$, and $H$ is a fixed connected graph. Then, given a sampling ratio $p_n\in (0,1)$ which satisfies \eqref{eq:boundp},  the HT estimator $\widehat N(H, G_n)$ is consistent for $N(H, G_n)$ if and only if the following holds: For all $\e>0$ 
		\begin{align} \label{e:cond2}
		\lim_{n\to\infty} \frac1{N(H, G_n)} \sum_{\substack{ A  \subset V(G_n) \\ 1 \leq |A| \leq |V(H)|}} t_H(A)\ind\{t_H(A)>\e p_n^{|A|} N(H, G_n)\}=0 . 
		\end{align}
\end{theorem}

\begin{remark} Note that since every term in the sum in \eqref{e:cond2} is non-negative,  \eqref{e:cond2} is equivalent to 
\begin{align} \label{e:cond2_sum}
\lim_{n\to\infty} \frac1{N(H, G_n)} \sum_{\substack{ A  \subset V(G_n) \\ |A| = s}} t_H(A)\ind\{t_H(A)>\e p_n^{s} N(H, G_n)\}=0 , 
\end{align} 
for all $\e > 0$ and all $1 \leq s \leq |V(H)|$. To understand the implications of the condition in \eqref{e:cond2} (or equivalently, \eqref{e:cond2_sum}) note that 
\begin{align}\label{th-iden}
\sum_{\substack{A \subset V(G_n) \\ 1 \leq |A| \leq |V(H)|}}  t_H(A) & = \sum_{K=1}^{|V(H)|}\sum_{\substack{A \subset V(G_n) \\  |A| =K}}  \frac1{|\ah|}\sum_{\s: \bar{\s}\supseteq A } M_H(\s) \nonumber \\ 
& = \sum_{K=1}^{|V(H)|}\sum_{\s \in V(G_n)_{|V(H)|}}  \frac1{|\ah|}M_H(\s)\sum_{\substack{A \subseteq \bar{\s} \\  |A| =K}} 1 \nonumber \\ 
& =\sum_{K=1}^{|V(H)|}  N(H, G_n)\binom{|V(H)|}{K } = (2^{|V(H)|} - 1)  N(H, G_n) .
\end{align}  
Hence, \eqref{e:cond2} demands that the contribution to $N(H, G_n)$ coming from subsets of vertices with `high' local counts is asymptotically negligible. 
\end{remark}

The proof of Theorem \ref{thm:consistency} is given in Section \ref{sec:consistency}. To show  (\ref{e:cond2}) is sufficient for consistency, we define a truncated random variable  $T^+_{\e}(H, G_n)$ (see \eqref{eq:THGn_epsilon}), which is obtained by truncating the HT estimator whenever the local counts functions are large, more precisely, if $t_H(A) > \varepsilon p_n^{|A|} N(H, G_n)$. Then the proof involves two steps: (1) showing that the difference between $T^+_{\e}(H, G_n)$ and $T(H, G_n)$ is asymptotically negligible whenever (\ref{e:cond2}) holds (Lemma \ref{claim}), and (2) a second moment argument to show that $T^+_{\e}(H, G_n)$ concentrates around its expectation. For the necessity, assuming condition (\ref{e:cond2}) does not hold, an application of the well-known Fortuin-Kasteleyn-Ginibre (FKG) correlation inequality \cite[Chapter 2]{inequality} shows that with positive probability no $|V(H)|$-tuple with `high' local count functions is observed. Moreover, conditional on this event, there is a positive probability (bounded away from $0$) that the HT estimator is atypically small. This implies that the (unconditional) probability of the HT estimator being atypically small is also bounded away from zero, which shows the inconsistency of the HT estimator.

In Section \ref{sec:graphs} we will use Theorem \ref{thm:consistency} to derive the precise thresholds for consistency of the HT estimator for many natural classes of graph ensembles. The condition in \eqref{e:cond2_sum} simplifies for specific choices of the motif $H$, as illustrated for the number of edges ($H=K_2$) in the example below.

\begin{example}\label{rem:edge} 
Suppose $H=K_2$ is an edge. Then $N(K_2,G_n) = |E(G_n)|$ is the number of edges in $G_n$ and, recalling the calculations in Example \ref{example:local_A}, the assumption in \eqref{e:cond2_sum} is equivalent to the following two simultaneous conditions: For all $\e > 0$, 
\begin{align}\label{eq:edge_condition}
\lim_{n\to\infty}p_n^2 |E(G_n)| = \infty \quad \text{and} \quad \displaystyle \lim_{n\to\infty}\frac{1}{|E(G_n)|}\sum_{v = 1}^{|V(G_n)|} d_v\ind\{d_v>\e p_n|E(G_n)| \}= 0, 
\end{align}
where $d_v$ is the degree of the vertex $v$ in $G_n$. Note that first condition requires that the  expected number of edges in the sampled graph goes to infinity, and the second condition ensures that the fraction of edges incident on vertices with `high' degree (greater than $\varepsilon p_n |E(G_n)|$) is small. In Example \ref{example:Gn_consistent} we construct a sequence of graphs $\{G_n\}_{n \geq 1}$ for which $p_n^2 |E(G_n)| \rightarrow \infty$, but the HT estimator $\hat N(K_2, G_n)$ is inconsistent, illustrating the necessity of controlling the number of edges incident on the high-degree vertices, as in the second condition of \eqref{eq:edge_condition}. The condition in \eqref{e:cond2_sum}   can be similarly simplified for $H=K_{1,2}$ and $H=K_3$ using the calculations in Example \ref{example:local_A}.  
\end{example}

\begin{figure*}[h]
\centering
\begin{minipage}[l]{1.0\textwidth}
\centering
\vspace{-0.05in}
\includegraphics[width=3.75in,height=2.75in]
    {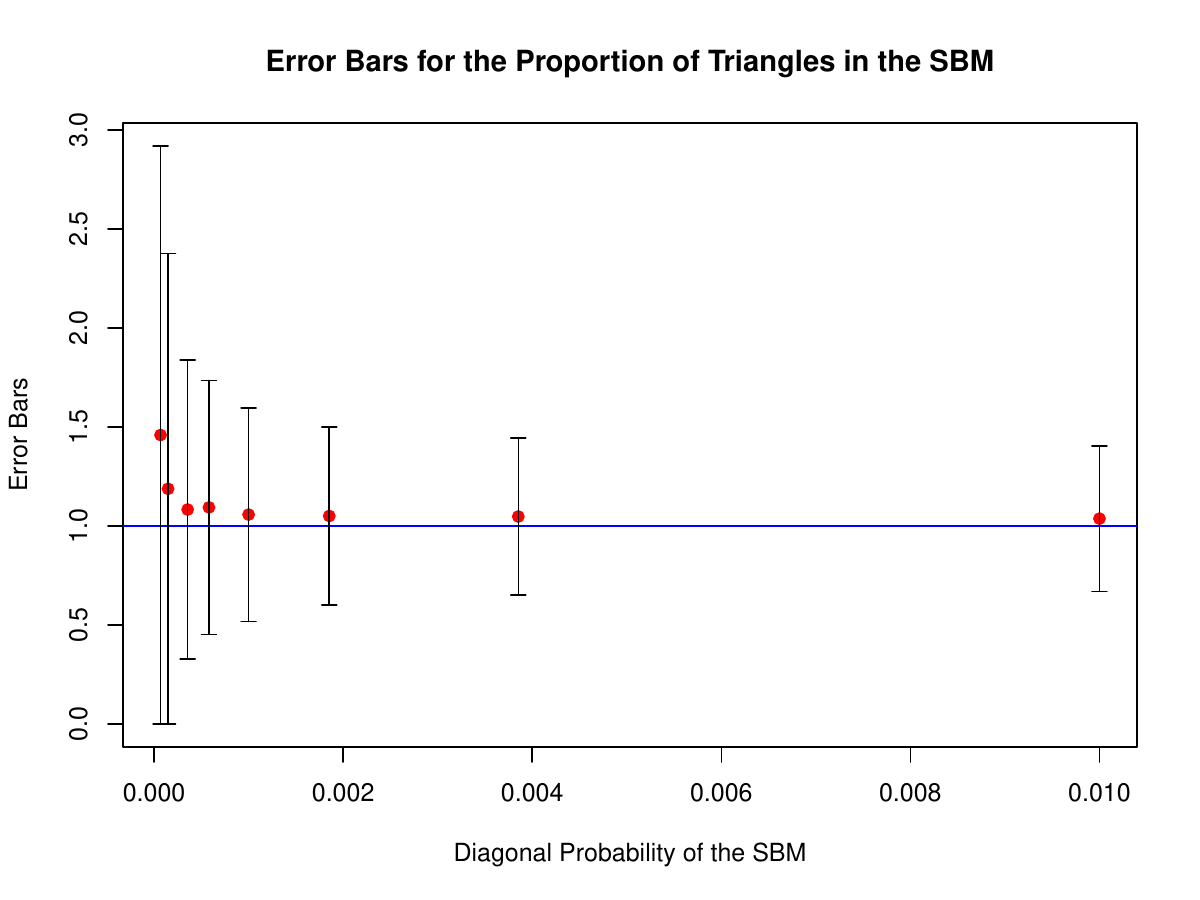}\\
\end{minipage} 
\caption{\small{Error bars for $\hat N(K_3, G_n)/N(K_3, G_n)$ in a 2-block stochastic block model on $n=10000$ vertices and equal block size, with off-diagonal probability $0.5$ and diagonal probability varying between 0 and $0.01$ (shown along the horizontal axis).}}
\label{fig:errorbars_blockmodel}
\end{figure*}

Figure \ref{fig:errorbars_blockmodel} shows the empirical 1-standard deviation error bars for estimating the number of triangles in a 2-block stochastic block model (SBM) with equal block sizes, where edges between vertices in the same block are present independently with probability $a \in (0, 1)$ and edges between vertices in different blocks are present independently with probability $b \in (0, 1)$. Here, fixing $a,b \in (0, 1)$ we consider a realization of $G_n$ from a stochastic block model on $n=10000$ vertices with equal block sizes and diagonal probability $a$ and off-diagonal probability $b=0.5$, and sampling ratio $p_n=0.03$. Figure \ref{fig:errorbars_blockmodel}  then shows the empirical 1-standard deviation error bars of $\hat N(K_3, G_n)/N(K_3, G_n)$ over 1000 repetitions, for a range of 8 values of $a$ between 0 and 0.01 (as shown along the horizontal axis). Note that as $a$ increases, the sizes of the error bars decrease, that is, $\hat N(K_3, G_n)$ becomes a more accurate estimator of $N(K_3, G_n)$. This is because one of the conditions that determine the consistency of $\hat N(K_3, G_n)$ is that the expected number of triangles in the sampled graph diverges, that is, $\Ex[T(K_3, G_n)] = p_n^3 \Ex[N(K_3, G_n)]$ (which is obtained by taking $s=3$ in \eqref{e:cond2_sum}). Now, as $a$ increases, $\Ex[N(K_3, G_n)]$, which is the expected number of triangles in the SBM, increases, hence  $\Ex[T(K_3, G_n)]$ increases, improving the accuracy of $\hat N(K_3, G_n)$ for estimating $N(K_3, G_n)$.

\subsubsection{A Simpler Variance Condition}

In this section we discuss a simpler sufficient condition for the consistency of the HT estimator,  arising from the direct application of Chebyshev's inequality, which will be useful in applications. To this end, note that 
\begin{align}\label{eq:consistent_2moment} 
\lim_{n\to\infty} \frac{1}{ N(H, G_n)^2} \sum_{\substack{ A  \subset V(G_n) \\ 1 \leq |A| \leq |V(H)|}} \frac{t_H(A)^2}{p_n^{|A|}}=0
\end{align} 
is a sufficient condition for \eqref{e:cond2}, since 
$$ t_H(A)\ind\{t_H(A)>\e p_n^{|A|}  N(H, G_n)\} \leq  \frac{t_H(A)^2}{\e p_n^{|A|} N(H, G_n)}.$$ 
The condition in \eqref{eq:consistent_2moment}, which does not require any truncations,  is often easier to verify, as will be seen in the examples discussed below. To derive \eqref{eq:consistent_2moment} without using \eqref{e:cond2}, use  Chebyshev's inequality to note that a straightforward sufficient condition for the consistency of the estimate $\hat N(H, G_n)$ is that $\vr[\hat N(H, G_n)] = o(\hat N(H, G_n)^2)$. This last condition is equivalent to \eqref{eq:consistent_2moment}, as can be seen by invoking Lemma \ref{compare} to get the estimate 
\begin{align*}
\mathrm{Var}(\hat N(H, G_n) ) & = \Theta\left( \sum_{\substack{ A  \subset V(G_n) \\ 1 \leq |A| \leq |V(H)|}}  \frac{t_H(A)^2}{p_n^{|A|}} \right).
\end{align*} 
Even though the variance condition \eqref{eq:consistent_2moment} is natural and often easier to verify, it is not necessary for consistency, as shown in the example below.

\begin{example} (The variance condition is not necessary for consistency) Let $H=K_2$ be the edge, and $G_n$ be the disjoint union of an $a_n$-star $K_{1, a_n}$ and $b_n$ disjoint edges, with $a_n\ll b_n\ll a_n^{3/2}$. Then, 
\begin{align}\label{eq:ve}
|V(G_n)|=a_n+1+2b_n=(1+o(1))  2b_n,\quad N(H, G_n)=|E(G_n)| = a_n+b_n=(1+o(1))  b_n.
\end{align} 
In this case, the HT estimator is consistent whenever the sampling probability $p_n$ satisfies $\frac{1}{\sqrt{b_n}}\ll p_n\ll a_n^2/b_n.$ To see this, note that $p_n^2|E(G_n)|=(1+o(1))  p_n^2 b_n\gg 1$, that is, the first condition in \eqref{eq:edge_condition} holds. Also, fixing $\varepsilon>0$ and noting that $p_n|E(G_n)|=(1+o(1))  p_n b_n\gg 1$  implies, for all $n$ large only the central vertex of the $a_n$-star satisfies the $d_v>\e p_n|E(G_n)|$ cutoff. Hence, 
\[ \sum_{v=1}^{|V(G_n)|}d_v{\bf 1}\{d_v>\varepsilon p_n|E(G_n)\}=a_n=o(b_n),\]
verifying second condition in \eqref{eq:edge_condition}.  However, since 
	\[\frac{1}{p_n|E(G_n)|^2}\sum_{v=1}^{|V(G_n)|} d_v^2 = \frac{1}{p_n b_n^2}\Big(a_n^2+a_n+b_n\Big)=(1+o(1))  \frac{a_n^2}{p_n b_n^2}\rightarrow\infty ,\]
the variance condition \eqref{eq:consistent_2moment} does not hold. Thus for this example one needs the full strength of Theorem \ref{thm:consistency} to show that the HT estimator is consistent.
\end{example}

\subsection{Asymptotic Normality of the HT Estimator}
\label{sec:normal_ZHGn} 

In this section, we determine the precise conditions under which the HT estimator is asymptotically normal. For this, recall the definition of $Z(H, G_n)$ from \eqref{eq:ZHGn_I},   
\begin{align}\label{eq:ZHGn} 
Z(H, G_n):= \frac{\hat N(H, G_n)- N(H, G_n)}{\sqrt{\vr[\hat N(H, G_n)]}} = \frac{T(H, G_n)-p_n^{|V(H)|}N(H, G_n)}{\sigma(H,G_n)},
\end{align}
where $\sigma(H,G_n)^2:=\vr[T(H,G_n)]$. To begin with, one might wonder whether the conditions which ensure the consistency of $\hat N(H, G_n)$ is enough to imply the asymptotic normality of $Z(H, G_n)$. However, it is easy to see that this is not the case. In fact, there are examples where  $\hat N(H, G_n)$ is consistent, but $Z(H, G_n)$ has a non-Gaussian limiting distribution (see Example \ref{example:consistent} in Appendix \ref{sec:examples}). Hence, to establish the asymptotic normality of $Z(H, G_n)$ additional conditions are needed. To state our result we need the following definition: 

\begin{definition}\label{eq:connected}
Fix $r \geq 1$. Given a collection of $r$ tuples $\{\s_1,\s_2, \ldots, \s_r \}$ from $V(G_n)_{|V(H)|}$, let $\mathcal{G}(\s_1,\ldots,\s_r)$ be the simple graph with vertex set $\{\s_1,\ldots,\s_r\}$, with an edge between $\s_i$ and $\s_j$ whenever $\bar \s_i\cap \bar \s_j\ne \emptyset$ (see Figure \ref{fig:Gs} for an illustration). We will say the collection $\{\s_1,\ldots,\s_r\}$ is {\it connected}, if the graph $\mathcal{G}(\s_1,\ldots,\s_r)$ is connected. The set of all $r$ tuples $\{\s_1,\ldots,\s_r\}$ in $V(G_n)_{|V(H)|}$ such that the collection $\{\s_1,\ldots,\s_r\}$ is connected will be denoted by $\mathcal{K}_{n,r}$. 
\end{definition}

\begin{figure*}[h]
\centering
\begin{minipage}[c]{1.0\textwidth}
\centering
\includegraphics[width=2.45in]
    {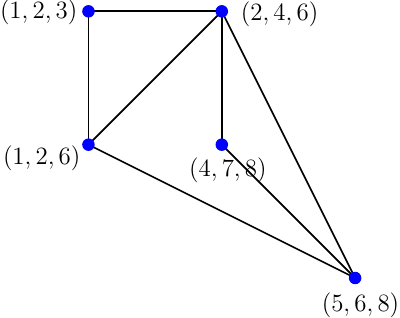}\\
\end{minipage} 
\caption{\small{The graph $\mathcal{G}(\s_1, \s_2, \s_3, \s_4, \s_5)$ as in Definition \ref{eq:connected} with $\s_1=(1, 2, 3)$, $\s_2=(1, 2, 6)$, $\s_3=(4, 7, 8)$, $\s_4=(2, 4, 6)$, and $\s_5=(5, 6, 8)$.}}
\label{fig:Gs}
\end{figure*}

Now, denote by $W_n$ the random variable 
\begin{align}\label{W_n:defn}
W_n :=\sum_{\{\s_1,\s_2,\s_3,\s_4\}\in \mathcal{K}_{n, 4}} |Y_{\s_1}Y_{\s_2}Y_{\s_3}Y_{\s_4}|,
\end{align} 
where $Y_{\s}: =  \frac{1}{|Aut(H)|} \prod_{(i,j)\in E(H)} a_{s_i s_j} (X_{\s}-p_n^{|V(H)|})$. In the following theorem we give a quantitative error bound (in terms of the Wasserstein distance) between $Z(H, G_n)$ and the standard normal distribution $N(0, 1)$, in terms of the expected value of the random variable $W_n$. To this end, recall that the Wasserstein distance between random variables $X \sim \mu$ and $Y \sim \nu$ on $\R$ is defined as 
$$\mathrm{Wass}(X, Y) = \sup \left\{\left|\int f \mathrm d\mu - \int f \mathrm d \nu \right| : f \text{ is } 1-\text{Lipschitz}\right\},$$
where a function $f: \mathbb R \rightarrow \mathbb R$ is 1-Lipschitz if $|f(x) - f(y)| \leq |x-y|$, for all $x, y \in \mathbb R$.

\begin{theorem}\label{thm:wass}Fix a connected graph $H$, a network $G_n=(V(G_n),E(G_n))$, and a sampling ratio $p_n$ which satisfies \eqref{eq:boundp}. Then 
	\begin{align}\label{e:wass}
	\operatorname{Wass}(Z(H, G_n), N(0, 1)) \lesssim \frac{|V(H)|}{(1-\kappa)^3}\cdot\sqrt{\frac{\Ex[W_n]}{\sigma(H,G_n)^4}},
	\end{align}
where $Z(H, G_n)$ and $W_n$ are as defined in \eqref{eq:ZHGn} and \eqref{W_n:defn}, respectively. Moreover, if $p_n \in \left(0, \tfrac1{20}\right]$, then $\frac{\Ex[W_n]}{\sigma(H,G_n)^4} \le \Ex[Z(H,G_n)^4]-3$ and, as a consequence,  
\begin{align}\label{e:wass_II}
\operatorname{Wass}(Z(H, G_n), N(0, 1)) \lesssim |V(H)|\cdot\sqrt{\Ex[Z(H,G_n)^4]-3},
\end{align}
\end{theorem}

The proof of this result is given in Appendix \ref{sec:clt_pf}. In addition to giving an  explicit rate of convergence between $Z(H, G_n)$ and $N(0, 1)$, Theorem \ref{thm:wass} shows that for $p_n$ small enough, the asymptotic normality of the (standardized) HT estimator exhibits a curious fourth-moment phenomenon, that is, $Z(H, G_n) \stackrel{D} \rightarrow  N(0, 1)$ whenever $\Ex[Z(H,G_n)^4] \rightarrow 3$ (the fourth moment of the standard normal distribution). The proof uses Stein's method for normal approximation \cite{barbour2005introduction,clt_normal,Ste72} and is a consequence of more general result about the asymptotic normality and the fourth-moment phenomenon of certain random multilinear forms in Bernoulli variables, which might be of independent interest (Theorem \ref{thm:wass1}).

\begin{remark} \label{remark:4moment} 
The fourth moment phenomenon was first discovered by Nualart and Peccati \cite{nualart2005central}, who showed that the convergence of the first, second, and fourth moments to $0, 1$, and $3$, respectively, guarantees asymptotic normality for a sequence of multiple stochastic Wiener-It\^o integrals of fixed order. Later, Nourdin and Peccati \cite{NoPe09} provided an error bound for the fourth moment theorem of \cite{nualart2005central}. The fourth moment phenomenon has since then emerged as a unifying principle governing the central limit theorems for various non-linear functionals of random fields \cite{BDM,peccati_book,nourdin2010invariance}.  We refer the reader to the book \cite{book_np} for an introduction to the topic and the website \url{https://sites.google.com/site/malliavinstein/home} for a list of the recent results. The result in Theorem \ref{thm:wass} is an example of the fourth-moment phenomenon in the context of motif estimation. In fact, the result in Section \ref{sec:pf_normal_ZHGn} on the asymptotic normality of general random multilinear forms suggests that the fourth-moment phenomenon is more universal, and we expect it to emerge in various other combinatorial estimation problems, where counting statistics similar to $T(H, G_n)$ arise naturally. 
\end{remark}

\begin{remark} Note that the result in \eqref{e:wass_II} requires an upper bound on the sampling ratio $p_n \leq \frac{1}{20}$. This condition ensures that the leading order of the central moments of $T(H, G_n)$ is the same as the leading order of its raw moments (as shown in Lemma \ref{abs-bound}), a fact which is used to estimate the error terms arising from the Stein's method calculations. Interestingly, it is, in fact, necessary to assume an upper bound on $p_n$ for the limiting normality and the fourth-moment phenomenon of the HT estimator to hold (see Example \ref{example:ab} in Appendix \ref{sec:examples}). This example constructs a sequence of graphs $\{G_n\}_{n \geq 1}$ for which if $p_n$ is chosen large enough, then $\Ex[Z(K_2, G_n)^4] \rightarrow 3$, but $Z(K_2, G_n)$ does not converge to $N(0, 1)$.  However, in applications,  where it is natural to chose $p_n \ll 1$ to have any significant reduction in the size of the sampled graph, the fourth moment phenomenon always holds.
\end{remark}

We now discuss how the results above can be used to construct asymptotically valid confidence intervals  for the parameter $N(H, G_n)$. To this end, we need to consistently estimate $\sigma(H, G_n)^2$, the variance of $T(H, G_n)$. The following result shows that it is possible to consistently estimate $\sigma(H, G_n)^2$ whenever the error term in \eqref{e:wass} goes to zero, which combined with the asymptotic normality of $Z(H, G_n)$ gives a confidence for $N(H, G_n)$ with asymptotic coverage probability $1-\alpha$.

\begin{proposition}\label{ppn:variance_estimation}  
Fix a connected graph $H$, a network $G_n=(V(G_n),E(G_n))$, and a sampling ratio $p_n$ which satisfies \eqref{eq:boundp}. Suppose $\Ex[W_n]=o(\sigma(H, G_n)^4)$, where $W_n$ is as defined in \eqref{W_n:defn}. Then the following hold,  as $n \rightarrow \infty$: 

\begin{itemize} 

\item[(a)] The HT estimator $\hat N(H,G_n)$ is consistent for $N(H,G_n)$.

\item[(b)] Let 
	{$$\hat \sigma(H,G_n)^2:=\frac1{|\ah|^2}\sum_{K=1}^{|V(H)|}  \sum_{\substack{\s_1, \s_2 \in V(G_n)_{|V(H)|} \\ 
			| \bar \s_1 \bigcap \bar \s_2 | = K }} (1-p_n^{K})M_H(\s_1)M_H(\s_2)X_{\s_1}X_{\s_2}.$$}
Then $\hat \sigma(H,G_n)^2$ is a consistent estimate of $\sigma(H,G_n)^2$, that is, $\frac{\hat \sigma(H,G_n)^2}{\sigma(H,G_n)^2} \stackrel{P} \rightarrow 1$. 
	
\item[(c)] Let $\hat \sigma(H,G_n)_+:=\sqrt{\max(0, \hat \sigma(H,G_n)^2)}$. Then, as $n \rightarrow \infty$,  
$$\mathbb P\left(N(H, G_n) \in \left[\hat N(H,G_n) -  z_{\frac{\alpha}{2}} \frac{\hat \sigma(H,G_n)_+}{p_n^{|V(H)|}}, \hat N(H,G_n)   +  z_{\frac{\alpha}{2}} \frac{\hat \sigma(H,G_n)_+}{p_n^{|V(H)|}}  \right] \right) \rightarrow 1-\alpha, $$ 
where $z_{\frac{\alpha}{2}}$ is the $(1-\frac{\alpha}{2})$-th quantile of the standard normal distribution $N(0, 1)$. 	 
\end{itemize}	
\end{proposition} 

The proof of this result is given in Appendix \ref{sec:variance_estimation_pf}. The proof of (a) entails showing that $\sigma(H,G_n)^2=o((\Ex[T(H,G_n)])^2)$. This is a consequence of the assumption $\Ex[W_n]=o(\sigma(H, G_n)^4)$ and the more general bound $\sigma(H,G_n)^6 \lesssim_H \Ex[W_n] (\Ex[T(H,G_n)])^2$,  which can be proved by expanding out the terms and an application of the H\"older's inequality. For (b), note that  $\hat \sigma(H,G_n)^2$ is an unbiased estimate of $\sigma(H,G_n)^2$, hence, to prove the consistency of  $\hat \sigma(H,G_n)^2$ it suffices to show that $\vr[\hat \sigma(H,G_n)^2]=o(\sigma(H, G_n)^4)$, under the given assumptions. Finally, (c) is an immediate consequence of (b) and the asymptotic normality of $Z(H, G_n)$ proved in Theorem \ref{thm:wass}.

Given the result in Theorem \ref{thm:wass}, it is now natural to wonder whether the convergence of the fourth moment $\Ex[Z(H, G_n)^4]\rightarrow 3$ is necessary for the asymptotic normality of $Z(H, G_n)$. This however turns out to be not the case. In fact, Example \ref{example:4_normal} gives a sequence of graphs $\{G_n\}_{n \geq 1}$ for which $Z(K_2, G_n)$ is asymptotic normal, but $\Ex[Z(K_2, G_n)^4] \nrightarrow 3$,  showing that the (untruncated) fourth-moment condition is not necessary for the asymptotic normality of the HT estimator. As we will see, in this example the graph $G_n$ has a few `high' degree vertices which forces  $\Ex[Z(H, G_n)^4]$ to diverge. However, the existence of a `small' number of high degree vertices does not effect the distribution of the rescaled statistic. This suggests that, as in the case of consistency in Theorem \ref{thm:consistency}, to obtain the precise condition for the asymptotic normality of $Z(H, G_n)$ we need to appropriately truncate the graph $G_n$, by removing a small number of hubs with `high' local count functions, and  consider the moments of the truncated statistic. Towards this end, fix $M > 0$ and define the event 
\begin{align}\label{eq:C_truncation}
\mathcal{C}_M(A)=\{t_H(A)^2> M p_n^{2|A|-2|V(H)|}\vr[T(H, G_n)]\},
\end{align} 
and 
$\mathcal{C}_{M}(\s)^c=\bigcap_{ A \subseteq \s: A  \ne \emptyset} \mathcal{C}_M(A)^c$. (For any set $A$, $A^c$ denotes the complement of $A$.) Then consider the truncated statistic, 
	\begin{align} \label{def:trunc}
	T_M^{\circ}(H, G_n):=\frac1{|\ah|}\sum_{\s\in V(G_n)_{|V(H)|}}M_H(\s) X_{\s}\ind\{\mathcal{C}_M(\s)^c\}, 
	\end{align}
and define 
\begin{align}\label{eq:truncation_Z}
Z_M^{\circ}(H, G_n):=\frac{T_M^{\circ}(H, G_n)-\Ex[T_M^{\circ}(H, G_n)]}{\sigma(H,G_n)}. 
\end{align} 
The following theorem gives a necessary and sufficient condition for asymptotic normality for $Z(H, G_n)$ in the terms of the second and fourth moments of the truncated statistic \eqref{def:trunc}. 

\begin{theorem}\label{thm:normality} Suppose $G_n=(V(G_n),E(G_n))$ is a sequence of graphs, with $|V(G_n)| \rightarrow \infty$, and $H$ is a fixed connected graph. Then, given a sampling ratio $p_n \in (0, \frac{1}{20}]$, the rescaled statistic  
$Z(H, G_n) \stackrel{D} \rightarrow N(0, 1)$
if and only if 
\begin{align}\label{e:4th} &
\limsup_{M \to\infty}\limsup_{n\to\infty}|\Ex[Z_M^{\circ}(H, G_n)^2]-1|=0, \text{ and }  \limsup_{M \to\infty}\limsup_{n\to\infty}  |\Ex[Z_M^{\circ}(H, G_n)^4]-3|=0 , 
\end{align}  
holds simultaneously.
\end{theorem}

This result shows that the asymptotic normality of $Z(H, G_n)$ is characterized by a truncated fourth-moment phenomenon, more precisely, the convergence of the second and fourth-moments of $Z_M^{\circ}(H, G_n)$ to 1 and 3, respectively. Note that the second moment condition in \eqref{e:4th} ensures that $\vr[T_M^{\circ}(H, G_{n})]=(1+o(1)) \vr[T(H, G_{n})]$. Hence, the fourth-moment condition in \eqref{e:4th} and the Theorem \ref{thm:wass} implies that 
$$\frac{T_M^{\circ}(H, G_n)-\Ex[T_M^{\circ}(H, G_n)]}{\sqrt{\vr[T_M^{\circ}(H, G_n)]}} \stackrel{D} \rightarrow N(0, 1).$$  Therefore, to establish the sufficiency of the conditions in \eqref{e:4th}, it suffices to show that the difference between $T(H, G_{n})$ and $T_M^\circ(H, G_{n})$ scaled by $\vr[T(H, G_n)]$ is small, which follows from the properties of the truncation event \eqref{eq:C_truncation} (see Lemma \ref{properties}). To prove that \eqref{e:4th} is also necessary for the asymptotic normality of $Z(H, G_n)$, we show all moments of $Z_M^{\circ}(H, G_n)$ are bounded (Lemma \ref{moment_r}), which combined with the fact that  $T(H, G_{n}) - T_M^\circ(H, G_{n}) \stackrel{P} \rightarrow 0$ and uniform integrability, implies the desired result (see Appendix \ref{sec:pf_normality} for details).

\subsection{Thresholds for Consistency and Normality}
\label{sec:graphs}

In this section, we apply the results above to derive the thresholds for consistency and asymptotic normality of the HT estimator in various natural graph ensembles. Throughout this section we will assume that $p_n \in (0, \frac{1}{20}]$.

\subsubsection{Bounded Degree Graphs} 

We begin with graphs which have bounded maximum degree. Towards this, denote by $d_v$ the degree of the vertex $v$ in $G_n=(V(G_n), E(G_n))$, and let $\Delta(G_n)=\max_{v \in V(G_n)} d_v$ be the maximum degree of the graph $G_n$. 

\begin{proposition}[Bounded degree graphs] \label{bdd-deg} Suppose $\{G_n\}_{n \geq 1}$ is a sequence of graphs with bounded maximum degree, that is, $\Delta:=\sup_{n \geq 1}\Delta(G_n)=O(1)$. Then for any connected graph $H$ the following hold: 
\begin{itemize}
\item[(a)] If $p_n^{|V(H)|} N(H, G_n) \gg 1$, then the HT estimator $\hat N(H, G_n)$ is consistent for $N(H, G_n)$, and the rescaled statistic $Z(H,G_n) \stackrel{D} \rightarrow N(0, 1)$. Moreover, 
$$\mathrm{Wass}(Z(H, G_n), N(0, 1)) \lesssim_{\Delta,H} \sqrt{\frac{1}{p_n^{|V(H)|}N(H,G_n)}}.$$ 
 
\item[(b)] If $p_n^{|V(H)|} N(H, G_n) =O(1)$, then the HT estimator $\hat N(H, G_n)$ is not consistent for $N(H, G_n)$ and the rescaled statistic $Z(H,G_n)$ is not asymptotically normal.

\end{itemize} 
\end{proposition}

Recall that $\Ex[T(H, G_n)] = p_n^{|V(H)|} N(H, G_n)$. Therefore, in other words, the result above shows that the HT estimator is consistent and asymptotic normal in bounded degree graphs whenever the expected number of copies of $H$ in the sampled graph diverges, whereas it is inconsistent whenever the expected number copies remains bounded. The proof of Proposition \ref{bdd-deg} is given in Appendix \ref{sec:degree_pf}. For (a), using Proposition \ref{ppn:variance_estimation} , it is suffices to bound $\frac{1}{\sigma(H, G_n)^4} \Ex[W_n]$. This involves, recalling the definition of $W_n$ from \eqref{W_n:defn}, bounding the number of copies of various subgraphs in $G_n$ obtained by the union of 4 isomorphic copies $H$, which in this case can be estimated using the maximum degree bound on $G_n$. For (b), we show that whenever $\Ex[T(H, G_n)] = p_n^{|V(H)|} N(H, G_n) = O(1)$, there is a positive chance that $T(H, G_n)$ is zero, which immediately rule out consistency and normality.

\subsubsection{Erd\H{o}s-R\'enyi Random Graphs} 

We now derive the thresholds for consistency and asymptotic normality in various random graph models. We begin with the Erd\H{o}s-R\'enyi model $G_n \sim \mathcal G(n, q_n)$, which is a random graph on $n$ vertices where each edge is present independently with probability $q_n \in (0, 1)$. Here the location of the phase transition
is related to the notion of balancedness of a graph. 

\begin{definition}\cite[Chapter 3]{janson2011book} \label{defn:m} For a fixed connected graph $H$, define 
\begin{align*}
m(H)=\max_{H_1\subseteq H}\frac{|E(H_1)|}{|V(H_1)|},
\end{align*}
where the maximum is over all non-empty subgraphs $H_1$ of $H$. The graph $H$ is said to be {\it balanced}, if $m(H)=\frac{|E(H)|}{|V(H)|}$, and {\it unbalanced} otherwise. 
\end{definition}

\begin{theorem}[Erd\H{o}s-R\'enyi graphs] \label{thm:er-consistency}
 Let $G_n \sim \mathcal G(n, q_n)$ be an Erd\H{o}s-R\'enyi random graph with edge probability $q_n\in (0,1)$. Then for any connected graph H the following hold: 
\begin{enumerate}
\item[(a)]
If $np_nq_n^{m(H)}\gg 1$, then the HT estimator $\hat N(H, G_n)$ is consistent for $N(H, G_n)$, and the rescaled statistic $Z(H,G_n) \stackrel{D} \rightarrow N(0, 1)$. Moreover,  
\begin{align*}
\operatorname{Wass}(Z(H,G_n),N(0,1))=O_P\left((np_nq_n^{m(H)})^{-\frac12}\right).
\end{align*}

\item[(b)]
If $np_nq_n^{m(H)}=O(1)$, then $\hat N(H,G_n)$ is not consistent for $N(H,G_n)$, and $Z(H,G_n)$ is not asymptotically normal. 
\end{enumerate}
\end{theorem}

The proof of this result is given in Appendix \ref{sec:er_pf}. Here, to estimate $W_n$, we first take  expectation over the randomness of the graph, and then use an inductive counting argument (Lemma \ref{lm:graph_B}) combined with a second moment calculation, to obtain the desired bound.

\begin{figure*}[h]
\centering
\begin{minipage}[l]{1.0\textwidth}
\centering
\vspace{-0.15in}
\includegraphics[width=3.85in]
    {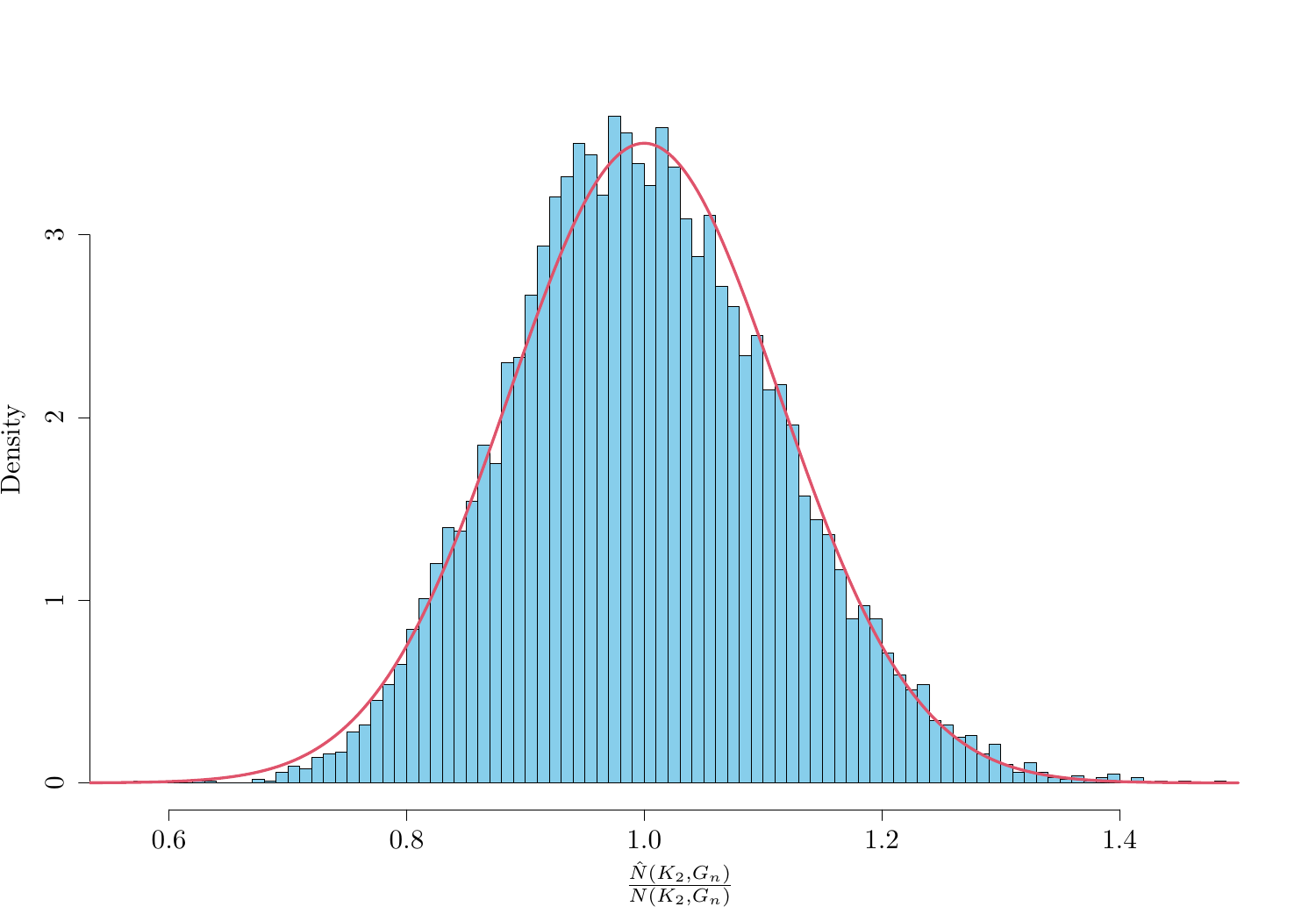}\\
\end{minipage} 
\caption{\small{Histogram of $\hat N(K_2, G_n)/N(K_2, G_n)$ in the Erd\H os-R\'enyi random graph $G_n \sim \mathcal G(10000, 0.5)$ with sampling ratio $p_n=0.03$ over 10000 replications, and the limiting normal density (plotted in red). }}
\label{fig:edgehistogram}
\end{figure*}

\begin{remark}
To interpret the threshold in Theorem \ref{thm:er-consistency}, recall that $n q_n^{m(H)}$ is the threshold for the occurrence of $H$ in the random graph $\mathcal{G}(n, q_n)$ \cite[Theorem 3.4]{janson2011book}. More precisely, whenever $n q_n^{m(H)} = O(1)$ the number of copies of $H$ in $\mathcal{G}(n, q_n)$  is $O_P(1)$, whereas if $n q_n^{m(H)} \gg 1$, the number of copies of $H$ in $G_n$ diverges. In this case, conditional on the set of sampled vertices $S$,  the observed graph behaves like the Erd\H{o}s-R\'enyi model $\mathcal G(|S|, q_n)$. As a result, since $S \sim \mathrm{Bin}(n, p_n)$, the observed graph (unconditionally) looks roughly like the model $\mathcal{G}(n p_n, q_n)$. Therefore, Theorem \ref{thm:er-consistency} essentially shows that the HT estimator is consistent and asymptotically normal whenever the number of copies of $H$ in sampled graph diverges (which happens if $n p_n q_n^{m(H)} \rightarrow \infty$), whereas it is inconsistent whenever the number of copies of $H$ is bounded in probability.  The histogram in Figure \ref{fig:edgehistogram} illustrates the asymptotic normality of the HT estimator for the number of edges ($H=K_2$). Here, we fix a realization of the Erd\H os-R\'enyi random graph $G_n \sim \mathcal G(n, q_n)$, with $n=10000$ and $q_n=\frac{1}{2}$, choose the sampling ratio $p_n=0.03$, and plot the histogram of $\hat N(K_2, G_n)/N(K_2, G_n)$ over 10000 replications. Note that, as expected, the histogram is centered around 1, with the red curve showing the limiting normal density. 
\end{remark}

Note that Theorem \ref{thm:er-consistency} above gives a CLT for $\hat N(H, G_n)$ centered around $N(H, G_n)$, when $np_nq_n^{m(H)}\gg 1$. However, since $G_n$ is a random graph $N(H, G_n)$ is itself random, and it is natural to wonder whether one can obtain 
a CLT for $\hat N(H, G_n)$ centered around $\Ex[N(H, G_n)]$, where the expectation is taken with respect to the randomness of $G_n$. This question is not just specific to the Erd\H os-R\'enyi model, it arises whenever $G_n$ is generated from any underlying stochastic model. To address this issue suppose $\{G_n\}_{n \geq 1}$ is a sequence of random graphs (from some generative model) and define
\begin{align}\label{eq:def1}
	\mathcal A(H, G_n):=\frac{\hat{N}(H,G_n)-\Ex [N(H,G_n)]}{\sqrt{\mathrm{Var}[\hat{N}(H,G_n)]}} , 
\end{align}
where the expectation and the variance above are taken over both the randomness of the sampling scheme and the graph $G_n$. Note that 
\begin{align}\label{eq:ZEH}
\mathcal A(H, G_n) =  \sqrt{\frac{\mathrm{Var}_{G_n}[\hat N(H,G_n)]}{\mathrm{Var}[\hat{N}(H,G_n)]}} \cdot Z(H, G_n) + \sqrt{\frac{\mathrm{Var}[N(H,G_n)]}{\mathrm{Var}[\hat{N}(H,G_n)]}} \cdot \mathcal E(H, G_n), 
\end{align} 
where
\begin{align}\label{eq:ZE}
Z(H,G_n):=\frac{\hat{N}(H,G_n)- N(H,G_n)}{\sqrt{\mathrm{Var}_{G_n}[\hat{N}(H,G_n)]}}  \quad \text{and} \quad \mathcal E(H, G_n):=\frac{{N}(H,G_n)-\Ex [N(H,G_n)]}{\sqrt{\mathrm{Var}[{N}(H,G_n)]}} , 
\end{align}
with $\Ex_{G_n}$ and $\mathrm{Var}_{G_n}$ denoting the conditional expectation and conditional variance taken conditionally on the random graph $G_n$. Recall that Theorem \ref{thm:wass} deals with the CLT of $Z(H,G_n)$ conditional on the graph $G_n$ (often known as a {\it quenched} CLT in the language of statistical physics). Given this result, to obtain a CLT for $\mathcal A(H,G_n)$  (that is, an {\it annealed} CLT in statistical physics terminology), we would need to show a CLT for $\mathcal E(H, G_n)$ and establish that the conditional variance $\vr_{G_n}[\hat{N}(H,G_n)]$ is consistent for its expectation (see Lemma \ref{lem:aez} for the formal statement). In particular, for the  Erd\H{o}s-R\'enyi (ER) model $G(n, q_n)$  both these results can be easily established and we have the following result. 

\begin{corollary}[Erd\H{o}s-R\'enyi graphs (annealed version)] \label{cor:ann-er-consistency}
	Let $G_n \sim \mathcal G(n, q_n)$ be an Erd\H{o}s-R\'enyi random graph with edge probability $q_n\in (0,1)$. Then for any connected graph H the following hold: 
	\begin{enumerate}
		\item[(a)]
		If $np_nq_n^{m(H)}\gg 1$, then the HT estimator $\hat N(H, G_n)$ is consistent for $\Ex[N(H, G_n)]$ and $\mathcal A(H, G_n) \stackrel{D} \rightarrow N(0, 1)$. 
		\item[(b)]
		If $np_nq_n^{m(H)}=O(1)$, then $\hat N(H,G_n)$ is not consistent for $\Ex[N(H,G_n)]$, and $\mathcal A(H, G_n)$ is not asymptotically normal. 
	\end{enumerate}
\end{corollary}

The proof of Corollary \ref{cor:ann-er-consistency} is  given in Appendix \ref{sec:annealed}. This is a consequence of a more general result (see Lemma \ref{lem:aez}) about the CLT of $\mathcal A(H, G_n)$ (when $G_n$ is generated according to some stochastic model). In particular, in Lemma \ref{lem:aez} we show that $\mathcal A(H, G_n) \stackrel{D} \rightarrow N(0, 1)$ whenever the following conditions hold: (a) conditional on the graph sequence $\{G_n\}_{n \geq 1}$, $Z(H,G_n) \stackrel{D} \rightarrow N(0,1)$, (b) $\mathcal E(H, G_n) \stackrel{D} \rightarrow N(0,1)$, and (c) $\vr_{G_n}[\hat{N}(H,G_n)]$ is consistent for its expected value $\Ex[\vr_{G_n}[\hat{N}(H,G_n)]]$. These conditions can be easily verified for the Erd\H{o}s-R\'enyi model $\mathcal G(n, q_n)$ whenever $np_nq_n^{m(H)}\gg 1$, which establishes the result in Corollary \ref{cor:ann-er-consistency} (1). 
\begin{remark} The normality condition (assumption (b)) on $\mathcal E(H, G_n)$ in Lemma \ref{lem:aez} can be removed if instead of assumption (c) the following stronger condition holds: 
\begin{align}\label{eq:varNH}
\frac{\mathrm{Var}[\hat{N}(H,G_n)]}{\mathbb E [\mathrm{Var}[\hat{N}(H,G_n)|G_n]]}\stackrel{P}{\to} 1. 
\end{align}  
This is because \eqref{eq:varNH} implies $\mathrm{Var}[N(H,G_n)] \ll \mathrm{Var}[\hat N(H,G_n)]$, hence, recalling \eqref{eq:ZEH}, the CLT of $\mathcal A(H, G_n)$ follows from the conditional CLT of $Z(H,G_n)$, since $\mathcal E(H, G_n)$ is bounded in probability. In the Erd\H os-R\'enyi model, there is a regime of the parameters $p_n, q_n$ where \eqref{eq:ZEH} holds. There is also a regime where $\mathrm{Var}[N(H,G_n)]$ and $\mathrm{Var}[\hat N(H,G_n)]$ are of the same order (that is, \eqref{eq:varNH} does not hold), where one needs to invoke Lemma \ref{lem:aez} to establish the CLT of $\mathcal A(H, G_n)$. (Recall that unlike \eqref{eq:varNH}, assumption (c) in Lemma \ref{lem:aez} holds in the full range of parameters in Erd\H os-R\'enyi model.) Nevertheless, condition \eqref{eq:varNH} broadens the scope of our results and can be useful in other random graph models. 
\end{remark}

\subsubsection{Random Regular Graphs}

As a corollary to Theorem \ref{thm:er-consistency} we can also derive the threshold for random regular graphs. 
To this end, denote by $\mathscr{G}_{n, d}$ the collection of all simple $d$-regular graphs on $n$ vertices, where $1 \leq d \leq n - 1$ is such that $n d$ is even. 

\begin{corollary}[Random regular graphs]\label{regulargraph} Suppose $G_n$ is a uniform random sample from $\mathscr{G}_{n, d}$ and $H = (V(H), E(H))$ is a connected graph with maximum degree $\Delta(H)$.  
\begin{enumerate}

\item[$(a)$] If $d \gg 1$, then setting $q_n=d/n$ the following hold: 
		
\begin{itemize}
\item If $np_nq_n^{m(H)}\gg 1$, then $\hat N(H,G_n)$ is consistent for $N(H,G_n)$, and $Z(H, G_n)$ converges in distribution to $N(0,1)$.

\item If $np_nq_n^{m(H)}=O(1)$, then $\hat N(H,G_n)$ is not consistent for $N(H,G_n)$, and $Z(H,G_n)$ is not asymptotically normal.

\end{itemize}

\item[$(b)$] If $d = O(1)$, then assuming $\Delta(H) \leq d$, the following hold:

\begin{itemize}
\item If $|E(H)| = |V(H)|-1$, then $\hat N(H,G_n)$ is consistent for $N(H,G_n)$ and $Z(H, G_n)$ converges in distribution to $N(0,1)$ if and only if $n p_n^{|V(H)|} \gg 1$.

\item If $|E(H)| \geq |V(H)|$, then $\hat N(H,G_n)$ is not consistent for $N(H,G_n)$, and $Z(H,G_n)$ is not asymptotically normal, irrespective of the value of $p_n$. 

\end{itemize}

\end{enumerate}
\end{corollary}

It is well-known that the typical behavior of the number of small subgraphs in a random $d$-regular graph asymptotically equals to that in a Erd\H{o}s-R\'enyi graph $\mathcal G(n, q_n)$, with $q_n =d/n$, whenever $d \gg 1$ \cite{KIM20071961,krivelevich2001random}. As a result, the threshold for consistency and asymptotic normality for random $d$-regular graphs obtained in Corollary \ref{regulargraph} above, match with the threshold for Erd\H{o}s-R\'enyi graphs obtained in Theorem \ref{thm:er-consistency} with $q_n=d/n$, whenever $d \gg 1$. However, this analogy with the Erd\H{o}s-R\'enyi model is no longer valid when $d=O(1)$. In this case, to compute the threshold we invoke Proposition \ref{bdd-deg} instead, which deals with the case of general bounded degree graphs. Note that here it suffices to assume $\Delta(H) \leq d$, since $N(H, G_n)=0$ whenever $\Delta(H) > d$. Therefore, assuming $\Delta(H) \leq d$, there are two cases: (1) $|E(H)| = |V(H)|-1$ (that is, $H$ is a tree) and (2) $|E(H)| \geq |V(H)|$ (that is, $H$ has a cycle). In the second case, it can be easily shown that $N(H, G_n)=O_P(1)$, hence, by Proposition \ref{bdd-deg} (b) consistency and asymptotic normality does not hold. On the other hand, in the first case, by a inductive counting argument, it can shown that $N(H, G_n)= \Theta_P(n)$. Hence, by Proposition \ref{bdd-deg} (a), the threshold for consistency and asymptotic normality is $n p_n^{|V(H)|}\gg1$. The details of the proof are given in Appendix \ref{sec:randomregular_pf}.

\subsubsection{Graphons} In this section we apply our results for dense graph sequences. The asymptotics of dense graphs can be studied using the framework of graph limit theory (graphons), which was developed by Borgs et al. \cite{BCLSV08,BCLSV12} (for a detailed exposition see the book by Lov\'asz \cite{Lov12}), and commonly appears in various popular models for network analysis (see \cite{block_model_graphon,bickel2011method,chatterjee_pd,chatterjee_degree,crane_network_book,gao2018minimax,tang2013universally} and the references therein). For a detailed exposition of the theory of graph limits refer to Lov\' asz \cite{Lov12}. Here, we recall the basic definitions about the convergence of graph sequences. If $F$ and $G$ are two graphs, then define the homomorphism density of $F$ into $G$ by
$$t(F,G) :=\frac{|\hom(F,G)|}{|V (G)|^{|V (F)|},}$$
where  $|\hom(F,G)|$ denotes the number of homomorphisms of $F$ into $G$. In fact, $t(F, G)$ is the proportion of maps $\phi: V (F) \rightarrow V (G)$ which define a graph homomorphism. 

To define the continuous analogue of graphs, consider $\mathscr{W}$ to be the space of all measurable functions from $[0, 1]^2$ into $[0, 1]$ that satisfy $W(x, y) = W(y,x)$, for all $x, y 
\in [0, 1]$. For a simple graph $F$ with $V (F)= \{1, 2, \ldots, |V(F)|\}$, let
$$t(F,W) =\int_{[0,1]^{|V(F)|}}\prod_{(i,j)\in E(F)}W(x_i,x_j) \mathrm dx_1\mathrm dx_2\cdots \mathrm dx_{|V(F)|}.$$

\begin{definition}\cite{BCLSV08,BCLSV12,Lov12}\label{defn:graph_limit} A sequence of graphs $\{G_n\}_{n\geq 1}$ is said to {\it converge to $W$} if for every finite simple graph $F$, 
\begin{equation*}
\lim_{n\rightarrow \infty}t(F, G_n) = t(F, W).
\end{equation*}
\end{definition}

The limit objects, that is, the elements of $\mathscr{W}$, are called {\it graph limits} or {\it graphons}. A finite simple graph $G=(V(G), E(G))$ can also be represented as a graphon in a natural way: Define $$W^G(x, y) =\boldsymbol 1\{(\ceil{|V(G)|x}, \ceil{|V(G)|y})\in E(G)\},$$ that is, partition $[0, 1]^2$ into $|V(G)|^2$ squares of side length $1/|V(G)|$, and let $W^G(x, y)=1$ in the $(i, j)$-th square if $(i, j)\in E(G)$, and 0 otherwise.

The following result gives the threshold for consistency and asymptotic normality of the HT estimator for a sequence of graphs $\{G_n\}_{n \geq 1}$ converging to a graphon $W$. 

\begin{proposition}[Graphons] \label{dense} Fix a connected graph $H$ and suppose $G_n = (V(G_n), E(G_n))$ is a sequence of graphs converging to a graphon $W$ such that $t(H, W) > 0$. Then the following hold: 
\begin{itemize}
\item[(a)] 
If $|V(G_n)|p_n \gg 1$, then the HT estimator $\hat N(H, G_n)$ is consistent for $N(H, G_n)$ and the rescaled statistic $Z(H,G_n)$ is asymptotically normal. Moreover, 
$$\mathrm{Wass}(Z(H, G_n), N(0, 1)) \lesssim_{H} (|V(G_n)|p_n)^{-\frac{1}{2}}.$$
 
\item[(b)] If $|V(G_n)| p_n=O(1)$,  then the HT estimator $\hat N(H, G_n)$ is not consistent for $N(H, G_n)$ and the rescaled statistic $Z(H,G_n)$ is not asymptotically normal. 
 
 \end{itemize}
 
\end{proposition}

Note that the assumption $t(H, W) > 0$ ensures that the density of the graph $H$ in the graphon $W$ is positive, which can be equivalently reformulated as $N(H, G_n) = \Theta(|V(G_n)|^{|V(H)|})$. In fact, as will be evident from the proof, the result above holds for any sequence of graphs with $N(H, G_n) = \Theta(|V(G_n)|^{|V(H)|})$.

\subsection{Organization} The rest of the article is organized as follows. The proof of Proposition \ref{ppn:variance_estimation} is given Section \ref{sec:pf_consistency}. Consequences of our results and future directions are discussed in Section \ref{sec:summary}. The proofs of Theorem \ref{thm:wass}, Proposition \ref{ppn:variance_estimation}, and a more general fourth-moment phenomenon for random multilinear forms are discussed in Appendix \ref{sec:pf_normal_ZHGn}. The thresholds for consistency and normality for the various graph ensembles discussed above in Section \ref{sec:graphs} are proved in Appendix \ref{sec:graphs_pf}. The relevant moment estimates and the proof of Theorem \ref{thm:normality} are given in Appendix \ref{sec:normal_ZHGn_II}. Finally, in Appendix \ref{sec:examples} we compute the asymptotics of the HT estimator in various examples, which illustrate the necessity of the different conditions in the results mentioned above.

\section{Proof of Theorem \ref{thm:consistency}}\label{sec:pf_consistency}

In this section, we prove the necessary and sufficient condition for the consistency of the estimate $\hat N(H, G_n)$. We start with a few definitions.  Fix an $\e>0$. For each set $ A \subset V(G_n)$ and each $\s\in V(G_n)_{|V(H)|}$, define the following events 
\begin{align}\label{eq:bnea} \mathscr{B}_{n,\e}(A):=\{t_H( A )> \e p_n^{|A|} N(H, G_n)\}, \quad \mathscr{B}_{n,\e}(\s)^c:=\bigcap_{A: A \subseteq \s, A \ne \emptyset}\mathscr{B}_{n,\e}(A)^c. \end{align}  Consider the following truncation of $T(H, G_n)$ (recall \eqref{def:stat}): 
\begin{align}\label{eq:THGn_epsilon}
T^+_{\e}(H, G_n)=\frac1{|\ah|}\sum_{\s\in V(G_n)_{|V(H)|}} \prod_{(i, j) \in E(H)} a_{s_i s_j} X_{\s}\ind \{\mathscr{B}_{n,\e}(\s)^c \}.
\end{align} 
Moreover, let $N_{\e}^+(H, G_n):=\frac{1}{p_n^{|V(H)|}}\Ex[T^+_{\e}(H, G_n)]$ be the truncation of the true motif count $N(H, G_n)$. This truncation has the following properties: 

\begin{lemma} \label{claim} Define 
$$M_n:=  \sum_{\substack{ A  \subset V(G_n) \\ 1 \leq  |A| \leq  |V(H)|}}  t_H( A )\ind \{\mathscr{B}_{n,\e}(A) \}.$$ Then the following hold: 
\begin{enumerate}
\item[$(a)$] $\frac{M_n}{2^{|V(H)|}-1} \le  N(H, G_n)- N_{\e}^+(H, G_n)  \le M_n$.
\item[$(b)$] $\Pr(T(H, G_n) \neq T^+_{\e}(H, G_n))\le \frac{M_n}{\e  N(H, G_n)}$.
\end{enumerate}
\end{lemma}	
	
\begin{proof} Note that 
$$\Delta_n:= N(H, G_n)-N_{\e}^+(H, G_n)=\frac1{|\ah|}\sum_{\s\in V(G_n)_{|V(H)|}} \prod_{(i, j) \in E(H)} a_{s_i s_j} \ind\{\mathscr{B}_{n,\e}(\s)\}.$$ 
Since $\mathscr{B}_{n,\e}(\s) = \bigcup_{A: A \subseteq \s, A \ne \emptyset}\mathscr{B}_{n,\e}(A)$ is the union of $2^{|V(H)|}-1$ many sets,  applying the elementary inequality
\begin{equation*}
\frac1m\sum_{r=1}^m \ind\{B_r\} \le \ind\left\{\bigcup\limits_{r=1}^m B_r \right \} \le \sum_{r=1}^m \ind\{B_r\},
\end{equation*}
for any finite collection of sets $B_1, B_2, \ldots, B_m$, gives 
$$\frac{M_n}{2^{|V(H)|}-1} \leq \Delta_n \leq M_n,$$
with 
\begin{align*}
M_n & =\frac1{|\ah|}\sum_{\s\in V(G_n)_{|V(H)|}} \sum_{\substack{ A  \subseteq \s \\ 1 \leq  |A| \leq  |V(H)|}} \prod_{(i, j) \in E(H)} a_{s_i s_j} \ind\{\mathscr{B}_{n,\e}(A)\}  =  \sum_{\substack{ A  \subset V(G_n) \\ 1 \leq  |A| \leq  |V(H)|}} t_H(A)\ind\{\mathscr{B}_{n,\e}(A)\}, 
		\end{align*}
where last equality follows by interchanging the order of the sum and recalling the definition of $t_H(A)$ in \eqref{def:tj}. This proves the result in (a). 

We now proceed to prove (b). For any $A \subset V(G_n)$ define $X_{A}:= \prod_{u \in A} X_u$. Hence,  recalling definitions \eqref{def:stat} and \eqref{eq:THGn_epsilon} gives, 
\begin{align*}
\Pr(T(H, G_n)\neq T^+_{\e}(H, G_n))   & \leq \Ex[T(H, G_n) - T^+_{\e}(H, G_n)]    \\ 
& \le  \sum_{\substack{ A  \subset V(G_n) \\ 1 \leq  |A| \leq  |V(H)|}} \Pr(X_{A}= 1)\ind\{\mathscr{B}_{n,\e}(A)\} \\ 
& \le  \sum_{\substack{ A  \subset V(G_n) \\ 1 \leq  |A| \leq  |V(H)|}} p_n^{|A|}  \ind\{t_H(A)>\e p_n^{|A|} N(H, G_n)\}  \\ 
& \le \frac1{\e  N(H, G_n)}  \sum_{\substack{ A  \subset V(G_n) \\ 1 \leq  |A| \leq  |V(H)|}}  t_H(A) \ind\{t_H(A)>\e p_n^{|A|} N(H, G_n)\} \\ 
& \leq \frac{M_n}{\e N(H, G_n)}
\end{align*}
This completes the proof of (b).
\end{proof}

\noindent{\it Proof of Theorem \ref{thm:consistency} (Sufficiency)}: Recall that condition \eqref{e:cond2} assumes $\frac{M_n}{N(H, G_n)} \rightarrow 0$, where $M_n$ is as defined above in Lemma \ref{claim}. Therefore, Lemma \ref{claim} and the condition in \eqref{e:cond2}  together implies that  
\begin{align}\label{eq:THGn_I}
\frac{\Ex[T_{\e}^+(H, G_n)]}{\Ex[T(H, G_n)]} = \frac{N_{\e}^+(H, G_n)}{N(H, G_n)}\to 1 \quad \text{and} \quad \Pr(T(H, G_n)= T^+_{\e}(H, G_n))\to 1, 
\end{align} 
as $n\to \infty$, for every fixed $\e>0$. Now, write
\[\frac{\hat N(H,G_n)}{\Ex[\hat N(H,G_n)]}=\frac{T(H,G_n)}{\Ex[T(H,G_n)]}=\frac{T(H,G_n)}{T^+_{\e}(H, G_n)} \cdot \frac{T^+_{\e}(H, G_n)}{\Ex[T^+_{\e}(H, G_n)]} \cdot \frac{\Ex[T^+_{\e}(H, G_n)]}{\Ex[T(H,G_n)]}.\] 
Note that, by \eqref{eq:THGn_I}, the first and the third ratios in the RHS above converge to $1$ in probability for every fixed $\e$. Therefore, to prove the consistency of $\hat N(H,G_n)$ it suffices to show that the ratio $\frac{T^+_{\e} (H, G_n)}{\Ex[T^+_{\e} (H, G_n)]} \stackrel{P}\rightarrow 1$, as $n\rightarrow\infty$ followed by $\e\rightarrow 0$. This follows by the using Chebyshev's inequality if we show that 
\begin{align}\label{eq:double}
\lim_{\e\rightarrow 0}\lim_{n\rightarrow\infty}\frac{\vr[T^+_{\e} (H, G_n)]}{(\Ex[T^+_{\e} (H, G_n)])^2}=0.
\end{align}
To this effect, we have 
\begin{align*}
& \vr\left[T^+_{\e}(H, G_n)\right] \\ 
& =  \frac1{|\ah|^2 } \sum_{\substack{\s_1, \s_2 \in V(G_n)_{|V(H)|} \\ \bar \s_1 \bigcap \bar \s_2 \ne \emptyset} }  \mathrm{Cov}(X_{\s}, X_{\s_2} )  M_H(\s_1) M_H(\s_2)  \ind\{\mathscr{B}_{n,\e}(\s_1)^c\} \ind\{\mathscr{B}_{n,\e}(\s_2)^c\}. 
\end{align*}
Now, if $|\bar \s_1  \bigcap \bar \s_2 |=K$, then $\operatorname{Cov}[X_{\s_1}, X_{\s_2}]=p_n^{2{|V(H)|}-K}-p_n^{2|V(H)|} \le p_n^{2|V(H)|-K}$. Thus, 
\begin{align}\label{e:uf}
& \vr\left[T^+_{\e}(H, G_n)\right] \nonumber \\ 
&  \le \frac{1}{|\ah|^2}\sum_{K=1}^{|V(H)|} p_n^{2 |V(H)| - K}\sum_{\substack{\s_1, \s_2 \in V(G_n)_{|V(H)|} \\ K=|\bar \s_1\bigcap \bar \s_2|}}  M_H(\s_1) M_H(\s_2)  \ind\{\mathscr{B}_{n,\e}(\s_1)^c\} \ind\{\mathscr{B}_{n,\e}(\s_2)^c\}.
\end{align}
We now focus on the inner sum in the RHS of \eqref{e:uf}. Note that 
\begin{align}\label{eq:var_12}
\sum_{\substack{\s_1,\s_2 \in V(G_n)_{|V(H)|} \\ K=|\bar \s_1\bigcap \bar \s_2|}} & M_H(\s_1) M_H(\s_2)  \ind\{\mathscr{B}_{n,\e}(\s_1)^c\} \ind\{\mathscr{B}_{n,\e}(\s_2)^c\} \nonumber \\ 
& =\sum_{\substack{A \subset V(G_n) \\ |A| = K }} \sum_{\substack{\s_1,\s_2 \in V(G_n)_{|V(H)|} \\ \bar \s_1  \cap\bar \s_2 = A}}  M_H(\s_1) M_H(\s_2)  \ind\{\mathscr{B}_{n,\e}(\s_1)^c\} \ind\{\mathscr{B}_{n,\e}(\s_2)^c\}  \nonumber  \\ 
& \le \sum_{\substack{A \subset V(G_n) \\ |A| = K }}  \sum_{\s_1: \bar \s_1\supseteq A} \sum_{\s_2: \bar \s_2\supseteq A} M_H(\s_1) M_H(\s_2)  \ind\{\mathscr{B}_{n,\e}(\s_1)^c\} \ind\{\mathscr{B}_{n,\e}(\s_2)^c\}. 
\end{align}
The argument inside the sum now separates out. Therefore, applying the fact
$$\sum_{\s_1: \bar \s_1 \supseteq A } M_H(\s_1)\ind\{\mathscr{B}_{n,\e}(\s_1)^c\} \le \sum_{\bm s_1: \bar \s_1\supseteq A} M_H(\s_1)\ind\{\mathscr{B}_{n,\e}(A)^{c} \}=|\aut(H)|t_H(A)\ind\{\mathscr{B}_{n,\e}(A)^{c} \},$$
it follows  from \eqref{e:uf} and \eqref{eq:var_12} that 
\begin{align*}
\vr\left[T^+_{\e}(H, G_n)\right] & \le \sum_{K=1}^{|V(H)|}p_n^{2|V(H)|-K} \sum_{\substack{A \subset V(G_n) \\ |A| = K }}  t_H(A)^2\ind\{\mathscr{B}_{n,\e}(A)^{c} \}  \nonumber \\ 
& \le \e N(H, G_n)\sum_{K=1}^{|V(H)|} p_n^{2 |V(H)|}  \sum_{\substack{A \subset V(G_n) \\ |A| = K }}  t_H( A )\ind\{\mathscr{B}_{n,\e}(A)^{c}\} \tag*{(since $ t_H( A ) \leq \e p_n^{|A|} N(H, G_n)$ on $\mathscr{B}_{n,\e}(A)^{c}$)} \nonumber \\ 
&\le \e p_n^{2 |V(H)|}  N(H, G_n) \sum_{K=1}^{|V(H)|}  \sum_{\substack{A \subset V(G_n) \\ |A| = K }} t_H(A)  \nonumber \\ 
&  \e p_n^{2 |V(H)|} N(H, G_n)\sum_{K=1}^{|V(H)|}\binom{|V(H)|}{K }  N(H, G_n) 
 = \e p_n^{2 |V(H)|} (2^{|V(H)|} - 1) N(H, G_n)^2  ,
\end{align*}
where the last line uses \eqref{th-iden}. Since \eqref{eq:double} is immediate from this, we have verified sufficiency. \\

\noindent{\it Proof of Theorem \ref{thm:consistency} (Necessity)}: We will show the contrapositive statement, that is, if \eqref{e:cond2} fails, then $\hat N(H, G_n)$ is not consistent for $N(H, G_n)$. Towards this, assume \eqref{e:cond2} fails. Define 
\begin{align}\label{eq:E1}
E_1:= \left\{ X_{\s}=0 \text{ for all } \s\in V(G_n)_{|V(H)|}\mbox{ with }\ind\{\mathscr{B}_{n,\e}(\s)^c\}=0 \right\},
\end{align} 
and, for $1 \leq K \leq |V(H)|$, let 
$$E_{2, K} = \left\{ X_A : = \prod_{u \in A } X_u =0 \text{ for all } A \subset V(G_n) \text{ where } |A| = K \text{ and } \ind\{\mathscr{B}_{n,\e}(A)\}  = 1 \right \}.$$  Take any $\s\in V(G_n)_{|V(H)|}$ with $\ind\{\mathscr{B}_{n,\e}(\s)^c\}=0$. By definition (recall \eqref{eq:bnea}), this implies $\ind\{\mathscr{B}_{n,\e}(A)\}=1$ for some $A\subseteq \s, A\neq \emptyset$. In particular, under the event $\bigcap_{K=1}^{|V(H)|} E_{2, K}$, we have $X_{A}=0$, forcing $X_{\s}=0$. Hence, $E_1 \supset \bigcap_{K=1}^{|V(H)|} E_{2, K}$.  Note that 
$$E_{2, K} = \bigcap_{K=1}^{|V(H)|} \bigcap_{\substack{ A \subset V(G_n) : |A| = K  \\ \ind\{\mathscr{B}_{n,\e}(A)\}  = 1} } \{ X_{A}=0 \},$$
and the event $\{X_A=0\}$ is a decreasing event, for all $A \subset V(G_n)$ with $1 \leq |A| \leq |V(H)|$.\footnote{An event $\mathcal{D} \subseteq \{0, 1\}^{|V(G_n)|}$ is said to be {\it decreasing} if for two vectors $\bm x=(x_a)_{a \in V(G_n)} \in \{0, 1\}^{|V(G_n)|}$ and $\bm y = (y_a)_{a \in V(G_n)} \in   \{0, 1\}^{|V(G_n)|}$, with $\{a: y_a = 1\} \subseteq \{a: x_a = 1\}$, $\bm x \in \mathcal{D}$ implies $\bm y \in \mathcal D$. Then the FKG inequality states that if $\mathcal{D}_1, \mathcal{D}_2 \subseteq \{0, 1\}^{|V(G_n)|}$ are two decreasing events, $\mathbb P(\mathcal{D}_1 \cap \mathcal{D}_2) \geq \mathbb P(\mathcal{D}_1) \mathbb P(\mathcal{D}_2)$ (see \cite[Chapter 2]{inequality}).} Then the FKG inequality between decreasing events for product measures on $\{0, 1\}^{|V(G_n)|}$ \cite[Chapter 2]{inequality} gives,  
\begin{align*}
\Pr(E_1)  \ge \Pr\left(\bigcap\limits_{K=1}^{|V(H)|} E_{2, K} \right) & \ge  \prod_{K=1}^{|V(H)|} \prod_{\substack{ A \subset V(G_n) : |A| = K  \\ \ind\{\mathscr{B}_{n,\e}(A)\}  = 1} }  \Pr(X_{A}=0) \nonumber \\ 
& \ge \prod_{K=1}^{|V(H)|} (1-p_n^K)^{\sum_{A \subset V(G_n) : |A| = K }\ind\{\mathscr{B}_{n,\e}(A)\}} 
\end{align*}
Now, since $p_n$ is bounded away from 1 (recall \eqref{eq:boundp}), there exists a constant $c > 0$ such that $\log(1-p_n^K)> -c p_n^{K}$, for all $1 \leq K \leq |V(H)|$. Hence, 
\begin{align}\label{eq:pfconsistency_I}
\Pr(E_1) &  \ge \exp\left(-c\sum_{K=1}^{|V(H)|} p_n^{|K|} \sum_{ A \subset V(G_n) : |A| = K }\ind\{\mathscr{B}_{n,\e}(A)\}\right)  \nonumber \\ 
&  \ge \exp\left(-\frac c{\e N(H, G_n)}\sum_{K=1}^{|V(H)|} \sum_{A \subset V(G_n)  : |A| = K }t_H(A)\right) \nonumber \\ 
& \ge e^{-\frac{c(2^{|V(H)|}-1)}{\e}} ,
\end{align}
where the last step uses \eqref{th-iden}. Now, since \eqref{e:cond2} does not hold, there exists $\varepsilon > 0$ and $\delta \in (0, 1)$ such that
	\begin{equation*}
	\limsup_{n\to\infty} \frac1{N(H, G_n)}\sum_{K=1}^{|V(H)|} \sum_{ A \subset V(G_n)  : |A| = K  }t_H(A)\ind\{\mathscr{B}_{n,\e}(A)\} > (2^{|V(H)|}-1)\frac{2\delta}{1+\delta},
	\end{equation*}
From Lemma \ref{claim}, it follows that along a subsequence, $N(H, G_n)-N_\e^+(H, G_n) > \frac{2\delta}{1+\delta} N(H, G_n)$, that is, $(1+\delta)  N_\e^+(H, G_n) < (1-\delta) N(H, G_n)$. Thus, by Markov inequality, along a subsequence
\begin{align}\label{eq:pfconsistency_II}
\Pr\left(T^+_{\e}(H, G_n) \ge (1-\delta) p_n^{|V(H)|} N(H, G_n) \right) & \le \Pr\left( T^+_{\e}(H, G_n) \ge (1+\delta) p_n^{|V(H)|} N_\e^+(H, G_n) \right) \nonumber \\ 
& \le \frac1{1+\delta}.
\end{align}
Also, observe that $\{T^+_{\e}(H, G_n)\le (1-\delta)p_n^{|V(H)|} N(H, G_n)\}$ is a decreasing event, because if $\bm X=(X_a)_{a \in V(G_n)} \in \{T^+_{\e}(H, G_n)\le (1-\delta)p_n^{|V(H)|} N(H, G_n)\}$ then any vector $\bm X' = (X_a')_{a \in V(G_n)}$ obtained changing a subset of the ones in $\bm X$ to zeros does not increase the value of $T^+_{\e}(H, G_n)$ and hence,  $\bm X'  \in \{T^+_{\e}(H, G_n)\le (1-\delta)p_n^{|V(H)|} N(H, G_n)\}$. Similarly, $E_1$ (recall definition in \eqref{eq:E1}) is a decreasing event. Hence, by the FKG inequality,  
\begin{align}\label{eq:pfconsistency_III}
\Pr\left(T^+_{\e}(H, G_n)\le (1-\delta)p_n^{|V(H)|} N(H, G_n) \big| E_1\right) \ge \Pr\left(T^+_{\e}(H, G_n)\le (1-\delta)p_n^ {|V(H)|} N(H, G_n) \right). 
\end{align} 
This implies, 
\begin{align*}
\Pr(\hat N(H, G_n) & \le (1-\delta)   N(H, G_n)) \nonumber \\ 
& \ge  \Pr(T(H, G_n)\le (1-\delta)p_n^{|V(H)|}  N(H, G_n)|E_1)\Pr(E_1) \nonumber \\ 
	& \ge \Pr(T^+_{\e}(H, G_n)\le (1-\delta)p_n^{|V(H)|} N(H, G_n)|E_1)\Pr(E_1) \tag*{(since $T(H,G_n)\ge T^+_\e(H,G_n))$} \nonumber \\ 
&  \ge \Pr(T^+_{\e}(H, G_n)\le (1-\delta)p_n^{|V(H)|} N(H, G_n))\Pr(E_1) \tag*{(by \eqref{eq:pfconsistency_III})} \nonumber \\ 
&\ge \frac{\delta}{1+\delta}\Pr(E_1)
	\end{align*}
where the last step uses \eqref{eq:pfconsistency_II}. This is a contradiction to the consistency of $\hat N(H, G_n)$, since $\liminf_{n\rightarrow\infty}\Pr(E_1)>0$ by \eqref{eq:pfconsistency_I},  completing the proof of the desired result.

\section{Discussions and Future Directions}
\label{sec:summary}

The theme that emerges from the examples considered in the paper is that in most of the natural network models, the HT estimator $\hat N(H, G_n)$ is consistent and asymptotically normal whenever the expected number of copies of $H$ in the sampled graph diverges, and inconsistent and not asymptotically normal otherwise. For dense graphs (graphons) this implies, sampling at rate $p_n \gg 1/|V(G_n)|$ ensures that the HT estimator is consistent and asymptotically normal. For sparser graphs one needs to sample at rate $p_n \gg N(H, G_n)^{-\frac{1}{|V(H)|}}$ which can be much larger, depending on the magnitude of $N(H, G_n)$. In particular, this implies that there is a non-trivial sampling rate beyond which HT estimator is consistent for sparser graphs (even for bounded degree graphs), as soon as the number of copies of $H$ in $G_n$ is diverging. 
An interesting question is whether under this assumption ($N(H,G_n)\to \infty$),  it is possible to improve the estimation accuracy of $N(H,G_n)$ using other sampling strategies, such as neighborhood sampling \cite{handcock2010modeling,klusowski2018counting,kolaczyk2009statistical}, snowball sampling \cite{goodman1961snowball}, or random walk based exploration methods \cite{leskovec2006sampling,ribeiro2010estimating}. However, not much is known about the  asymptotic fluctuations of the resulting estimates in these sampling models. In fact, it has been shown recently in \cite{klusowski2018counting} that the natural inverse probability weighted estimator might not be minimax optimal in the neighborhood sampling scheme. Therefore, it is encouraging to see that the HT estimator in the simple (albeit idealized) subgraph sampling model provides consistent and asymptotically exact confidence intervals for large classes of natural network models. These results are the first steps towards understanding properties of more practical (and complicated) models for network sampling, and will provide useful benchmarks for comparing the performances of different estimates arising from other sampling schemes.

From a computational perspective, the subgraph sampling scheme has time complexity $O(|V(G_n)|)$. Since on average the sampled graph as $O(p_n|V(G_n)|)$ vertices, one way to reduce the computational cost is to sample without replacement a uniform random subset of size $N=p_n |V(G_n)|$ from $V(G_n)$, and then consider the induced graph as before. This can be done in $O(N \log N)$ time \cite{sampling_algorithm_gb,sampling_algorithm}, which is faster whenever $N \ll |V(G_n)|$ (up  to a logarithmic factor). In certain situations, the asymptotic properties of the HT estimator in the sampling without replacement  model should be the same as that in the subgraph sampling model with sampling probability $p_n=N/|V(G_n)|$.  For example, we conjecture that using  \cite[Theorem 4]{diaconis1980finite} one should be able to derive consistency of the HT estimator in the sampling without replacement model, at least for certain regimes of $p_n$. In a similar manner, using the asymptotic normality  for the HT estimator in the subgraph sampling model along with  the conditional approach in \cite{bertail2017empirical}, one should be able  verify a similar result for the sampling without replacement model in certain regimes of $p_n$ as well. The exact detection boundary of the sampling without replacement model seems to be an interesting question for possible future research.
\\

\small{
\noindent\textbf{Acknowledgements.} The authors thank Sohom Bhattacharya for pointing out the reference \cite{KIM20071961}, and Jason Klusowski for helpful discussions. The authors also thank the Associate Editor and the anonymous referees for their detailed and thoughtful comments which greatly improved the quality and the presentation of the paper. The authors
also thank Ayoushman Bhattacharya and Nilanjan Chakraborty for pointing out an error in a previous draft. }

	\bibliographystyle{plain}	
	\bibliography{bib}

\appendix

\normalsize

\section{Proofs of Theorem \ref{thm:wass} and Proposition \ref{ppn:variance_estimation}  }  
\label{sec:pf_normal_ZHGn}

Fix $r \geq 2$ and start by defining 
\begin{align}\label{e:notations}
{Z}_{\s}:=X_{\s}-p_n^{r}.
\end{align}   
We begin with some moment estimates in Appendix \ref{sec:expectation}. These estimates are used to complete the proofs of Theorem \ref{thm:wass} and  Proposition \ref{ppn:variance_estimation}  in Appendix \ref{sec:clt_pf} and Appendix \ref{sec:variance_estimation_pf}, respectively.

\subsection{Moment Estimates} 
\label{sec:expectation}

In this section we collect various estimates on the mixed moments of the random variables $\{Z_{\s}: \s \in V(G_n)_r \}$, which arise in the higher-order moments of $Z(H, G_n)$.

\begin{lemma}[Bound on the free product] \label{free-bound} Let $\s_1,\s_2,\ldots,\s_L \in V(G_n)_{r}$ with $a=|\bigcup_{j=1}^L \bar \s_j|$. Assume further that $\bar \s_i \bigcap \bigcup_{j\neq i}\{\bar \s_j\} \neq \emptyset$, for all $1 \leq i \leq L$. Then for any $p_n \in [0,1]$, \begin{align} \label{e:free-bound}
	\left|\Ex[{Z}_{\s_1}{Z}_{\s_2}\ldots {Z}_{\s_L}]-p_n^{a}\right| \leq (2^L - 1) p_n^{a+1}.
	\end{align}
Moreover, for $L=2$, we also have $\Ex[{Z}_{\s_1}{Z}_{\s_2}]\ge p_n^a(1-p_n) \ge 0$.	
\end{lemma}

\begin{proof} Using  $\Ex[\prod_{j=1}^{L} {X}_{\s_j}]=p_n^a$, gives 
\begin{align} \label{e:diffeq} 
 \left | \Ex[{Z}_{\s_1}{Z}_{\s_2}\ldots {Z}_{\s_L}] - p_n^a \right| & =
\left|\Ex[{Z}_{\s_1}{Z}_{\s_2}\ldots {Z}_{\s_L}] - \Ex[{X}_{\s_1}{X}_{\s_2}\ldots {X}_{\s_L} ] \right| \nonumber \\ 
& \le \sum_{K=0}^{L-1}\sum_{1 \leq j_1 < j_2 <\ldots < j_t \leq K } p_n^{|V(H)|(L-K)}\Ex\left[\prod_{t=1}^K X_{\s_{j_t}} \right].
	\end{align}
Note that $\Ex\left[\prod_{t=1}^K X_{\s_{j_t}} \right] =p_n^{|\bigcup_{t=1}^K \bar \s_{j_t}|}$, and 
\begin{align*} 
a=\left|\bigcup_{j=1}^L \bar \s_j \right| < \left|\bigcup_{t=1}^{K} \bar \s_{j_t} \right| + (L-K) r, \end{align*} 
where the inequality above is strict because of the given condition. Therefore, each term on the RHS of \eqref{e:diffeq} can be bounded above by $p_n^{a+1}$. Note that there are  $2^{L}-1$ terms in the double sum in the RHS of \eqref{e:diffeq}.  Hence, the RHS of \eqref{e:diffeq} is bounded above by $(2^{L}-1)p_n^{a+1}$, which completes the proof of \eqref{e:free-bound}.
	
Now, for $L=2$, by the tower property of conditional expectations it follows that
\begin{align*}
\Ex\left[\Ex\left({Z}_{\s_1}{Z}_{\s_2}\Big|\{ X_j: j \in \bar \s_1 \bigcap \bar \s_2\} \right) \right] & =p_n^{|\bar \s_1\setminus	\bar \s_2|+|\bar \s_2\setminus \bar \s_1|} \Ex\left[ \prod_{j \in  \bar \s_1 \bigcap \bar \s_2 } X_j \right] - p^{2 r} \nonumber \\ 
& =p_n^{\left|\bar \s_1\bigcup \bar \s_2 \right|}(1-p_n^{\left|\bar \s_1\bigcap \bar \s_2 \right|}) \ge p_n^a(1-p_n),
\end{align*}
since $|\bar \s_1\bigcup \bar \s_2 | = a$ and $|\bar \s_1\bigcap \bar \s_2| \geq 1$. 
\end{proof}

Note that the condition $ \s_i \bigcap \bigcup_{j\neq i}\{\s_j\} \neq \emptyset$, for all $1 \leq i \leq L$, ensures that the $\Ex[Z_{\s_1} Z_{\s_2} \ldots Z_{\s_L}] \ne 0$. Otherwise, one of the ${Z}_{\s_i}$ factors out from the expectation to yield a zero expectation. 

\begin{lemma}[Bound for absolute product] \label{abs-bound}  Let $\s_1,\s_2,\ldots,\s_L \in V(G_n)_{r}$ with $a=| \bigcup_{j=1}^L \bar \s_j |$. If $p_n$ satisfies \eqref{eq:boundp}, then \begin{align}\label{eq:moment_Z_bound}
	 (1-\kappa)^L p_n^{a}  \le \Ex|{Z}_{\s_1}{Z}_{\s_2}\ldots {Z}_{\s_L}|\le  (L+2)! p_n^{a} .
	\end{align} 
\end{lemma}
\begin{proof} To prove the lower bound observe that 
$$\Ex\left|\prod_{j=1}^L{Z}_{\s_j}\right| \ge (1-p_n^{r})^L \Pr(X_{\s_j}=1\mbox{ for all } 1\leq j \leq L) =(1-p_n^{r})^Lp_n^{a},$$
from which the desired bound follows by using \eqref{eq:boundp}.
	\\
	 
	For the upper bound we use induction on $L$. For $L=1$, we have $$\Ex\left|Z_{\s_1}\right| \leq \Ex\left|X_{\s_1}\right| + p_n^{r}  \le 2p_n^{r}.$$ This proves the claim for $L=1$. Now, suppose the upper bound in \eqref{eq:moment_Z_bound} holds for $\s_1,\s_2,\ldots,\s_L \in V(G_n)_{r}$ for some value of $L$. Then consider the case when $\s_1,\s_2,\ldots,\s_{L+1} \in V(G_n)_{r}$  with $|\bigcup_{j=1}^{L+1} \bar \s_j |=a$. Note that
	\begin{align} \label{e:abs-up}
	\Ex\left|\prod_{j=1}^{L+1} {Z}_{\s_j}\right| & \le  (1-p_n^{r})^{L+1}\Pr(X_{\s_m}=1\mbox{ for all } 1 \leq m \leq L+1)+\sum_{m=1}^{L+1}\Ex\left[\left|\prod_{j=1}^{L+1} {Z}_{\s_j}\right|\ind\{X_{\s_m}=0\} \right] \nonumber \\ 
	& \le  p_n^{\left|\bigcup_{j=1}^{L+1} \bar \s_j \right|}+\sum_{m=1}^{L+1}p_n^{r}\Ex\left|\prod_{j \neq m} {Z}_{\s_j}\right|\ind\{X_{\s_m}=0\} \nonumber \\ 
	& \le  p_n^a+p_n^{r}\sum_{j=1}^{L+1}\Ex\prod_{j \neq m} \left|{Z}_{\s_j}\right| \nonumber \\ 
	& \le  p_n^a+ (L+2)! (L+1) p_n^{r + \left|\bigcup_{j \ne m} \bar \s_j \right|} ,
	\end{align}
where the last step uses the induction hypothesis. Note that 
$$r+\left|\bigcup_{j \ne m} \bar \s_j \right| \geq  \left|\bigcup_{j=1}^{L+1} \bar \s_j \right| = a.$$ Hence, the  RHS~of (\ref{e:abs-up}) can be bounded by $p_n^a[1+(L+2)! (L+1)]\le p_n^a(L+3)!$, thus verifying the result for $L+1$. This proves  the desired upper bound by induction. 
\end{proof}

\subsection{Proof of Theorem \ref{thm:wass}}
\label{sec:clt_pf}

This is a consequence of a more general result about the asymptotic normality of multilinear forms in the variables $\{X_u: u \in V(G_n)\}$, which might be of independent interest. To state this general result, we need a few definitions: Fix an integer $r \geq 1$ and consider a function $\alpha : V(G_n)_{r} \to \mathbb{R}_{\ge 0}$. Define 
$$S(\alpha, G_n):=\sum_{\s\in V(G_n)_{r}} \alpha(\s)X_{\s},$$
where $X_{\s} = \prod_{u=1}^r X_{s_u}$. Assume $\sigma(\alpha,G_n)^2:=\operatorname{Var}(S(\alpha, G_n) )>0$,  and consider the rescaled statistic $U(\alpha,G_n)$: 
\begin{align}\label{zalpha}
U(\alpha,G_n) & :=\frac{S(\alpha,G_n)-\Ex[S(\alpha,G_n)]}{\sigma(\alpha,G_n)}. 
\end{align}
The following theorem gives a quantitative error bound between $U(\alpha,G_n) $ and the standard normal distribution (in the Wasserstein distance) in terms of the expected value of 
\begin{align}\label{eq:walpha}
W(\alpha,G_n) & :=\sum_{\{\s_1,\s_2,\s_3,\s_4\}\in \mathcal{K}_{n,4}}\alpha(\s_1)\alpha(\s_2)\alpha(\s_3)\alpha(\s_4)|Z_{\s_1}Z_{\s_2}Z_{\s_3}Z_{\s_4}|.
\end{align}
where $Z_{\s}=X_{\s}-p_n^{r}$. This error bound can then be expressed in term of the fourth-moment difference $\Ex[U(\alpha,G_n)^4]-3$ for $p_n$ small enough, which shows 
$U(\alpha,G_n) \stackrel{D} \rightarrow N(0, 1)$ whenever $\Ex[U(\alpha,G_n)^4] \rightarrow 3$. 

\begin{theorem}\label{thm:wass1}  Fix an integer $1 \leq r \leq \frac{|V(G_n)|}{4}$, a network $G_n=(V(G_n),E(G_n))$, and a sampling ratio $p_n$ which satisfies \eqref{eq:boundp}. Then \begin{align}\label{e:wass1}
\operatorname{Wass}(U(\alpha, G_n), N(0,1)) \lesssim \frac{r}{(1-\kappa)^3} \sqrt{ \frac{\Ex[W(\alpha,G_n)]}{\sigma(\alpha,G_n)^4}},
\end{align}
where $U(\alpha, G_n)$ and $W(\alpha,G_n)$ are defined in \eqref{zalpha} and \eqref{eq:walpha} respectively.  Moreover,  if $p_n \in \left(0, \tfrac1{20}\right]$ then, $\frac{\Ex[W(\alpha,G_n)]}{\sigma(\alpha,G_n)^4} \lesssim  \Ex[U(\alpha,G_n)^4]-3$, and as a consequence, 
\begin{align}\label{e:UGn}
\operatorname{Wass}(U(\alpha, G_n), N(0,1)) \lesssim_\alpha \sqrt{ \Ex[U(\alpha,G_n)^4]-3 } . 
\end{align}
\end{theorem}

Given $H=(V(H), E(H))$, the result in Theorem \ref{thm:wass} follows from Theorem \ref{thm:wass1} above with $r=|V(H)|$ and $\alpha(\s)= \frac{1}{|Aut(H)|} \prod_{(i,j)\in E(H)}a_{s_i,s_j}$. The proof of Theorem \ref{thm:wass1} is given below. \\ 
 
\noindent {\it Proof of Theorem} \ref{thm:wass1}: Hereafter, we will drop the dependency on $\alpha$ and $G_n$ from the notations $\sigma(\alpha,G_n), U(\alpha,G_n),$ and $W(\alpha,G_n)$ and denote them by $\sigma, U,$ and $W$, respectively.  Define
 \begin{align}\label{eq:Y_I}
 Y_{\s}:=\alpha(\s)(X_{\s}-p_n^{r}) \quad \text{and} \quad U=\frac1{\sigma}\sum_{\s\in V(G_n)_{r}} Y_{\s}. \end{align}
Moreover, for $\bm s \in V(G_n)_r$ define,  
$$U_{\s}:=\frac1{\sigma}\sum_{\s'\in V(G_n)_{r}: \bar\s\cap\bar\s'=\emptyset} Y_{\s'},$$
which is the sum over all $Y_{\s'}$ such that $\bar{\s}'$ disjoint from $\bar{\s}$. 
Clearly, $Y_{\s}$ and $U_{\s}$ are independent for each $\s\in V(G_n)_{r}$. We now take a twice continuously differentiable function $f : \mathbb R \rightarrow \mathbb R$ such that $|f|_\infty \le 1, |f'|_\infty \le \sqrt{\frac2\pi}, |f''|_\infty \le 2.$ 
	Note that, because $\Ex[ Y_{\s} f(U_{\s}) ]  = \Ex[ Y_{\s}] \Ex[ f(U_{\s}) ]  = 0$, 
	\begin{align}\label{eq:defn_12}
	\Ex[Uf(U)]-\Ex[f'(U)] & =\frac1{\sigma}\sum_{\s\in V(G_n)_{r} } \Ex[Y_{\s}(f(U)-f(U_{\s}))] - \Ex[f'(U)] =A_1 +A_2,
	\end{align}
	where
	\begin{align} 
	\label{e:A1def}
	& A_1  :=  \frac1{\sigma}\sum_{\s\in V(G_n)_{r}} \Ex[Y_{\s}(f(U)-f(U_{\s})-(U-U_{\s})f'(U))] \\ \label{e:A2def}
	& A_2  :=\frac1{\sigma}\sum_{\s\in V(G_n)_{r}} \Ex[Y_{\s}(U-U_{\s})f'(U)]-\Ex[f'(U)] 
	\end{align} 
The proof is now completed in three steps: (1) $|A_1| \lesssim \frac{r}{\sigma^2(1-\kappa)^3} \sqrt{ \Ex[W] }$ (Lemma \ref{lm:A1}), (2) $|A_2|  \lesssim \frac{1}{\sigma^2} \sqrt{ \Ex[W] }$ (Lemma \ref{lm:A2}), and (3) $\Ex[W]\lesssim \sigma^4(\Ex[U^4]-3)$, for all $p_n \in (0,  \frac1{20}]$ (Lemma \ref{lm:W_4}).

\begin{lemma}\label{lm:A1} For $A_1$ as defined in \eqref{e:A1def}, $|A_1| \lesssim  \frac{r}{\sigma^2(1-\kappa)^3} \sqrt{ \Ex[W] }$. 
\end{lemma} 

\begin{proof} Using \eqref{e:A1def}, note that 
	\begin{align}\label{e:A1bound}
	|A_1| & \le \frac1{2\sigma}\sum_{\s\in V(G_n)_{r}}\Ex\left|Y_{\s}(U-U_{\s})^2\right||f''|_{\infty}  \le \frac1{\sigma^3}\sum_{\substack{\s_1, \s_2 \s_3 \in V(G_n)_{r} \\ \bar \s_1\bigcap \bar \s_2 \neq \emptyset, \bar \s_1\bigcap \bar \s_3 \neq\emptyset}} \Ex|Y_{\s_1}Y_{\s_2}Y_{\s_3}|.
	\end{align}
By Lemma \ref{abs-bound}, each term in the sum above can be bounded as follows: 
	$$\Ex|Y_{\s_1}Y_{\s_2}Y_{\s_3}| \lesssim \alpha(\s_1)\alpha(\s_2)\alpha(\s_3)p_n^{\left|\bar \s_1 \bigcup \bar \s_2 \bigcup \bar \s_3 \right|}.$$ 
Therefore, from \eqref{e:A1bound}, 
\begin{align}\label{e:A1bound_II}
	|A_1| & \lesssim \frac1{\sigma^3}\sum_{\substack{\s_1, \s_2 \s_3 \in V(G_n)_{r} \\ \bar \s_1\bigcap \bar \s_2 \neq \emptyset, \bar \s_1\bigcap \bar \s_3 \neq\emptyset}}  \alpha(\s_1)\alpha(\s_2)\alpha(\s_3)p_n^{\left|\bar \s_1 \bigcup \bar \s_2 \bigcup \bar \s_3 \right|} \nonumber \\
	& \lesssim \frac{1}{\sigma^3} \sum_{K=0}^{r-1} \sum_{L=0}^{r-1} p_n^{r + K + L}  \sum_{\substack{  \s_1,  \s_2  \in V(G_n)_{r} \\ | \bar \s_2\setminus \bar \s_1|=K}}\alpha(\s_1)\alpha(\s_2)\sum_{\substack{ \s_3\in V(G_n)_{r} \\ | \bar \s_3\setminus ( \bar \s_1\bigcup \bar \s_2 ) |=L}}\alpha( \s_3), 
	\end{align} 
using $\left|\bar \s_1 \bigcup \bar \s_2 \bigcup \bar \s_3 \right| = r + K + L$, if $ | \bar \s_2\setminus \bar \s_1|=K$ and $ | \bar \s_3\setminus ( \bar \s_1\bigcup \bar \s_2 ) |=L$.

For $0\le K,L\le r - 1$, define 
	\begin{align*}
	N_{K, L}(\alpha, G_n)  & = \sum_{\substack{\s_1, \s_2  \in V(G_n)_{r} \\ |\bar \s_2\setminus \bar \s_1|=K}}\alpha(\s_1)\alpha(\s_2)\left(\sum_{\substack{\s_3\in V(G_n)_{r} \\ |\bar \s_3\setminus (\bar \s_1\bigcup \bar \s_2)|=L}}\alpha(\s_3)\right)^2,
	\end{align*} 
and use \eqref{eq:walpha} along with Lemma \ref{abs-bound} with $L\mapsto 4$) to get
	\begin{align*}
	\Ex[W]  & \gtrsim (1-\kappa)^4\sum_{\substack{\s_1,\s_2\in V(G_n)_{r} \\ \bar \s_1\bigcap \bar \s_2\neq\emptyset}} \sum_{\substack{\s_3\in V(G_n)_{r} \\ |\bar \s_3\bigcap (\bar \s_1\cup \bar \s_2)|=L}}\sum_{\substack{\s_4\in V(G_n)_{r} \\ |\bar \s_4\bigcap (\bar \s_1\cup \bar \s_2)|=L}} p_n^{\left|\bar \s_1 \bigcup \bar \s_2 \bigcup \bar \s_3 \bigcup \bar \s_4 \right|} \alpha(\s_1)\alpha(\s_2)\alpha(\s_3)\alpha(\s_4) \\ 
	& \gtrsim (1-\kappa)^4p_n^{r+ K + 2 L}N_{K, L}(\alpha,G_n) , 
	\end{align*} 
where the last step uses the fact that $\left|\bar \s_1 \bigcup \bar \s_2 \bigcup \bar \s_3 \bigcup \bar \s_4 \right| \le r + K + 2 L$, if $|\bar \s_2\setminus \bar \s_1|=K$, $|\bar \s_3\setminus \{\bar \s_1\bigcup \bar \s_2\}|=L$, and $| \bar \s_4\setminus \{ \bar \s_1\bigcup \bar \s_2\}|=L$. Also, note that for any $K \in [0, r-1]$, by Lemma \ref{abs-bound} (with $L\mapsto 2$) we have
$$\sigma^2 \gtrsim  (1-\kappa)^2p_n^{r+K} \sum_{\substack{\s_1, \s_2  \in V(G_n)_{r} \\ |\bar \s_2\setminus \bar \s_1|=K}}\alpha(\s_1)\alpha(\s_2).$$ Thus, by Cauchy Schwarz inequality we have
	\begin{align*}
	\sigma^2\Ex[W] & \gtrsim (1-\kappa)^6p_n^{2r+2K+2L} N_{K, L}(\alpha,G_n) \sum_{\substack{\s_1, \s_2  \in V(G_n)_{r} \\ |\bar \s_2\setminus \bar \s_1|=K}}\alpha(\s_1)\alpha(\s_2)  \\ 
	& \gtrsim (1-\kappa)^6\left[p_n^{r+ K + L}\sum_{\substack{\s_1, \s_2  \in V(G_n)_{r} \\ |\bar \s_2\setminus \bar \s_1|=K}}\alpha(\s_1)\alpha(\s_2)\sum_{\substack{\s_3\in V(G_n)_{r} \\ |\bar \s_3\setminus \{\bar \s_1\bigcup \bar\s_2\}|=L}} \alpha(\s_3)\right]^2. 
	\end{align*}
Therefore, from \eqref{e:A1bound_II}, $|A_1| \lesssim \frac{r}{\sigma^2(1-\kappa)^3} \sqrt{ \Ex[W] }$, completing the proof of the lemma.  
\end{proof}

\begin{lemma}\label{lm:A2} For $A_2$ as defined in \eqref{e:A2def}, $|A_2| \lesssim \frac{1}{\sigma^2} \sqrt{\Ex[W]}$. 
\end{lemma} 

\begin{proof} 
Setting $$S=\sigma\sum_{\s\in V(G_n)_{r}} Y_{\s}(U-U_{\s})=\sum_{\substack{\s_1,\s_2\in V(G_n)_{r} \\ \bar \s_1\bigcap \bar \s_2\neq\emptyset}} Y_{\s_1}Y_{\s_2}$$ we have $\Ex[S]=\sigma^2$. Thus, recalling $|f'|_\infty \leq \sqrt{\frac{2}{\pi}}$, 
	\begin{align}\label{eq:A2_I}
	|A_2| = \left|\Ex\left[f'(U)\left(\frac{S}{\sigma^2}-1\right)\right] \right| \le |f'|_{\infty}\Ex\left|\frac{S}{\sigma^2}-1\right| \lesssim \frac1{\sigma^2}\sqrt{\vr[S]}.
	\end{align}
Now, observe that
	\begin{align}\label{eq:A2_II}
	\frac1{\sigma^4}\vr[S] & = \frac1{\sigma^4}\sum_{\substack{\s_1,\s_2\in V(G_n)_{r} \\ \bar \s_1\bigcap \bar \s_2\neq\emptyset}} \sum_{\substack{\s_3,\s_4\in V(G_n)_{r} \\ \bar \s_3 \bigcap \bar\s_4\neq\emptyset}} \operatorname{Cov}[Y_{\s_1}Y_{\s_2},Y_{\s_3}Y_{\s_4}] \nonumber \\  
	& =   \frac{1}{\sigma^4 } \underbrace{\sum_{\substack{\s_1,\s_2\in V(G_n)_{r} \\ \bar \s_1\bigcap \bar \s_2\neq\emptyset}} \sum_{\substack{\s_3,\s_4\in V(G_n)_{r} \\ \bar \s_3\bigcap \bar \s_4\neq\emptyset}}}_{\{\bar \s_1 \bigcup \bar \s_2\}\bigcap \{\bar \s_3 \bigcup \bar \s_4\}\neq \emptyset}   \operatorname{Cov}[Y_{\s_1}Y_{\s_2},Y_{\s_3}Y_{\s_4}] \nonumber \\  
& \leq  \frac{1}{\sigma^4 } \underbrace{\sum_{\substack{\s_1,\s_2\in V(G_n)_{r} \\ \bar \s_1\bigcap \bar \s_2\neq\emptyset}} \sum_{\substack{\s_3,\s_4\in V(G_n)_{r} \\ \bar \s_3\bigcap \bar \s_4\neq\emptyset}}}_{\{\bar \s_1 \bigcup \bar \s_2\}\bigcap \{\bar \s_3 \bigcup \bar \s_4\}\neq \emptyset}  \Ex[Y_{\s_1}Y_{\s_2} Y_{\s_1}Y_{\s_2}] \nonumber \\ 
& \le \tfrac{1}{\sigma^{4}} \Ex[W], 
	\end{align}
where the second equality is because $\operatorname{Cov}[Y_{\s_1}Y_{\s_2},Y_{\s_3}Y_{\s_4}]  = 0$ when $\{\bar \s_1 \bigcup \bar \s_2\} \bigcap \{\bar \s_3 \bigcup \bar\s_4\}=\emptyset$, third inequality follows from $\Ex[Y_{\s_1}Y_{\s_2}] \geq 0$, $\Ex[Y_{\s_3}Y_{\s_4}] \ge 0$ (by Lemma \ref{free-bound}), and  the last step uses the definition of $W=W(\alpha, G_n)$ in  \eqref{eq:walpha}. Combining \eqref{eq:A2_I} and \eqref{eq:A2_II} the proof of the lemma follows. 
\end{proof}

\begin{lemma}\label{lm:W_4} For $W$ as defined in \eqref{eq:walpha}, $\Ex[W] \lesssim \sigma^4(\Ex[U^4]-3)$, for all $p_n \in (0,  \frac1{20}]$.  
\end{lemma}

\begin{proof} Note that 
	\begin{align} \label{e:sig-iden}
	1=\Ex[U^2]^2=\frac1{\sigma^4}\sum_{\substack{\s_1,\s_2\in V(G_n)_{r} \\ \bar \s_1\bigcap \bar \s_2\neq\emptyset}} \Ex[Y_{\s_1}Y_{\s_2}]\sum_{\substack{\s_3,\s_4\in V(G_n)_{r} \\ \bar \s_3 \bigcap \bar\s_4\neq\emptyset}} \Ex[Y_{\s_3}Y_{\s_4}]
	\end{align}
and 	
\begin{align}\label{eq:expectation_4}
\Ex[U^4]=\frac1{\sigma^4}\sum_{\s_1,\s_2,\s_3,\s_4\in V(G_n)_{r}} \Ex[Y_{\s_1}Y_{\s_2}Y_{\s_3}Y_{\s_4}] . 
\end{align} 
Given $L \geq 1$ and a collection of $r$-tuples $\{\s_1,\ldots,\s_L\}$ from $V(G_n)_{r}$, we say that the collection is {\it weakly connected}, if 
\begin{align}\label{eq:weak_connected}
\bar \s_i \bigcap \left(\bigcup_{j \ne i} \bar \s_j \right) \ne \emptyset,\text{ for all }1 \leq i \leq L.
\end{align}
(Any collection which is connected is also weakly connected, but the converse is not necessarily true.) Note that, since the random variables $Y_{\s_1}, Y_{\s_2}, Y_{\s_3}$, and $Y_{\s_4}$ have mean zero, $\Ex[Y_{\s_1}Y_{\s_2}Y_{\s_3}Y_{\s_4}]$ is zero if $\{\s_1,\s_2,\s_3,\s_4\} $ is not weakly connected. Hence, there are only two ways in which $\Ex[Y_{\s_1}Y_{\s_2}Y_{\s_3}Y_{\s_4}]$ is non-zero: (a) $\{\s_1,\s_2,\s_3,\s_4\}$ is connected and (b) $\{\s_1,\s_2,\s_3,\s_4\}$ is weakly connected with two connected components each consisting of two $r$-tuples from $\{\s_1,\s_2,\s_3,\s_4\}$.  Since in the second case there are three ways to form the pairing, the sum in \eqref{eq:expectation_4} gives, 
\begin{align}\label{eq:claim_0}
\notag& \Ex[U^4] -3 \\ 
\notag & = \frac1{\sigma^4}\sum_{\{\s_1,\s_2,\s_3,\s_4\} \in \mathcal{K}_{n,4}} \Ex[Y_{\s_1}Y_{\s_2}Y_{\s_3}Y_{\s_4}] +\frac3{\sigma^4} \underbrace{\sum_{ \{\s_1,\s_2 \} \in \mathcal{K}_{n, 2} } \sum_{\{\s_3,\s_4 \} \in \mathcal{K}_{n, 2} }}_{\{\bar \s_1 \bigcup \bar \s_2\} \bigcap \{ \bar \s_3 \bigcup \bar \s_4\}= \emptyset} \Ex[Y_{\s_1}Y_{\s_2}Y_{\s_3}Y_{\s_4}] - 3 \\ 
\notag & = \frac1{\sigma^4}\sum_{\{\s_1,\s_2,\s_3,\s_4\} \in \mathcal{K}_{n,4}} \Ex[Y_{\s_1}Y_{\s_2}Y_{\s_3}Y_{\s_4}] +\frac3{\sigma^4} \underbrace{\sum_{ \{\s_1,\s_2 \} \in \mathcal{K}_{n, 2} } \sum_{\{\s_3,\s_4 \} \in \mathcal{K}_{n, 2} }}_{\{\bar \s_1 \bigcup \bar \s_2\} \bigcap \{ \bar \s_3 \bigcup \bar \s_4\}= \emptyset} \Ex[Y_{\s_1}Y_{\s_2}]\Ex[Y_{\s_3}Y_{\s_4}] - 3 \\
 &=\frac{1}{\sigma^4}\sum_{\{\s_1,\s_2,\s_3,\s_4\} \in \mathcal{K}_{n,4}} \Ex[Y_{\s_1}Y_{\s_2}Y_{\s_3}Y_{\s_4}] - \frac3{\sigma^4} \underbrace{\sum_{ \{\s_1,\s_2 \} \in \mathcal{K}_{n, 2} } \sum_{\{\s_3,\s_4 \} \in \mathcal{K}_{n, 2} }}_{\{\bar \s_1 \bigcup \bar \s_2\} \bigcap \{ \bar \s_3 \bigcup \bar \s_4\} \ne \emptyset} \Ex[Y_{\s_1}Y_{\s_2}]\Ex [Y_{\s_3}Y_{\s_4}]  , 
\end{align} 
where the second step follows from independence and the last step uses \eqref{e:sig-iden}. 

We now claim that \begin{align}\label{eq:claim_1}\Ex[Y_{\s_1}Y_{\s_2}Y_{\s_3}Y_{\s_4}] \ge \frac{15}{4} \Ex[Y_{\s_1}Y_{\s_2}]\Ex[Y_{\s_3}Y_{\s_4}].
\end{align} Given \eqref{eq:claim_1}, it follows from \eqref{eq:claim_0} that
\[\Ex[U^4]-3\gtrsim \frac{1}{\sigma^4}\sum_{\{\s_1,\s_2,\s_3,\s_4\} \in \mathcal{K}_{n,4}} \Ex[Y_{\s_1}Y_{\s_2}Y_{\s_3}Y_{\s_4}],\]
from which the desired bound follows on using Lemmas \ref{free-bound} and \ref{abs-bound}. Thus, it suffices to verify \eqref{eq:claim_1}. To this effect, define $Z_{\bm s}=X_{\bm s} - p_n^r$.  Now, note that for $p_n\le\frac1{20}$ we have $\frac{1-15p_n}{p_n}\ge \frac{15(1+3p_n)^2}{4}$. Then applying Lemma \ref{free-bound} 
and the inequality $\left| \bar \s_1 \bigcup \bar \s_2 \bigcup \bar \s_3 \bigcup \bar \s_4 \right| \le \left| \bar \s_1 \bigcup \bar \s_2 \right| + \left|\bar \s_3 \bigcup \bar \s_4 \right|-1$ (since $\{\bar \s_1 \bigcup \bar \s_2\} \bigcap \{ \bar \s_3 \bigcup \bar \s_4\} \ne \emptyset$) gives, 
	\begin{align*}
	\Ex[{Z}_{\s_1}{Z}_{\s_2}{Z}_{\s_3}{Z}_{\s_4}]  \ge p_n^{ \left| \bar \s_1 \bigcup \bar \s_2 \bigcup \bar \s_3 \bigcup \bar \s_4 \right| }(1-15p_n)  & \ge p_n^{\left| \bar \s_1 \bigcup \bar \s_2 \right| + \left|\bar \s_3 \bigcup \bar \s_4 \right| }\frac{1-15p_n}{p_n}  \\ 
	& \ge  \frac{15}{4}(1+3p_n)p_n^{ \left|\bar \s_1 \bigcup \bar \s_2 \right| }(1+3p_n)p_n^{ \left|\bar \s_3 \bigcup \bar \s_4 \right| }   \\ 
	& \ge \frac{15}{4}\Ex[{Z}_{\s_1}{Z}_{\s_2}]\Ex[{Z}_{\s_3}{Z}_{\s_4}] . 
	\end{align*}
Here, in the last step we used Lemma \ref{free-bound} which implies that $\Ex[{Z}_{\s_1}{Z}_{\s_2}] \leq p_n^{\left|\bar \s_1 \bigcup \bar \s_2 \right|} + 3 p_n^{\left|\bar \s_1 \bigcup \bar \s_2 \right| +1 }$, and similarly,  $\Ex[{Z}_{\s_3}{Z}_{\s_4}] \leq p_n^{\left|\bar \s_3 \bigcup \bar \s_4 \right|} + 3 p_n^{\left|\bar \s_3 \bigcup \bar \s_4 \right| +1 }$.  Also, note that $\Ex[{Z}_{\s_1}{Z}_{\s_2}] \geq 0$ and $\Ex[{Z}_{\s_3}{Z}_{\s_4}] \geq 0$ by Lemma \ref{free-bound}. Now, since $ Y_{\bm s} = \alpha(\bm s) Z_{\bm s}$ (recall \eqref{eq:Y_I}), 
	$$\Ex[{Y}_{\s_1}{Y}_{\s_2}{Y}_{\s_3}{Y}_{\s_4}] \ge \frac{15}{4}\Ex[{Y}_{\s_1}{Y}_{\s_2}] \Ex[{Y}_{\s_3}{Y}_{\s_4}],$$
thus verifying \eqref{eq:claim_1}. This completes the proof of the lemma.

\end{proof}

Lemma \ref{lm:A1} and Lemma \ref{lm:A2}, together with \eqref{eq:defn_12} gives,  
	\begin{align*}
	\Ex[Uf(U)]-\Ex[f'(U)] &  \lesssim \frac{r}{\sigma^2(1-\kappa)^3} \sqrt{ \Ex[W] },
	\end{align*}
for any twice continuously differentiable function $f : \mathbb R \rightarrow \mathbb R$ such that $|f|_\infty \le 1, |f'|_\infty \le \sqrt{\frac2\pi}, |f''|_\infty \le 2.$ Taking a supremum over $f$ in this class and using \cite[Lemma 1]{chyanotes} gives the desired conclusion in part (a). The conclusion in part (b) then follows from  Lemma \ref{lm:W_4}. \hfill $\Box$

\subsection{Proof of Proposition \ref{ppn:variance_estimation} }\label{sec:variance_estimation_pf} 

We begin with the proof of (a). For this it suffices to show that $\sigma(H,G_n)^2=o((\Ex[T(H,G_n)])^2)$. This follows from the more general bound 
\begin{align}\label{eq:gen_bound}
\sigma(H,G_n)^6 \lesssim_{H, \kappa} \Ex[W_n] (\Ex[T(H,G_n)])^2,  
\end{align}
since $\Ex[W_n]=o(\sigma(H, G_n)^4)$. For verifying \eqref{eq:gen_bound},
fixing $L\in [1, |V(H)|-1]$ we consider the following count:
\begin{align*}
N_{L}(H, G_n) 
& = \sum_{\substack{\s_1 \in V(G_n)_{|V(H)|}}}M_H(\s_1)\left(\sum_{\substack{\s_2\in V(G_n)_{|V(H)|} \\ |\bar \s_2\setminus \bar \s_1|=L}}M_H(\s_2)\right)^3.
\end{align*}
Then recalling the definition of \eqref{W_n:defn} and by Lemma \ref{abs-bound} (with $r=|V(H)|$) gives, 
\begin{align*} 
\Ex[W_n] & \gtrsim \sum_{\s_1, \s_2, \s_3, \s_4 \in \mathcal{K}_{n, 4}}  \Ex|Z_{\s_1} Z_{\s_2} Z_{\s_3} Z_{\s_4}| M_H(\s_1)M_H(\s_2)M_H(\s_3)M_H(\s_4)  \nonumber \\ 
& \gtrsim_\kappa \sum_{\s_1, \s_2, \s_3, \s_4 \in \mathcal{K}_{n, 4}}  p_n^{\left|\bar \s_1 \bigcup \bar \s_2 \bigcup \bar \s_3 \bigcup \bar \s_4 \right|} M_H(\s_1)M_H(\s_2)M_H(\s_3)M_H(\s_4)  \nonumber \\ 
& \gtrsim_\kappa p_n^{|V(H)| + 3L}N_{L}(H,G_n). 
\end{align*} 
Now, noting that $\Ex[T(H,G_n)]= \frac{p_n^{|V(H)|}}{|Aut(H)|} \sum_{\bm s \in V(G_n)_{|V(H)|}} M_H(\bm s)$ and an application of the H\"{o}lder's inequality gives, 
\begin{align*}
\Ex[W_n] (\Ex[T(H,G_n)])^2 & \gtrsim_{H,\kappa} p_n^{3|V(H)|+3L}N_{L}(H,G_n)\left[\sum_{\s\in V(G_n)_{|V(H)|}} M_H(\s)\right]^2\\
& \gtrsim_{H,\kappa}\left[p_n^{|V(H)|+L}\sum_{\substack{\s_1 \in V(G_n)_{|V(H)|}}}M_H(\s_1)\sum_{\substack{\s_2\in V(G_n)_{|V(H)|} \\ |\bar \s_2\setminus \{\bar \s_1\}|=L}}M_H(\s_2)\right]^3
\end{align*}
Hence, for each $L \in [1, |V(H)|-1]$, we have 
$$\sqrt[3]{\Ex[W_n] (\Ex[T(H,G_n)])^2}\gtrsim_{H,\kappa} p_n^{|V(H)|+L} \sum\limits_{\substack{\s_1, \s_2  \in V(G_n)_{|V(H)|} \\ |\bar \s_2\setminus \bar \s_1|=L}}M_H(\s_1)M_H(\s_2).$$ 
Summing over $L \in [1, |V(H)|-1]$ we get 
$$\sqrt[3]{\Ex[W_n] (\Ex[T(H,G_n)])^2} \gtrsim_{H,\kappa} \sigma(H,G_n)^2,$$ from which \eqref{eq:gen_bound} follows. This completes the proof of (a).

Next, note that $\Ex[\hat \sigma(H,G_n)^2]=\sigma(H,G_n)^2$ and 
{\begin{align*}
	\vr[\hat \sigma(H,G_n)^2] & \lesssim \sum_{\s_1, \s_2, \s_3, \s_4 \in \mathcal{K}_{n, 4}}  p_n^{\left|\bar \s_1 \bigcup \bar \s_2 \bigcup \bar \s_3 \bigcup \bar \s_4 \right|} M_H(\s_1)M_H(\s_2)M_H(\s_3)M_H(\s_4) \\ & \lesssim_{\kappa} \Ex[W_n]=o(\sigma(H, G_n)^4)
\end{align*}
by Lemma \ref{abs-bound} (with $r=|V(H)|$)} and the assumption that $\Ex[W_n]=o(\sigma(H, G_n)^4)$. This shows the consistency of $\hat \sigma(H,G_n)^2$ in (b). The proof of (c) is an immediate consequence of part (b) and \eqref{e:wass}.

\section{{Proofs of Results from Section \ref{sec:graphs}}}
\label{sec:graphs_pf}

In this section we will prove the results stated in Section \ref{sec:graphs}. We begin with the following key lemma that will be useful in establishing inconsistency and non-normality of the HT estimates.

\begin{lemma}\label{converse-tool} 
For a fixed connected graph $H$ the following hold:
\begin{enumerate}
\item[(a)] There exists a constant $c = c_H > 0$ (depending on $H$) such that  
$$\Pr(T(H,G_n)=0) \ge e^{-cp_n^{|V(H)|}N(H,G_n)},$$
for all $n$ large enough. 

\item[(b)]
	If $\liminf_{n\rightarrow\infty} \Pr(T(H,G_n)=0)>0$, the estimator $\hat N(H,G_n)$ is neither consistent nor asymptotically normal.
	\end{enumerate}
\end{lemma}

\begin{proof} By the FKG inequality, 
\begin{align*}
\Pr({T}(H,G_n)=0)& =\Pr\left(M_H({\s})X_{\s}=0, \text{ for all } \s\in V(G_n)_{|V(H)|}\right) \\ & \ge \prod_{\s\in V(G_n)_{|V(H)|}}\Pr(X_{\s}=0)^{M_H(\s)} 
\\ & =(1-p_n^{|V(H)|})^{\sum_{\s\in V(G_n)_{|V(H)|}}M_H(\s)}   \ge e^{-cp_n^{|V(H)|}N(H,G_n)}
\end{align*}
where the last inequality uses $\frac{1}{|Aut (H)|}\sum_{\s\in V(G_n)_{|V(H)|}}M_H(\s)=N(H, G_n)$ and $\log (1-p_n^{V(H)}) \gtrsim_H - p_n^{|V(H)|}$ for $n$ large enough (using \eqref{eq:boundp}). This completes the proof of (a)

For (b), note that inconsistency is immediate, as on the set $T(H,G_n)=0$ (which happens with probability bounded away from $0$),  the ratio $\frac{\hat N(H,G_n)}{N(H,G_n)} = \frac{T(H,G_n)}{p_n^{|V(H)|} N(H,G_n)} =0$, which does not converge to $1$. This also implies non-normality because a random variable which takes the value $0$ with probability bounded away from $0$ cannot converge after centering and scaling to a continuous distribution. 
\end{proof}

\subsection{Proof of Proposition \ref{bdd-deg}}
\label{sec:degree_pf}  

To begin with, use \eqref{e:tvar} and the inequality $t_H(A)^2 \geq t_H(A)$ to note that
 \begin{align}\label{eq:sigma_degree} 
 \sigma(H,G_n)^2 &   \gtrsim_H p_n^{|V(H)|}\sum_{A \subseteq V(G_n): |A|=|V(H)|}t_H(A)^2 \nonumber  \\
 & \ge p_n^{|V(H)|} \sum_{A \subseteq V(G_n): |A|=|V(H)|} t_H(A) \nonumber \\
&  \gtrsim_H p_n^{|V(H)|} N(H,G_n),
 \end{align}  
where the last inequality uses the fact that 	
\begin{align*}
\sum_{A \subseteq V(G_n): |A|=|V(H)|} t_H(A)  & = \frac1{|\ah|} \sum_{\substack{A \subseteq V(G_n) \\  |A| = |V(H)|}}  \sum_{\s: \bar{\s}\supseteq A } M_H(\s)  \nonumber \\ 
& =  \frac1{|\ah|} \sum_{\s\in V(G_n)_{|V(H)|}} M_H(\s)\sum_{\substack{A \subseteq \bar{\s} \\  |A| =|V(H)|}} 1 \nonumber \\ 
& =  N(H, G_n) . 
\end{align*}

Proceeding to estimate $\Ex[W_n]$ (recall the definition of $W_n$ from \eqref{W_n:defn}) we have
	\begin{align*}
	\Ex[W_n] & \lesssim \sum_{\{\s_1,\s_2,\s_3,\s_4\}\in \mathcal{K}_{n,4}} p_n^{\left|\bar \s_1\bigcup \bar \s_2 \bigcup \bar \s_3 \bigcup \bar \s_4 \right|}M_H(\s_1)M_H(\s_2)M_H(\s_3)M_H(\s_4) \tag*{(by Lemma \ref{abs-bound})}\\ 
	& \le p_n^{|V(H)|}\sum_{\{\s_1,\s_2,\s_3,\s_4\}\in \mathcal{K}_{n,4}}M_H(\s_1)M_H(\s_2)M_H(\s_3)M_H(\s_4)	. 
	\end{align*}
Now, without loss of generality by permuting the labels $\{1,2,3,4\}$ if necessary, we can assume that $\bar \s_a\bigcap (\bigcup_{b=1}^{a-1} \bar \s_b ) \ne \emptyset$, for each $1 \leq a \leq 4$. Recall that $\Delta=\sup_{n \geq 1} \max_{v\in V(G_n)}d_v=O(1)$. Then, for each $\s_1 \in V(G_n)_{|V(H)|}$ fixed
$$ \sum_{\substack{ \s_2,\s_3,\s_4 \in V(G_n)_{|V(H)|} \\ \{\s_1,\s_2,\s_3,\s_4\}\in \mathcal{K}_{n,4}}} M_H(\s_2)M_H(\s_3)M_H(\s_4) \lesssim_{\Delta, H} 1,$$
since the assumption $\bar \s_2 \bigcap \bar \s_1 \ne \emptyset$ gives at most $\Delta^{|V(H)|}$ choices for $\bm s_2$, and similarly for $\s_3$ and $\s_4$, as well. Hence, 	
\begin{align*}
\Ex[W_n] \lesssim_{H, \Delta} p_n^{|V(H)|}\sum_{\s_1 \in V(G_n)_{V(H)}}M_H(\s_1) \lesssim_{H} p_n^{|V(H)|}N(H,G_n).
\end{align*}
Combining the above with \eqref{eq:sigma_degree}, \eqref{e:wass}, and using the assumption $p_n^{|V(H)|}N(H,G_n) \gg 1$ we get,
	\begin{align*}
\mathrm{Wass}(Z(H, G_n), N(0, 1)) \lesssim_H \sqrt{\frac{\Ex[W_n]}{\sigma(H,G_n)^4}} \lesssim_{\Delta,H} \sqrt{\frac{1}{p_n^{|V(H)|}N(H,G_n)}} \to 0. 
	\end{align*} 
This shows the asymptotic normality of $Z(H, G_n)$ whenever $p_n^{|V(H)|}N(H,G_n) \gg 1$. The consistency of the HT estimator $\hat N(H, G_n)$ also follows from Proposition \ref{ppn:variance_estimation} , completing the proof of (a). 

The result in (b) is an immediate consequence of the assumption $p_n^{|V(H)|}N(H,G_n) = O(1)$ and Lemma \ref{converse-tool}.  \qed

\subsection{Proof of Theorem \ref{thm:er-consistency} }
\label{sec:er_pf}

As in the statement of theorem, suppose $G_n \sim \mathcal{G}(n, q_n)$ be a realization of the Erd\H os-R\'enyi random graph. With $W_n$ as defined in \eqref{W_n:defn}, using Theorem \ref{thm:wass} and Proposition \ref{ppn:variance_estimation} , to prove (a) it suffices to show that $\sigma(H,G_n)^{-4}\Ex[W_n|G_n]=O_P((np_n q_n^{m(H)})^{-1} )$. To this effect, using Lemma \ref{abs-bound} gives
	\begin{align}\label{eq:random_graph_Gn}
	\Ex[W_n|G_n] \lesssim_H \sum_{\{\s_1,\s_2,\s_3,\s_4\}\in \mathcal{K}_{n,4}}  M_H(\s_1)M_H(\s_2)M_H(\s_3)M_H(\s_4)p_n^{\left|\bigcup_{r=1}^4\bar{\bm s}_{r} \right|}.
	\end{align} 
	Taking expectation over the randomness of the Erd\H{o}s-R\'enyi random graph gives 
	\begin{align}\label{eq:er-1}
	\Ex[M_H(\s_1)M_H(\s_2)M_H(\s_3)M_H(\s_4)] = q_n^{|E(G_n\left([\bigcup_{r=1}^4\bar{\bm s}_{r}]\right)|},
	\end{align}
	where $G_n([\bigcup_{r=1}^4\bar{\bm s}_{r}])$ is the subgraph of $G_n$ induced by the vertices in $\bigcup_{r=1}^4\bar{\bm s}_{r}$. This gives
	\begin{align}\label{annealed}
	\Ex[W_n] & =\Ex[\Ex[W_n|G_n]] \nonumber \\ 
	& \lesssim_H  \sum_{H_1, H_2, H_3, H_4 \in \mathcal H_n} p_n^{|V(\bigcup_{a=1}^4 H_a)|} q_n^{|E(\bigcup_{a=1}^4 H_a)|} \bm 1 \left\{\bigcup_{a=1}^4 H_a \text{ is connected} \right \} ,
	\end{align} 
where $\mathcal{H}_n$ denotes the collection of all labelled sub-graphs of $K_{n}$ which are isomorphic to $H$.\footnote{For any two simple graphs $F_1=(V(F_1), E(F_1))$ and $F_2= (V(F_2), E(F_2))$, define $F_1 \bigcup F_2 = (V(F_1) \bigcup V(F_2), E(F_1) \bigcup E(F_2))$ and  $F_1 \bigcap F_2 = (V(F_1) \bigcap V(F_2), E(F_1) \bigcap E(F_2))$.} 
For $r \geq 1$ fixed, define 
\begin{align}\label{eq:Nr}
\mathcal N_n(r):=\sum_{H_1, \ldots, H_r \in \mathcal{H}_n }p_n^{|V(\bigcup_{a=1}^r H_a)|}q_n^{|E(\bigcup_{a=1}^r H_a)|} \bm 1 \left\{\bigcup_{a=1}^r H_a \text{ is connected} \right \}. 
\end{align}
The following result gives an estimate of $\mathcal N_n(r)$. 

\begin{lemma}\label{lm:graph_B} For every integer $r \geq 2$ we have,    
\begin{align}\label{eq:er_normality}
\mathcal N_n(r) \lesssim_{H, r} \mathcal N_n(r-1) \sqrt{\mathcal N_n(2)}\left(np_n q_n^{m(H)}\right)^{-\frac12},
\end{align}
whenever $np_n q_n^{m(H)} \geq 1  $. 
\end{lemma}	 
	 
The proof of this lemma is given below in Appendix \ref{sec:pf_N_count}. We first use it to complete the proof of Theorem \ref{thm:er-consistency} (a). For this, using \eqref{eq:er_normality} twice gives
	\begin{align}\label{lim}
	\notag \mathcal N_n(4) & \lesssim_H \mathcal N_n(3)\sqrt{\mathcal N_n(2)} \left(np_n q_n^{m(H)} \right)^{-\frac{1}{2}} \nonumber \\ 
	& \lesssim_H \left[\mathcal N_n(2)^{\frac{3}{2}} \left(np_n q_n^{m(H)}\right)^{-\frac{1}{2}}\right]\left[\sqrt{\mathcal N_n(2)}\left(np_n q_n^{m(H)}\right)^{-\frac{1}{2}}\right] \nonumber \\
	& = \mathcal N_n(2)^2 \left(np_n q_n^{m(H)}\right)^{-1}.
	\end{align} 
Also, using Lemma \ref{free-bound} gives, 
	\begin{align}\label{eq:HGn_12}
	\sigma(H,G_n)^2 & = \sum_{\substack{\s_1,\s_2\in V(G_n)_{|V(H)|} \\ \bar \s_1\bigcap \bar \s_2\neq\emptyset}} M_H(\s_1)M_H(\s_2)\Ex[Y_{\s_1} Y_{\s_2}]  \\ 
	 & \gtrsim_H \sum_{\substack{\s_1,\s_2\in V(G_n)_{|V(H)|} \\ \bar \s_1\bigcap \bar \s_2\neq\emptyset}} M_H(\s_1)M_H(\s_2)p^{|\bar{\s}_1 \bigcup \bar{\s}_2|} \nonumber ,
	\end{align}
	and so taking an expectation over the randomness of the Erd\H{o}s-R\'enyi graph gives
		\begin{align}\label{expec}
	\Ex[\sigma(H,G_n)^2] \gtrsim_H  \sum_{H_1, H_2 \in \mathcal{H}_n }p_n^{|V(H_1 \bigcup H_2)|}q_n^{|E(H_1 \bigcup H_2)|} \bm 1 \left\{ H_1 \bigcup H_2 \text{ is connected} \right \}  = \mathcal N_n(2).
	\end{align}
Moreover, a direct expansion gives 
	\begin{align}\label{eq:er-2}
\notag & \operatorname{Var} [\sigma(H,G_n)^2] \\  
\notag & = \sum_{\substack{\s_1,\s_2\in V(G_n)_{|V(H)|} \\ \bar \s_1\bigcap \bar \s_2\neq\emptyset}}\sum_{\substack{\s_3,\s_4\in V(G_n)_{|V(H)|} \\ \bar \s_3\bigcap \bar \s_4\neq\emptyset}} \Ex[Y_{\s_1} Y_{\s_2}] \Ex[Y_{\s_3} Y_{\s_4}] \operatorname{Cov}[M_H(\s_1)M_H(\s_2),M_H(\s_3)M_H(\s_4) ] \\
	\notag & \lesssim_H \underbrace{\sum_{\substack{\s_1,\s_2\in V(G_n)_{|V(H)|} \\ \bar \s_1\bigcap \bar \s_2\neq\emptyset}} \sum_{\substack{\s_3,\s_4\in V(G_n)_{|V(H)|} \\ \bar \s_3\bigcap \bar \s_4\neq\emptyset}}}_{\{\bar \s_1 \bigcup \bar \s_2\}\bigcap \{ \bar \s_3 \bigcup \bar \s_4\}\neq \emptyset} p_n^{|\s_1\bigcup \s_2\bigcup \s_3\bigcup \s_4|}\Ex[M_H(\s_1)M_H(\s_2)M_H(\s_3)M_H(\s_4)] \\
 & \lesssim_H  \sum_{H_1, H_2, H_3, H_4 \in \mathcal H_n} p_n^{|V(\bigcup_{a=1}^4 H_a)|} q_n^{|E(\bigcup_{a=1}^4 H_a)|} \bm 1 \left\{\bigcup_{a=1}^4 H_a \text{ is connected} \right \} =\mathcal N_n(4), 
	\end{align} 
where the inequality in the third line uses Lemma \ref{free-bound} to get 
$$\Ex[Y_{\s_1}Y_{\s_2}]\Ex[Y_{\s_3} Y_{\s_4}] \lesssim_H \Ex[Z_{\s_1}Z_{\s_2}] \Ex[Z_{\s_3}Z_{\s_4}] \lesssim p_n^{|\bar \s_1\bigcup \bar \s_2|+|\bar \s_3 \bigcup \bar \s_4|} \le p^{|\bigcup_{a=1}^4 \bar \s_a|}.$$  Thus, using \eqref{lim}, \eqref{expec}, and \eqref{eq:er-2}, we have
	\begin{align}\label{eq:sigma_con}
	\frac{\operatorname{Var}[\sigma(H,G_n)^2]}{(\Ex[\sigma(H,G_n)^2])^2} \lesssim_H \frac{\mathcal N_n(4)}{\mathcal N_n(2)^2} \to 0, 
	\end{align}
which implies $\frac{\sigma(H,G_n)^2}{\Ex\sigma(H, G_n)^2}\stackrel{P}{\rightarrow}1$. Combining the estimates in \eqref{annealed}, \eqref{expec} and \eqref{eq:sigma_con} we 
	\[ \frac{\Ex[W_n|G_n]}{\sigma (H,G_n)^4}=O_P\left(\frac{\Ex[W_n|G_n] }{(\Ex[\sigma(H,G_n)^2])^2}\right)=O_P\left(\frac{\mathcal N_n(4)}{\mathcal N_n(2)^2}\right)=O_P\left((np_n q_n^{m(H)})^{-1}\right)\]
where the last bound uses \eqref{lim}. This completes the proof of Theorem \ref{thm:er-consistency} (a). 

Next, we prove (b). For this, let $H_1$ be the subgraph for which $m(H)=\frac{|E(H_1)|}{|V(H_1)|}$. Then by Lemma \ref{converse-tool} (a), 
\begin{align}\label{eq:H1}   
\Pr(T(H,G_n)=0|G_n)\ge \Pr(T(H_1,G_n)=0|G_n)\ge e^{-cp_n^{|V(H_1)|}N(H_1,G_n)}. 
\end{align}
Therefore, by Lemma \ref{converse-tool} (b) it suffices to show that $p_n^{|V(H_1)|} N(H_1,G_n)=O_P(1)$. This follows on noting that 
\begin{align*}
 p_n^{|V(H_1)|} \Ex[N(H_1,G_n)] \le p_n^{|V(H_1)|} n^{|V(H_1)|} q_n^{|E(H_1)|}=\Big(np_n q_n^{m(H)}\Big)^{|V(H_1)|} = O(1). 
\end{align*}

\subsubsection{Proof of Lemma \ref{lm:graph_B}}
\label{sec:pf_N_count}

Note that any collection $H_1, \ldots, H_r \in \mathcal{H}_n$ with $\bigcup_{a=1}^r H_a$ connected, can be ordered in such a way that the labeled graph $\Gamma_b:=\bigcup_{a=1}^b H_a$ is connected for $1\le b \le r$. Now, setting $F:=\Gamma_{r-1}\cap H_r$ we have
\[ |V(\Gamma_r)|=|V(\Gamma_{r-1})|+|V(H)|-|V(F)| \text{ and } 	|E(\Gamma_r)|= |E(\Gamma_{r-1})|+|E(H)|-|E(F)| .\] 
This gives the bound
		\begin{align}\label{recursion_I}
\mathcal N_n(r)& =\sum_{H_1, \ldots, H_r \in \mathcal{H}_n} p_n^{|V(\Gamma_r)|} q_n^{|E(\Gamma_r)|} \bm 1\{ \Gamma_r \text{ is connected}\} \nonumber \\
		& \lesssim_r \sum_{H_1, \ldots, H_r \in \mathcal{H}_n} p_n^{|V(\Gamma_r)|} q^{|E(\Gamma_r)|} 1\{\Gamma_b \text{ is connected for all } 1\le b \le r  \}   
		\nonumber \\
&= \sum_{H_1, \ldots, H_{r-1} \in \mathcal{H}_n}  p_n^{|V(\Gamma_{r-1})|} q_n^{|E(\Gamma_{r-1})|}   1\{\Gamma_b \text{ is connected for all } 1\le b \le r - 1  \}   Q_n(r),  
\end{align}
where 
\begin{align*}
Q_n(r) & := \sum_{F\subseteq H: F \ne \emptyset}   \sum_{\substack{H_r\in \mathcal{H}_n \\ \Gamma_{r-1}  \cap H_r  \simeq F}} p_n^{|V(H)|-|V(F)|}q_n^{|E(H)|-|E(F)|} \nonumber \\ 
& \leq  \sum_{F\subseteq H: F \ne \emptyset}  (np_n)^{|V(H)|-|V(F)|}q_n^{|E(H)|-|E(F)|}  \nonumber \\  
& \lesssim_H  \max_{F \subseteq H: F \ne \emptyset} \left\{ (np_n)^{|V(H)|-|V(F)|}q_n^{|E(H)|-|E(F)|} \right \} . 
\end{align*}
Using this inequality in \eqref{recursion_I} gives, 
		\begin{align}\label{recursion}
\mathcal N_n(r) & \lesssim_{H, r}  \max_{F \subseteq H: F \ne \emptyset} \left\{ (np_n)^{|V(H)|-|V(F)|}q_n^{|E(H)|-|E(F)|} \right \}  \mathcal N_n(r-1).
\end{align} 
		
		Also, for $H_1,H_2\in \mathcal{H}_n$, $|V(H_1 \bigcup H_2)|=2|V(H)|-|V(H_1 \bigcap H_2)|$ and $|E(H_1\bigcup H_2)|=2|E(H)|-|E(H_1\bigcap H_2)|$. Hence, 
		\begin{align*}
		\mathcal N_n(2) & = \sum_{H_1, H_2  \in \mathcal{H}_n} p_n^{2|V(H)|-|V(H_1\cap H_2)|} q_n^{2|E(H)|-|E(H_1\cap H_2)|} 1\left\{H_1 \bigcup H_2 \text{ is connected} \right \} \\ 
		& = \sum_{F \subseteq H: F \ne \emptyset} \sum_{\substack{H_1,H_2 \in \mathcal{H}_2 \\ H_1 \bigcap H_2 \simeq F}} p_n^{2|V(H)|-|V(F)|}q_n^{2|E(H)|-|E(F)|}
		\end{align*}
Now, since for inner sum there are $\gtrsim_H  n^{2|V(H)|-|V(F)|}$ choices for the vertices, we get
		\begin{align*}
	\mathcal N_n(2) \gtrsim_H  \max_{F \subseteq H: F \ne \emptyset}(n p_n)^{2|V(H)|-|V(F)|} q_n^{2|E(H)|-|E(F)|} . 
		\end{align*} 
Using this inequality on the RHS~of \eqref{recursion} we get that 
		\begin{align*}
		\mathcal N_n(r)& \lesssim_{H, r} \frac{ \mathcal N_n(r-1) \max_{F \subseteq H: F \ne \emptyset} \left\{ (np_n)^{|V(H)|-|V(F)|} q_n^{|E(H)|-|E(F)|} \right \}  \sqrt{\mathcal N_n(2)}}{\max\limits_{F \subseteq H: F \ne \emptyset}\left\{(np_n)^{|V(H)|- \frac{|V(F)|}{2}} q_n^{|E(H)|- \frac{|E(F)|}{2}}\right\}}  \\
		& \leq \mathcal N_n(r-1)\sqrt{\mathcal N_n(2)} \max_{F \subseteq H: F \ne \emptyset} \left\{ \frac{(np_n)^{|V(H)|-|V(F)|} q_n^{|E(H)|-|E(F)|}}{(np_n)^{|V(H)|- \frac{|V(F)|}{2}} q_n^{|E(H)|- \frac{|E(F)|}{2}}}\right\} \\
		& = \mathcal N_n(r-1)\sqrt{\mathcal N_n(2)} \max_{F \subseteq H: F \ne \emptyset} \left\{(np_n)^{-\frac{|V(F)|}{2}}q_n^{-\frac{|E(F)|}{2}} \right\} \\ & = \mathcal N_n(r-1)\sqrt{\mathcal N_n(2)}\max_{F \subseteq H: F \ne \emptyset} \Big\{np_n q_n^{\frac{|E(F)|}{|V(F)|}}\Big\}^{-\frac{|V(F)|}2} \\ 
		& \le \mathcal N_n(r-1)\sqrt{\mathcal N_n(2)}\left(np_n q_n^{m(H)}\right)^{-\frac12}, 
		\end{align*}
		where the last step uses $np_n q_n^{m(H)} \geq 1  $. This completes the proof of Lemma \ref{lm:graph_B}.

\subsection{Annealed CLT and Proof of Corollary \ref{cor:ann-er-consistency}}
\label{sec:annealed}

 In this section we discuss general conditions for obtaining annealed central limit theorems of $\hat N(H, G_n)$ in random graph models. We then use this result to prove Corollary \ref{cor:ann-er-consistency}. We begin by recalling the definitions of the rescaled statistics $\mathcal A(H, G_n)$, $Z(H, G_n)$, and $\mathcal E(H, G_n)$ from \eqref{eq:def1} and \eqref{eq:ZE}, respectively.

\begin{lemma}\label{lem:aez}  Let $\{G_n\}_{n \geq 1}$ be a sequence of random graphs such that the following hold: 
	\begin{itemize}
	\item[(a)] Conditional on the graph sequence $\{G_n\}_{n \geq 1}$, $Z(H,G_n) \stackrel{D} \rightarrow N(0,1)$, 
	\item[(b)] $\mathcal E(H, G_n) \stackrel{D} \rightarrow N(0,1)$, and
	\item[(c)]  $\frac{\vr_{G_n}[\hat{N}(H,G_n)]}{\Ex[\vr_{G_n}[\hat{N}(H,G_n)]]} \stackrel{P} \rightarrow 1$. 
	\end{itemize}
Then  $\mathcal A(H, G_n) \stackrel{D} \rightarrow N(0,1)$.
\end{lemma} 

\begin{proof} Define
	\begin{align}\label{def3}
		X(H,G_n):=\frac{\hat{N}(H,G_n)- N(H,G_n)}{\sqrt{\Ex[\vr_{G_n}[\hat{N}(H,G_n)]]}}.
	\end{align}
	Combining assumptions (a) and (c) of Lemma \ref{lem:aez}, we have, for any $M>0$, as $n\to \infty$,  	\begin{align}\label{comp}
		\sup_{t\in [-M,M]} \Ex\left|\Ex\left[e^{it X(H,G_n)}\big| G_n\right]-e^{-\frac{t^2}{2}}\right| \to 0.
	\end{align}
	Note that $\vr[\hat{N}(H,G_n)]=\Ex[\vr_{G_n}[\hat{N}(H,G_n)]]+\vr[N(H,G_n)]$. Thus setting
	\begin{align*}
		\alpha_n:=\frac{\Ex[\vr_{G_n}[\hat{N}(H,G_n)]]}{\vr[\hat{N}(H,G_n)]} \in [0,1], \mbox{ we have } 1-\alpha_n:=\frac{\vr[N(H,G_n)]}{\vr[\hat{N}(H,G_n)]}.
	\end{align*} 
	Using the above definition and recalling \eqref{eq:def1}, \eqref{eq:ZE}, and \eqref{def3}  we can write 
	$$\mathcal A(H, G_n)=\sqrt{\alpha_n}X(H,G_n)+\sqrt{1-\alpha_n}\mathcal E(H, G_n).$$ 
Now, let $Z \sim N(0,1)$ independent of $\{\mathcal E(H, G_n)\}_{n\ge 1}$. Fix $t\in \R$ and note that
\begin{align}\label{eq:def1_T12}
	\left|\Ex \left[e^{it\mathcal A(H, G_n)}\right]-e^{-\frac{t^2}{2}}\right| \le T_1 + T_2 ,
\end{align}
where 	
\begin{align*}
T_1& :=\left|\Ex \left[e^{it\sqrt{1-\alpha_n}\mathcal E (H, G_n)} \left\{e^{it\sqrt{\alpha_n}X(H,G_n)}-e^{it\sqrt{\alpha_n} Z}\right\} \right]\right|, \nonumber \\
T_2 & :=\left|\Ex \left[e^{it\sqrt{\alpha_n} Z+it\sqrt{1-\alpha_n}\mathcal E(H, G_n)} -e^{-\frac{t^2}{2}}\right]\right|
\end{align*}
Note that 
$$T_1 \le \Ex\left|\left[ e^{it\sqrt{\alpha_n}X(H,G_n)}\big| G_n\right]-e^{-\frac{\alpha_n t^2}{2} }\right| \rightarrow 0, $$
by \eqref{comp}. Also, 
$$T_2  \leq e^{-\frac{\alpha_n t^2}{2} } \left|\Ex \left[e^{it\sqrt{1-\alpha_n}\mathcal E(H, G_n)} \right]- e^{-\frac{(1-\alpha_n) t^2}{2} } \right| \rightarrow 0,$$ 
by assumption (b) of Lemma \ref{lem:aez}. Hence, by \eqref{eq:def1_T12}, $\Ex[e^{it\mathcal A(H, G_n)}]\to e^{-\frac{t^2}{2}}$, that is, $\mathcal A(H, G_n) \stackrel{D} \rightarrow N(0,1)$. 

\end{proof}


\subsubsection{Proof of Corollary \ref{cor:ann-er-consistency}}

To begin with suppose, $np_nq_n^{m(H)}\to \infty$. By Theorem \ref{thm:er-consistency}~(a), $\hat{N}(H,G_n)$ is consistent for $N(H,G_n)$ conditionally. Hence, to show $\hat{N}(H,G_n)$ is consistent for $\Ex[N(H,G_n)]$ unconditionally, it suffices to show that $N(H,G_n)$ is consistent for $\Ex[N(H,G_n)]$. Towards this by \cite[Lemma 3.5 and Lemma 3.6]{janson2011book} we have,  
\begin{align}
\frac{\vr[N(H, G_n)]}{\Ex[N(H, G_n)]^2} \lesssim \sum_{H_1 \subseteq H: |E(H_1)| > 0} \frac{1}{n^{|V(H_1)|} q_n^{|E(H_1)|} } \rightarrow 0, 
\end{align} 
since $nq_n^{m(H)} \geq np_nq_n^{m(H)}\to \infty$, by assumption. 
This shows $N(H,G_n)$ is consistent for $\Ex[N(H,G_n)]$ using Chebyshev's inequality, and hence, $\hat{N}(H,G_n)$ is consistent for $\Ex[N(H,G_n)]$. For the asymptotic normality note that in this case, $Z(H,G_n) \stackrel{D}{\to} N(0, 1)$ given the graph sequence $\{G_n\}_{n \geq 1}$ by Theorem \ref{thm:er-consistency}~(a). Next, note that $np_nq_n^{m(H)}\to \infty$ implies $nq_n^{m(H)}\to \infty$ and hence, by \cite[Theorem 2]{ruz} it follows that $\mathcal E(H,G_n) \stackrel{D}{\to} N(0,1)$. Moreover, 
\eqref{eq:sigma_con} gives, $$\frac{\vr_{G_n}[\hat{N}(H,G_n)]}{\Ex[\vr_{G_n}[\hat{N}(H,G_n)]]} \stackrel{P} \rightarrow 1.$$ Hence, by Lemma \ref{lem:aez} the result in Corollary \ref{cor:ann-er-consistency} (a) follows.

Next, suppose $np_nq_n^{m(H)}=O(1)$. Let $H_1$ be the subgraph for which $m(H)=\frac{|E(H_1)|}{|V(H_1)|}$. Then by Lemma \ref{converse-tool} (a) and Jensen's inequality, 
\begin{align*}  
	\Pr(T(H,G_n)=0)\ge \Pr(T(H_1,G_n)=0)\ge \Ex[e^{-cp_n^{|V(H_1)|}N(H_1,G_n)}] \ge \exp\left(-cp_n^{|V(H_1)}\Ex[N(H_1,G_n)]\right). 
\end{align*}
Therefore, by Lemma \ref{converse-tool} (b) it suffices to show that $p_n^{|V(H_1)|} \Ex[N(H_1,G_n)]=O(1)$. This follows on noting that 
\begin{align*}
	p_n^{|V(H_1)|} \Ex[N(H_1,G_n)] \le p_n^{|V(H_1)|} n^{|V(H_1)|} q_n^{|E(H_1)|}=\Big(np_n q_n^{m(H)}\Big)^{|V(H_1)|} = O(1). 
\end{align*}
This completes the proof of Corollary \ref{cor:ann-er-consistency} (b).

\subsection{Proof of Corollary \ref{regulargraph}} 
\label{sec:randomregular_pf}

Throughout this proof we will assume that $G_n \sim \mathcal G_{n, d}$ is the random graph $d$-regular graph on $n$ vertices, where $1 \leq d \leq n-1$. Also, recall that $q_n= d/n$.  \\

\noindent {\it Case } 1: We begin with case $1 \ll d \ll n$. Then using  \cite[lemma 2.1]{KIM20071961}, for any $\s_1, \s_2, \s_3, \s_4 \in V(G_n)_{|V(H)|}$, 
$$\Ex[M_H(\s_1)M_H(\s_2)M_H(\s_3)M_H(\s_4)]= (1+o(1))q_n^{|E(\mathcal{G}(\s_1, \s_2, \s_3, \s_4))|}.$$ 
Hence, by  \eqref{eq:random_graph_Gn}, 
\begin{align}\label{annealed_II}
	\Ex[W_n] & =\Ex[\Ex[W_n|G_n]] \nonumber \\ 
	& \lesssim_H  (1+o(1)) \sum_{H_1, H_2, H_3, H_4 \in \mathcal H_n} p_n^{|V(\bigcup_{a=1}^4 H_a)|} q_n^{|E(\bigcup_{a=1}^4 H_a)|} \bm 1 \left\{\bigcup_{a=1}^4 H_a \text{ is connected} \right \} \nonumber \\ 
 & = (1+o(1)) \mathcal N_n(4),
	\end{align} 
where $\mathcal N_n(r)$ is as defined in \eqref{eq:Nr}. Similarly, by \cite[Lemma 2.1]{KIM20071961}, the estimate in \eqref{expec} continue to hold with an extra $1+o(1)$: 
\begin{align}\label{expec_II} 
\Ex[\sigma(H,G_n)^2] \gtrsim_H  (1+o(1)) \mathcal N_n(2).
	\end{align}
Next, consider tuples $\s_1,\s_2,\s_3,\s_4 \in V(G_n)_{|V(H)|}$ such that $(\bar \s_1 \bigcup \bar \s_2)\bigcap (\bar \s_3\bigcup \bar \s_4)=\emptyset$. Then by  \cite[Lemma 2.1]{KIM20071961}, 
\begin{align}\label{eq:12_34}
\operatorname{Cov}\Big[M_H(\s_1)M_H(\s_2),M_H(\s_3)M_H(\s_4)\Big] 
& =o(1) \Ex[M_H(\s_1)M_H(\s_2)]\Ex[M_H(\s_1)M_H(\s_2)]. 
\end{align}
This implies, 
\begin{align}\label{eq:er-2}
& \underbrace{\sum_{\substack{\s_1,\s_2\in V(G_n)_{|V(H)|} \\ \bar \s_1\bigcap \bar \s_2\neq\emptyset}} \sum_{\substack{\s_3,\s_4\in V(G_n)_{|V(H)|} \\ \bar \s_3\bigcap \bar \s_4\neq\emptyset}}}_{\{\bar \s_1 \bigcup \bar \s_2\}\bigcap \{ \bar \s_3 \bigcup \bar \s_4\} = \emptyset} \Ex[Y_{\s_1} Y_{\s_2}] \Ex[Y_{\s_3} Y_{\s_4}] \mathrm{Cov}[M_H(\s_1)M_H(\s_2), M_H(\s_3)M_H(\s_4)] \nonumber \\ 
& = o((\Ex[\sigma(H, G_n)^2])^2),
\end{align}
where the last step uses \eqref{eq:12_34} and \eqref{eq:HGn_12}. Combining this with \eqref{eq:er-2} (which holds with $(1+o(1))$ factor), gives $\mathrm{Var}[\sigma(H, G_n)^2] = (1+o(1)) \mathcal N_n(4) +  o((\Ex[\sigma(H, G_n)^2])^2)$. Using this with \eqref{expec_II} and \eqref{lim} we have $\operatorname{Var}[\sigma(H,G_n)^2] =o((\Ex[\sigma(H,G_n)^2])^2)$, which implies $\frac{\sigma(H,G_n)^2}{\Ex\sigma(H, G_n)^2}\stackrel{P}{\rightarrow}1$. Now, combining the estimates in \eqref{annealed_II} and \eqref{expec_II} gives  
	\[ \frac{\Ex[W_n|G_n]}{\sigma (H,G_n)^4}=O_P\left(\frac{\Ex[W_n|G_n] }{(\Ex[\sigma(H,G_n)^2])^2}\right)=O_p\left(\frac{\mathcal N_n(4)}{\mathcal N_n(2)^2}\right)=O_p\left((np_n q_n^{m(H)})^{-1}\right)\]
where the last bound uses \eqref{lim}. This completes the proof of consistency and asymptotic normality above the threshold when $1 \ll d \ll n$.

For below the threshold, that is, $np_nq_n^{m(H)}=O(1)$, using Lemma \ref{converse-tool} (a) as in the proof of Theorem \ref{thm:er-consistency} it suffices to show that  $p_n^{|V(H_1)|} \Ex[N(H_1,G_n)]=O(1)$, where $H_1$ is the subgraph of $H$ such that $m(H) = \frac{|E(H_1)|}{|V(H_1)|}$. For this using \cite[Corollary 2.2]{KIM20071961} gives $$\Ex[N(H_1,G_n)] = (1+o(1)) n^{|V(H_1)|} q_n^{|E(H_1)|}.$$ This implies, 
\begin{align*}
 p_n^{|V(H_1)|} \Ex[N(H_1,G_n)] \le (1+o(1)) p_n^{|V(H_1)|} n^{|V(H_1)|} q_n^{|E(H_1)|}=\Big(np_n q_n^{m(H)}\Big)^{|V(H_1)|} = O(1). 
\end{align*}
This completes the proof of Corollary \ref{regulargraph} when $1 \ll d \ll n$. \\ 

\noindent {\it Case } 2: Next, suppose $d=\Theta(n)$. In this case, the second largest eigenvalue (in absolute value) of $G_n$ is almost surely $O(d^{3/4})$ \cite{krivelevich2001random}, hence, the graph $G_n$ has strong pseudo-random properties. In particular, it follows from \cite[Theorem 4.10]{krivelevich2006pseudo} that $n^{|V(F)|} q_n^{|E(F)|} \lesssim_H \mathbb E[N(F, G_n)] \lesssim_H n^{|V(F)|} q_n^{|E(F)|}$, for any fixed graph $F = (V(F), E(F))$. The result in Corollary \ref{regulargraph} for $d=\Theta(n)$ then follows by arguments similar to {\it Case } 1 above. Therefore, {\it Case } 1 and {\it Case } 2 combined completes the proof of Corollary \ref{regulargraph} for $d \gg 1$. \\

\noindent {\it Case } 3: Finally, consider the case $d= O(1)$. In this case, the graph sequence $G_n$ has bounded maximum degree and the result in Corollary \ref{regulargraph} (b) follows from Proposition \ref{bdd-deg} and the Lemma \ref{lm:regular_NHGn} below.

\begin{lemma}\label{lm:regular_NHGn} Fix $d \geq 2$. Suppose $G_n$ is a uniform random sample from $\mathscr{G}_{n,d}$ and $H = (V (H), E(H))$ is a fixed connected graph with maximum degree $\Delta(H) \leq d$. 

\begin{itemize}
\item[(a)] If $H$ is a tree, that is, $|E(H)| = |V(H)|-1$, then $N(H,G_n)= \Theta_P(n)$.

\item[(b)] If $|E(H)| \geq |V(H)|$, then $N(H,G_n)=O_P(1)$.

\end{itemize}
\end{lemma}

\begin{proof} (a) Let $\mathcal{T}_{m,d}$ be the collection of all trees with $m$ vertices having maximum degree $d$. For the proof of (a) we will induct on $m$. If $m=2$, then $H=K_2$ is just the edge, and $N(K_2, G_n)=\frac{nd}2=\Theta(n)$. Now, fix $m \ge 3$. Suppose the claim is true for all trees $F\in \mathcal{T}_{m-1,d}$. Fix a tree $H\in \mathcal{T}_{m,d}$. Consider the graph $F\in\mathcal{T}_{m-1,d}$ obtained by removing any leaf $v$ in $H$. The degree of the vertex $v$ in $F$ is at most $d-1$. For $3\le r\le m$, let $X_{r,n}$ be the number of $r$-cycles in $G_{n}$, and let $V_{m,n}$ be the set of all vertices which passes through a cycle of length $r$, for some $r\in [3,m]$. Since a cycle of length $r$ has exactly $r$ vertices, we have the trivial inequality $|V_{m,n}|\le \sum_{r=3}^mrX_{r,n}$. Also, it follows from
 \cite[Theorem 2]{bollobas1980} that
		 $$ \sum_{r=3}^m rX_{r, n} \stackrel{D}{\to} \sum_{r=3}^m r\cdot\operatorname{Pois}\left(\frac{(d-1)^r}{2r}\right),$$	
		 where the Poisson random variables are independent. In particular this implies $|V_{m,n}|=O_P(1)$. By induction hypothesis we have $N(F,G_{n})=\Theta_P(n)$, 
		 and so $N(F,G_{n}[V(G_n)\setminus V_{m,n}])=\Theta_P(n)$, where $G_n[V(G_n)\setminus V_{m,n}]$ denotes the induced subgraph of $G_n$ over the vertices in $V(G_n)\setminus V_{m,n}$. Now, consider a copy of $F$ having only vertices in $V(G_n)\setminus V_{m, n}$. Since $v$ does not pass through a cycle and the degree of $v$ in $F$ is at most $d-1$, $v$ must be connected to at least one new vertex (may belong to $V_{m, n}$ as well) which is not in that copy of $F$. This produces a copy of $H$ in $G_n$. Note that given a copy of $H$ there are only finitely many copies $F$ which are subgraphs of $H$. This implies, 
		 $$N(H,G_n) \gtrsim N(F,G_n[V(G_n)\setminus V_{m ,n}])= \Theta_P(n).$$ Since the upper bound $N(H,G_n)=O(n)$ holds trivially for any bounded degree graph,   we have $N(H,G_n)=\Theta_P(n)$, and so the proof of (a) is complete via induction. 
	
(b) Recall that when $|E(H)| > |V(H)|$, then $N(H,G_n)$ is zero asymptotically almost surely \cite[Lemma 2.7]{wormald1999models}.  Next, suppose $|E(H)| =  |V(H)|$ and $H$ is a cycle, then  $N(H,G_n)=O_P(1)$ (by \cite[Theorem 2.5]{wormald1999models}). Finally, suppose $|E(H)| =  |V(H)|$, but $H$ is not a cycle. In this case $H$ is unicyclic, that is, it has exactly one cycle $C_s$ for some $s \geq 3$. Then, $N(H, G_n) \leq N(C_s, G_n) d^{|V(H)|-s} = O_P(1)$, since $N(C_s, G_n) =O_P(1)$ (by \cite[Theorem 2.5]{wormald1999models}) and $d=O(1)$.  
\end{proof}

\subsection{Proof of Proposition \ref{dense}} 
\label{sec:dense_pf}

To begin with, use  \eqref{e:tvar} to get
	\begin{align} 
	\sigma(H,G_n)^2 & \gtrsim p_n^{2|V(H)|-1} \sum_{v \in V(G_n)}  t_H(\{v\})^2 \nonumber \\ 
	& \ge \frac{p_n^{2|V(H)|-1}}{|V(G_n)|} \left[\sum_{v\in V(G_n)} t_H(\{v\}) \right]^2 \tag*{(by the Cauchy-Schwarz inequality)}\nonumber \\ 
 & =p_n^{2|V(H)|-1} |V(G_n)|^{2|V(H)|-1}, 
 \label{eq:sigma_HGn}	
	\end{align}
	where the last step uses the fact that 
	$$\sum_{v\in V(G_n)}t_H(\{v\}) \gtrsim_H N(H,G_n) =\Theta_H(|V(G_n)|^{|V(H)|}).$$ 
	The first inequality above uses \eqref{th-iden} and the second equality is by the assumption $t(H, W)> 0$, which implies, $N(H, G_n) = \Theta(|V(G_n)|^{|V(H)|})$. For controlling $W_n$, note that for a tuple $\{\s_1, \s_2, \s_3, \s_4\}$ to be connected, the graph $\mathcal{G}(\s_1, \s_2, \s_3, \s_4)$ (recall Definition \ref{eq:connected}) can have at most $2(2|V(H)-1)-1=4|V(H)|-3$ vertices. This implies, 
	\begin{align*}
	\Ex[W_n]\lesssim_H \left(|V(G_n)| p_n\right)^{4|V(H)|-3}.
	\end{align*} 
Combining this with \eqref{e:wass} and \eqref{eq:sigma_HGn} gives 
$$\mathrm{Wass}(Z(H, G_n), N(0, 1)) \lesssim_H \sqrt{\frac{\Ex[W_n]}{\sigma(H,G_n)^{4}}}  \lesssim_H (|V(G_n)|p_n)^{-\frac{1}{2}} \to 0,$$ whenever $|V(G_n)|p_n \gg 1$, proving (a).

For (b), note that $N(H,G_n) p_n^{|V(H)|}\lesssim_H  (np_n)^{|V(H)|} = O(1)$, by assumption. Hence, the result in (b) is an immediate consequence of Lemma \ref{converse-tool}.

\section{Proof of Theorem \ref{thm:normality}} \label{sec:normal_ZHGn_II}

In Appendix \ref{sec:truncation} we prove various properties of the truncated statistic $T_M^{\circ}(H, G_n)$ (recall (\ref{def:trunc})). Using these properties we complete the proof of Theorem \ref{thm:normality} in Appendix \ref{sec:pf_normality}.

\subsection{Properties of the  Truncated Statistic $T_M^{\circ}(H, G_n)$} 
\label{sec:truncation}

In this section section, we collect some properties of the truncation \eqref{eq:C_truncation} and the truncated statistic \eqref{def:trunc}. For notational convenience define, 
\begin{align*}
r_H(A):=\frac1{|\ah|}\sqrt{\sum_{\substack{\s_1, \s_2  \in V(G_n)_{|V(H)|} \\ \bar \s_1\bigcap \bar \s_2=A}} M_H(\s_1)M_H(\s_2)}. 
\end{align*}
These counts are essentially the building blocks for $\vr[T(H, G_n)]$. To see this, recall that if $|\bar \s_1  \bigcap \bar \s_2 |=K$, then $\operatorname{Cov}[X_{\s_1}, X_{\s_2}]=p_n^{2{|V(H)|}-K}-p_n^{2|V(H)|} $. This means, 
\begin{align} 
\vr[T(H, G_n)] & =\frac1{|\ah|^2} \sum_{\substack{\s_1, \s_2 \in V(G_n)_{|V(H)|} \\ \bar \s_1 \bigcap \bar \s_2 \ne \emptyset}}  M_H(\s_1)M_H(\s_2) \operatorname{Cov}[X_{\s_1},X_{\s_2}] \nonumber \\ 
& = \frac1{|\ah|^2}\sum_{K=1}^{|V(H)|}  \sum_{\substack{\s_1, \s_2 \in V(G_n)_{|V(H)|} \\ 
 | \bar \s_1 \bigcap \bar \s_2 | = K }} M_H(\s_1)M_H(\s_2)(p_n^{2|V(H)|- K}-p_n^{2 |V(H)| })  \nonumber \\ 
\label{e:blockvar}  & =  \sum_{\substack{A \subset V(G_n) \\ 1\leq |A| \leq |V(H)|}} p_n^{2|V(H)|-|A|}(1-p_n^{|A|}) r_H(A)^2 . 
\end{align} 
We begin by showing that the order of variance of $T(H, G_n)$ remain the same if the functions $r_H(A)$ are replaced by the local count functions $t_H(A)$. 

\begin{lemma}\label{compare} Define 
$$\beta_H(p_n):= \sum_{K=1}^{|V(H)|} p_n^{2 |V(H)| - K}\sum_{A \subset V(G_n): |A| = K} t_H(A)^2.$$ Then 
\begin{align} \label{e:tvar}
\frac{(1-p_n) \beta_H(p_n)}{2^{|V(H)|}-1} \le  \vr[T(H, G_n)] \le \beta_H(p_n) .
\end{align}
\end{lemma}

\begin{proof} Recalling \eqref{e:blockvar} and using the bounds  $r_H(A)^2\le t_H(A)^2$ and $1-p_n^K \le 1$, gives $\vr[T(H, G_n)] \le \beta_H(p_n)$. 

For the other side, recalling the definition of $t_H(A)$ from \eqref{def:tj} we note the following identity,  
\begin{align*}
t_H(A)^2  & =\frac1{|\ah|^2}\sum_{\bar \s_1 \bigcap \bar\s_2\supseteq A }M_H(\s_1)M_H(\s_2) \\ & = \frac1{|\ah|^2} \sum_{A' \supseteq A}\sum_{\substack{\bar \s_1 \bigcap \bar \s_2= A' }}M_H(\s_1)M_H(\s_2) = \sum_{A' \supseteq A} r_H(A')^2
	\end{align*}
Taking sum over $A$ such that $|A|=K$  gives, 
\begin{align} \label{e:t-r-iden}
\sum_{A \subseteq V(G_n): |A|=K} t_H(A)^2 & = \sum_{\substack{A' \subseteq V(G_n) \\ K \leq |A'| \leq |V(H)|}} \sum_{\substack{A \subseteq A' \\ |A| = K}} r_H(A')^2 = \sum_{\substack{A' \subseteq V(G_n) \\ K \leq |A'| \leq |V(H)|}}  \binom{|A'|}{K}  r_H(A')^2 . 
	\end{align}
We now use the RHS of \eqref{e:t-r-iden} to rewrite $\beta_H(p_n)$. Interchanging the order of the sum we get
	\begin{align*}
	\beta_H(p_n) & = \sum_{K=1}^{|V(H)|} p_n^{2 |V(H)| - K}  \sum_{\substack{A' \subseteq V(G_n) \\ K \leq |A'| \leq |V(H)|}}  \binom{|A'|}{K}  r_H(A')^2 \\ 
		& = \sum_{\substack{A \subseteq V(G_n) \\ 1 \leq |A| \leq |V(H)|}}   r_H(A)^2 \sum_{K=1}^{|A|} p_n^{2{|V(H)|}-K} \binom{|A|}{K} \\ 
		& \le \label{e:upblock} \sum_{\substack{A \subseteq V(G_n) \\ 1 \leq |A| \leq |V(H)|}}   r_H(A)^2 p_n^{2 |V(H)|-|A|} (2^{|A|}-1) \tag*{(using the bound $p_n^{2{|V(H)|}-K} \leq p_n^{2{|V(H)|}-|A|}$)}\\ 
		& \le \frac{2^{|V(H)|}-1}{1-p_n} \sum_{\substack{A \subseteq V(G_n) \\ 1 \leq |A| \leq |V(H)|}}  r_H(A)^2 p_n^{2 |V(H)| - |A|}(1-p_n^{|A|})  \le \frac{2^{|V(H)|}-1}{1-p_n}  \vr[T(H, G_n)] ,
	\end{align*}
where the last step uses \eqref{e:blockvar}. 
	\end{proof}
	
Now, recall the definition of the truncated statistic $T_M^{\circ}(H, G_n)$ from \eqref{def:trunc}. Clearly, for each fixed $n$ as $M \to \infty$, $T_M^{\circ}(H, G_n) \uparrow T(H, G_n)$. 
The following lemma shows that this convergence is in probability and in $L^1$ (after proper scaling) uniformly in $n$. 
	
	\begin{lemma}\label{properties} The truncated statistic $T_M^{\circ}(H, G_n)$ defined in \eqref{def:trunc} has the following properties: 
		\begin{enumerate}
			\item[$(a)$] $\Pr(T(H, G_n)\neq T_M^{\circ}(H, G_n)) \le \frac{2^{|V(H)|} - 1}{M(1-p_n)}$.
			\item[$(b)$] $\frac1{\sqrt{\vr[T(H, G_n)]}}\Ex|T(H, G_n)-T_M^{\circ}(H, G_n)| \le \frac{2^{|V(H)|}-1}{\sqrt{M}(1-p_n)}$.
		\end{enumerate}
	\end{lemma}
	\begin{proof} 
	Observe that
		\begin{align*}
		\Pr(T(H, G_n) \neq T_M^{\circ}(H, G_n)) & \le  \sum_{\substack{A \subseteq V(G_n) \\ 1 \leq |A| \leq |V(H)|}}   \Pr(X_{A}=1)\ind\{\mathcal{C}_M(A)\} \\ 
& = \sum_{\substack{A \subseteq V(G_n) \\ 1 \leq |A| \leq |V(H)|}}   p_n^{|A|} \ind\{\mathcal{C}_M(A)\} \\
& \leq \sum_{\substack{A \subseteq V(G_n) \\ 1 \leq |A| \leq |V(H)|}}   \frac{p_n^{2|V(H)| -  |A|} }{M\vr[T(H, G_n)]} t_H(A)^2    \\ 
& \leq \frac{2^{|V(H)|} - 1}{M (1-p_n) } ,  
				\end{align*}
where the last step uses \eqref{e:tvar}. This completes the proof of (a). 

Note that
		\begin{align*}
		\Ex|T(H, G_n)-T_M^{\circ}(H, G_n)| & =\frac{p_n^{|V(H)|}}{|\ah|}\sum_{\s\in V(G_n)_{|V(H)|}} M_{H}(\bm s) \ind\{\mathcal{C}_M(\s)\} \\ 
		& \le \frac{p_n^{|V(H)|}}{|\ah|}\sum_{K=1}^{|V(H)|}\sum_{\s\in V(G_n)_{|V(H)|}}  M_{H}(\bm s)  \sum_{ A \subseteq \s : |A| = K }\ind\{\mathcal{C}_M(A)\} \\ 
		&  \le \frac{p_n^{|V(H)|}}{|\ah|}\sum_{K=1}^{|V(H)|}\sum_{A \subset V(G_n): |A| = K}\frac{p_n^{|V(H)|-K}t_H(A)}{\sqrt{M\vr[T(H, G_n)]}} \sum_{\s\supseteq A} M_H(\bm s)  \tag*{(recall \eqref{eq:C_truncation})}\\ 
		& \le \frac{1}{\sqrt{M\vr[T(H, G_n)]}}\sum_{K=1}^{|V(H)|}\sum_{A \subset V(G_n): |A| = K} p_n^{2 |V(H)|-K}t_H(A)^2,  \tag*{(recall \eqref{def:tj})}\\ 
	& 	\le \frac{(2^{|V(H)|} - 1) \vr[T(H, G_n)]}{(1-p_n) \sqrt{M}},
		\end{align*}
where the last step uses \eqref{e:tvar}. This completes the proof of (b). 
\end{proof}

Next, we show that the truncation (as in \eqref{def:trunc}) ensures all the higher-order moments of $T_M^{\circ}(H, G_n)$ are bounded. 

\begin{lemma} \label{moment_r} Fix $M > 0$ and an integer $R \geq 1$. Then for the truncated statistic $T^{\circ}_M(H, G_n)$ as defined in \eqref{def:trunc}, 
\begin{align} \label{e:moment_r}
\limsup_{n\to\infty} \left| \Ex\left[\frac{T_M^{\circ}(H, G_n)-\Ex[T_M^{\circ}(H, G_n)]}{\sqrt{\vr[T(H, G_n)]}}\right]^{R} \right| \lesssim_{M, R} 1.
	\end{align}
\end{lemma}

\begin{proof} Note that it suffices to prove the result for $R$ even. For $R \geq 2$ even, let $\mathcal{P}_{R}$ denote the collection of all partitions of the set $[R]:=\{1,2,\ldots, R\}$ such that each 
subset of the partition has at least two elements. A partition $\bm \lambda \in \mathcal{P}_R$ will be denoted by $\bm \lambda = \{\lambda_1,\lambda_2, \ldots,\lambda_{|\bm \lambda|} \} $, where $|\bm \lambda|$ denotes the number of subsets in the partition and $\lambda_1,\lambda_2, \ldots,\lambda_{|\bm \lambda|} \subseteq [R]$ are the subsets in the partition $\bm \lambda$. 

Given a collection of $R$ tuples $\s_1,\ldots,\s_R \in V(G_n)_{|V(H)|}$, $\bm \mu[\s_1,\ldots,\s_R]$ will denote the partition of $[R]$ obtained by considering the connected components of the graph $\mathcal{G}(\s_1,\ldots,\s_R)$ as in Definition \ref{eq:connected}. Then for any collection $\{\s_1,\ldots,\s_R\}$ which is weakly connected (recall definition in \eqref{eq:weak_connected}), the partition $\bm \mu[\s_1,\ldots,\s_R]\in \mathcal{P}_R$.  Note that $\Ex[{Z}_{\s_1} {Z}_{\s_2} \ldots {Z}_{\s_R}] =0$, unless $\bm \mu[\s_1,\ldots,\s_R]$ is weakly connected. Therefore, recalling the definition of $T_M^{\circ}(H, G_n)$ from \eqref{def:trunc} and 
denoting by $Z_{\s}= X_{\s} - p_n^{|V(H)|}$, for $\s \in V(G_n)_{|V(H)|}$, gives, 
	\begin{align} \label{e:momineq}
	(\Ex[T_M^{\circ}&(H, G_n)]-\Ex[T_M^{\circ}(H, G_n)])^R \nonumber \\ 
	\nonumber & \le  \frac{1}{|\ah|^{R}}\sum_{\s_1, \ldots, \s_R \in V(G_n)_{|V(H)|}} \left|\Ex\left[\prod_{r=1}^{R}{Z}_{\s_r}\right] \right|\prod_{r=1}^{R}M_H(\s_r)\ind\{\mathcal{C}_M(\s_r)^c\}\\
	\nonumber& =\frac{1}{|\ah|^{R}}\sum_{\bm \lambda \in \mathcal{P}_R}  \sum_{\substack{\s_1, \ldots, \s_R \in V(G_n)_{|V(H)|}\\ \bm \mu[\s_1,\ldots,\s_R]=\bm \lambda}}\prod_{t=1}^{|\bm \lambda|} \left\{ \left|\Ex\left[\prod_{r \in \lambda_t} {Z}_{\s_r}\right] \right| \prod_{r \in \lambda_t}  M_H(\s_r)\ind\{\mathcal{C}_M(\s_r)^c\} \right\} \\
	\nonumber& \le \frac{1}{|\ah|^{R}}\sum_{\bm \lambda\in \mathcal{P}_R}  \prod_{t=1}^{|\bm \lambda|} \sum_{\{\s_r, r \in \lambda_t\}\in \mathcal{K}_{n, |\lambda_t|}} \left\{   \Ex \left| \prod_{r \in \lambda_t} {Z}_{\s_r}  \right| \prod_{r \in \lambda_t}  M_H(\s_r)\ind\{\mathcal{C}_M(\s_r)^c\} \right \}\\
	& =\frac{1}{|\ah|^{R}}\sum_{\bm \lambda\in \mathcal{P}_R}  \prod_{t=1}^{|\bm \lambda|}Q_{n,|\lambda_t|},
\end{align} 
where $\mathcal{K}_{n, R}$ is the set of all $R$ connected $|V(H)|$-tuples as in Definition \ref{eq:connected}, and, for $1 \leq r \leq R$,  
\begin{align}\label{def:claim2}
Q_{n, r}:=&\sum_{\{\s_1,\ldots, \s_r\}\in \mathcal{K}_{n, r} }\Ex\left|\prod_{a=1}^{r}{Z}_{\s_a}\right| \prod_{a=1}^{r}M_H(\s_a)\ind\{\mathcal{C}_M(\s_a)^c\} 
\end{align}

Now, define $\mathcal{N}_r:=\{\bm \theta := (\theta_1, \ldots, \theta_r) \in \mathbb{N}^r: \theta_1=|V(H)|, 1\le \theta_a \le |V(H)| -1, \text{ for } 2 \leq a \leq r \}$. For $\bm \theta=(\theta_1,\ldots,\theta_r) \in \mathcal{N}_r$, set
$$S(\bm \theta):=\sum_{\s_1,\ldots, \s_r \in \mathcal{A}_r(\bm \theta)} \prod_{a=1}^{r} M_H(\s_a)\ind\{\mathcal{C}_M(\s_a)^c\} . $$ 
where $\mathcal{A}_r(\bm \theta)$ is the collection of all $\s_1,\ldots, \s_r \in V(G_n)_{|V(H)|}$ such that $|\bar \s_a \bigcap (\bigcup_{b=1}^{a-1} \bar \s_b ) ^c|=\theta_a$ for $a \in \{2, 3, \ldots,  r \} $. Then it follows from \eqref{def:claim2} and Lemma \ref{abs-bound} that 
\begin{align}\label{eq:r-s}
Q_{n, r}\lesssim_r  \sum_{\bm \theta = (\theta_1,\ldots,\theta_r) \in \mathcal{N}_r}  p_n^{\sum_{a=1}^{r} \theta_a} S( \bm \theta ).
\end{align} 
We now claim that for any $\bm \theta=(\theta_1, \ldots,\theta_r)\in \mathcal{N}_r$, 
\begin{align}\label{eq:claim2}
S(\bm \theta)\lesssim_{ M} p_n^{-\sum_{a=1}^{r} \theta_a} \sigma(H,G_n)^r.
\end{align} 
Given \eqref{eq:r-s} and \eqref{eq:claim2}, it follows that $Q_{n, r}  \lesssim_{M, R} \sigma(H, G_n)^r$, for  $1 \leq r \leq R$. This implies, from \eqref{e:momineq} that
\[\Ex\big(T_M^{\circ}(H, G_n)-\Ex[T_M^{\circ}(H, G_n)]\big)^R \lesssim_{M, R} \sum_{\bm \lambda\in \mathcal{P}_R}\prod_{t=1}^{|\bm \lambda|} \sigma(H,G_n)^{|\lambda_t|} \lesssim_{H, M, R} \sigma(H,G_n)^R,\]
as desired in \eqref{e:moment_r}, where the last bound uses the fact that $\sum_{j=1}^{|\bm \lambda|} |\lambda_j|=R$ for every $\bm \lambda\in \mathcal{P}_R$. 

It thus remains to verify \eqref{eq:claim2}. To this effect, 
for any $K \in [1,|V(H)|-1]$ let $W=\{v_{1},v_{2},\ldots,v_{K}\}$ be a set of distinct vertices in $V(G_n)$. Consider the set $$V(G_n)_{V(H), W} =\{\s \in V(G_n)_{|V(H)|}=(s_1,s_2,\ldots,s_{|V(H)|})\in V(G_n)_{|V(H)|}\mid s_{i}=v_i, \text{ for }  1 \leq i \leq K \}.$$
In other words, $V(G_n)_{V(H), W}$ is the collection of $|V(H)|$-tuples such that the first $K$ coordinates of the tuple are fixed to $v_{1},v_{2},\ldots,v_{K}$, respectively.  Note that 
	\begin{align}\label{e:free-fix}
	\frac{1}{|\ah|} \sum_{\s \in V(G_n)_{V(H), W} }  M_H(\s)\ind\{\mathcal{C}_M(\s)^c\}
	& \le \frac{1}{|\ah|}\sum_{\s : \bar \s \supseteq  W } M_H(\s)  \ind\{\mathcal{C}_M(W)^c \} \nonumber  \\ 
	& =t_H(W) \ind\{\mathcal{C}_M(W)^c \} \nonumber \\ 
	 & \le p_n^{|W| - |V(H)|}\sqrt{M \vr[T(H, G_n)]}.
	\end{align} 
 Keeping $\s_1,\ldots,\s_{r-1}$ fixed,  the sum over $\s_{r}$ in $S(\theta_1, \ldots, \theta_r)$ has exactly $\theta_r$ free coordinates and $|V(H)|-\theta_r$ fixed coordinates. Denoting $\bar{\bm\theta}=(\theta_1,\ldots,\theta_{r-1})$, this implies, 
 	\begin{align*}
	S(\theta_1, \ldots, \theta_r)&= \sum_{ \s_1, \ldots, \s_r \in \mathcal{A}_r(\bm \theta) } \prod_{a=1}^{r}M_H(\s_a)\ind\{\mathcal{C}_M(\s_a)^c\}   \nonumber \\ 
	& = \sum_{ \s_1,\ldots, \s_{r-1} \in \mathcal{A}_{r-1}( \bar{\bm\theta})} \prod_{a=1}^{r-1}M_H(\s_a)\ind\{\mathcal{C}_M(\s_a)^c\}   \sum_{\substack{\s_{r}\in V(G_n)_{|V(H)|} \\ |\bar \s_r \bigcap \{\bigcup_{a=1}^{r-1} \bar \s_a \}^c|= \theta_r }} M_H(\s_r) \ind\{\mathcal{C}_M(\s_r)^c\} \\ 
& \lesssim \sum_{\s_1,\ldots, \s_{r-1} \in \mathcal{A}_{r-1}(\bar{\bm \theta}) } \prod_{a=1}^{r-1}M_H(\s_a)\ind\{\mathcal{C}_M(\s_a)^c\}  p_n^{-\theta_r}\sqrt{M \sigma(H, G_n)^2} \tag*{(using \eqref{e:free-fix})}\\
&\lesssim_{M} p_n^{-\theta_r} \sigma(H, G_n) S(\theta_1,\ldots,\theta_{r-1}).
	\end{align*}
Continuing in this way using induction we get the bound 
	\begin{align*}
	S(\theta_1, & \ldots,\theta_r)   \\ 
	& \lesssim_{M, R}  \frac{\sigma(H, G_n)^{r-2}}{ p_n^{\sum_{a=3}^r\theta_a} }\sum_{\substack{\s_1, \s_2 \in V(G_n)_{|V(H)|} \\ |\bar \s_1 \bigcap \bar \s_2|=|V(H)|-\theta_2}} M_H(\s_1)M_H(\s_2)\ind\{\mathcal{C}_M(\s_1)^c\} \ind\{\mathcal{C}_M(\s_2)^c\} \\ 
	& \lesssim_{M, R}  \frac{\sigma(H, G_n)^{r-2}}{p_n^{\sum_{a=1}^r\theta_a} } \sum_{\theta_2=0}^{|V(H)| - 1} \sum_{\substack{\s_1, \s_2 \in V(G_n)_{|V(H)|} \\ |\bar \s_1 \bigcap \bar \s_2|=|V(H)|-\theta_2}} p_n^{\theta_1+\theta_2} M_H(\s_1)M_H(\s_2)\ind\{\mathcal{C}_M(\s_1)^c\} \ind\{\mathcal{C}_M(\s_2)^c\} \\ 
& \lesssim_{M, R} p_n^{-\sum_{a=1}^r \theta_a} \sigma(H, G_n)^r , 
	\end{align*}
	where the last step uses the bound $\operatorname{Cov}[Y_{\s_1},Y_{\s_2}] \lesssim p_n^{|V(H)|+\theta_2}=p_n^{\theta_1+\theta_2}$, along with the expansion of $\sigma(H, G_n)^2=\vr[T(H, G_n)]$ in \eqref{e:blockvar}. This verifies \eqref{eq:claim2}, and hence completes the proof of Lemma \ref{moment_r}. \end{proof}

\subsection{Proof of Theorem \ref{thm:normality}} 
\label{sec:pf_normality} 
Recall $Z(H, G_n)$ as defined in \eqref{eq:ZHGn}. Define $$U_{n, M}:=\frac{T_M^{\circ}(H, G_n)-\Ex[T(H, G_n)] }{\sqrt{\vr[T(H, G_n)] }}, \quad V_{n, M}:=\frac{T_M^{\circ}(H, G_n)-\Ex[T_M^{\circ}(H, G_n)]}{\sqrt{\vr[T(H, G_n)] }},$$  
and $W_{n, M}:=\frac{T_M^{\circ}(H, G_n)-\Ex[T_M^{\circ}(H, G_n)]}{\sqrt{\vr[T_M^{\circ}(H, G_n)]}}.$ \\ 
		
\noindent\textit{Proof of Sufficiency}:  To begin, use \eqref{e:4th} to conclude
\begin{align}\label{eq:2nd2}
\limsup_{M \to\infty}\limsup_{n \to \infty}  \left|\frac{\vr[T_M^{\circ}(H, G_{n})]}{\vr[T(H, G_{n})]}-1\right|=0,
\end{align}
which along with \eqref{e:4th} gives  $\limsup_{M \to\infty}\limsup_{n \to\infty}|\Ex[W_{n, M }^4]-3|=0.$ Then applying \eqref{e:UGn} to the double sequence $W_{n, M}$ we get 
		\begin{align}\label{e:w-wass}
		\limsup_{M \to\infty}\limsup\limits_{n\to\infty}\operatorname{Wass}(W_{n, M },N(0,1))=0.
		\end{align} 
		Now, using  
\begin{align}\label{eq:UnM_VnM}
U_{n, M}=\frac{1}{\sqrt{\Ex[V_{n, M}^2}]}W_{n, M}-\frac{\Ex[T(H, G_n)]-\Ex[T_M^{\circ}(H, G_n)]}{\sqrt{\vr[T(H, G_n)]}},
\end{align} 
the assumption in \eqref{e:4th}, \eqref{eq:2nd2}, and Lemma \ref{properties} (b) gives $\limsup_{M \to\infty}\limsup_{n\to\infty}\operatorname{Wass}(U_{n, M },W_{n,M})=0$. This along with \eqref{e:w-wass} gives $\limsup_{M \to\infty}\limsup_{n\to\infty}\operatorname{Wass}(U_{n, M },N(0,1))=0$, that is, 
\begin{align}\label{e:u-wass}
\limsup_{M\rightarrow\infty}\limsup_{n\rightarrow\infty}\sup_{t\in \mathbb{R}}|\Pr(U_{n,M}\le t)-\Phi(t)|=0,
		\end{align} 
where $\Phi$ denotes the standard normal distribution function. Moreover, note that for any $t\in \mathbb{R}$ we have
\begin{align*}
\Pr(Z(H,G_n)\le t)\le \Pr(U_{n,M}\le t)+\Pr( T(H,G_n)\ne T_M^\circ(H,G_n))\le \Pr(U_{n,M}\le t)+\frac{C}{\sqrt{M}}
\end{align*}
for some finite constant $C$ (not depending on $n$ and $M$) by Lemma \ref{properties} (b).  Next, noting that $U_{n,M}\le Z(H,G_n)$, we also have the lower bound $\Pr(Z(H,G_n)\le t)\ge \Pr(U_{n,M}\le t)$. Combining we get
\begin{align}\label{eq:iff}\sup_{t\in \R}|\Pr(Z(H,G_n)\le t)-\Pr(U_{n,M}\le t)|\le \frac{C}{\sqrt{M}}.
\end{align}
Therefore, taking limits as $n\rightarrow\infty$ followed by $M\rightarrow\infty$  gives 
\[\limsup_{M\rightarrow\infty}\limsup_{n\rightarrow\infty}\sup_{t\in \R}|\Pr(Z(H,G_n)\le t)- \Pr(U_{n,M}\le t)|=0.\]
This implies, by \eqref{e:u-wass}, $\limsup_{n\rightarrow\infty}\sup_{t\in \R}|\Pr(Z(H,G_n)\le t)-\Phi(t)|=0$, completing the proof of the sufficiency in Theorem \ref{thm:normality}. \\

\noindent\textit{Proof of Necessity}: Observe that 
\begin{align}\label{eq:observe}
V_{n, M}=U_{n, M}+\frac{\Ex[T(H, G_n)]-\Ex[T_M^{\circ}(H, G_n)]}{\sqrt{\vr[T(H, G_n)]}}. 
\end{align}
Then invoking Lemma \ref{properties} (b) it suffices to show that 
\begin{align}\label{eq:UnM_I}
\limsup_{M\rightarrow\infty}\limsup_{n\rightarrow\infty}|\Ex[U_{n,M}^2]- 1 | \quad \text{ and } \quad \limsup_{M\rightarrow\infty}\limsup_{n\rightarrow\infty}|\Ex[U_{n,M}^4] - 3 |. 
\end{align} 
To begin with, since $Z(H, G_n) \stackrel{D}{\to} N(0,1)$ by assumption and  using \eqref{eq:iff} it follows that
\begin{align*}
\limsup_{M\rightarrow\infty}\limsup_{n\rightarrow\infty}\sup_{t\in \mathbb{R}}|\Pr(U_{n,M}\le t)-\Phi(t)|=0. 
\end{align*}
Therefore, using uniform integrability, to show \eqref{eq:UnM_I} it suffices to prove that 
\begin{align}\label{eq:UnM_II}
\limsup_{M\rightarrow\infty}\limsup_{n\rightarrow\infty}\Ex[U_{n,M}^6] <\infty. 
\end{align} 
By way of contradiction, assume this does not hold, that is, $\limsup_{M\rightarrow\infty}\limsup_{n\rightarrow\infty}\Ex[U_{n,M}^6]=\infty$. Then there exists $M$ such that $\limsup_{n\rightarrow\infty}\Ex[U_{n,M}^6]>\limsup_{n\rightarrow\infty}\Ex[U_{n,1}^6]+\Ex[N(0,1)^6]$. By passing to a subsequence (which depends on the choice of $M$), without loss of generality we can assume 
\begin{align}\label{eq:U_nM_moment}
\lim_{n\rightarrow\infty}\Ex[U_{n,M}^6]>\limsup_{n\rightarrow\infty}\Ex[U_{n,1}^6] +\Ex[N(0,1)^6]. 
\end{align} 
Now, since $U_{n,M}$ is non-decreasing in $M$, $$-U_{n,M}\mathbf{1}\{U_{n,M}\le 0\}\le -U_{n,1}1\{U_{n,1}\le 0\}\le |U_{n,1}|.$$ This gives, using  \eqref{eq:U_nM_moment},  
\begin{align}\label{eq:UnM_III}
\lim_{n\rightarrow\infty}\Ex[U_{n,M}^6\mathbf{1}\{U_{n,M}>0\}]> \Ex[N(0,1)^6].
\end{align}
Moreover, for every $M > 0$ fixed, Lemma \ref{moment_r} shows that $\{U_{n,M}\}_{n \geq 1}$ is tight, and, hence, by passing to a further subsequence we can assume that $U_{n,M}$ converges in distribution to a random variable, which we denote by $U(M)$, as $n \rightarrow \infty$. Then taking limit as $n \rightarrow \infty$ in \eqref{eq:UnM_III} gives 
\begin{align}\label{eq:UM_moment}
\Ex[U(M)^6 \bm 1\{U(M)\ge 0\}] >\Ex[N(0,1)^6] .
\end{align} 
since $\lim_{n \rightarrow \infty}\Ex[U_{n,M}^6\mathbf{1}\{U_{n,M}>0 \} ] =   \Ex[U(M)^6 \bm 1\{U(M)\ge 0 \}]$ by the boundedness of the moments. However, $U(M)$ is stochastically smaller than $N(0,1)$, which implies, $\Ex[U(M)^6 \bm 1\{U(M)\ge 0\} ] \le \frac{1}{2}\Ex[N(0,1)^6] ,$ a contradiction to the \eqref{eq:UM_moment}. This proves \eqref{eq:UnM_II} and completes the proof of the necessity in Theorem \ref{thm:normality}.

\section{Examples}
\label{sec:examples}

In this section we discuss various examples which illustrate the necessity of the different conditions in the results obtained above. We begin with an example where the HT estimator is inconsistent for estimating $N(K_2, G_n)$, where the first condition in \eqref{eq:edge_condition} holds, but the second condition fails, which shows the necessity of truncating on the high-degree vertices to establish consistency.

\begin{example} 
\label{example:Gn_consistent}
(An inconsistent example) Take $G_n = K_{1, n}$ be the $n$-star graph, $H=K_2$ to be an edge, and $p_n =  \frac12$. Then $T(H, G_n)=0$ with probability $\frac{1}{2}$ (when the central vertex of the star is not chosen), and $T(H, G_n) \sim \mathrm{Bin}(n-1,\frac{1}{2})$ with probability $\frac{1}{2}$ (when the central vertex of the star is chosen). Consequently, noting that $\hat N(K_2, G_n) = \frac1{p_n^2}T(K_2, G_n) =4 T(K_2, G_n) $, it follows that
	$$\frac{\hat N(K_2, G_n)}{N(K_2, G_n)} = \frac{\hat N(K_2, G_n)}{n -1} \stackrel{D}{\rightarrow} \tfrac{1}{2} \delta_0 + \tfrac{1}{2} \delta_2,$$
where $\delta_a$ denotes the point mass at $a \in \R$. In particular, this shows that $\hat N(K_2, G_n)$ is inconsistent for $N(K_2, G_n)$. This is because, while the first condition in \eqref{eq:edge_condition} holds, the second condition fails, because $$\displaystyle\frac{1}{N(K_2,G_n)}\sum_{v = 1}^{|V(G_n)|} d_v\ind \{d_v>\e p_n|E(G_n)|\}=1,$$
for every $\e\in (0,2)$. 
\end{example}

Now, we construct an example where the HT estimator is consistent but its limiting distribution is non-Gaussian, in fact, it is discrete.

\begin{example}\label{example:consistent} (Non-Gaussian Limiting Distribution)  Consider a graph $G_n$ which has $r_n$ many disjoint $a_n$-stars, and $r_n$ many disjoint $b_n$ cliques, such that
\begin{align}\label{eq:abr}
r_n+b_n^{\frac{3}{2}}\ll a_n\ll b_n^2. 
\end{align} 
where $a_n, b_n, r_n$ are all integer sequences diverging to infinity. Note that $|V(G_n)|=r_n(a_n+1+b_n) = (1+o(1)) r_na_n$. Then with $H=K_2$ we have $N(K_2,G_n)=r_na_n+r_n{b_n\choose 2} = (\tfrac{1}{2}+o(1)) r_nb_n^2 $. In this case, 
\begin{align}\label{eq:THGn_example} 
T(K_2, G_n)=\sum_{i=1}^{r_n}\left[ X_iY_i+{Z_i\choose 2}\right], 
\end{align}
where $\left(\{X_i\}_{i=1}^{r_n}, \{Y_i\}_{i=1}^{r_n}, \{Z_i\}_{i=1}^{r_n}\right)$ are mutually independent, with $\{X_i\}_{i=1}^{r_n}$  i.i.d. $\mathrm{Ber}(1/r_n)$, $\{Y_i\}_{i=1}^{r_n}$ i.i.d. $\mathrm{Bin}(a_n, 1/r_n)$ and $\{Z_i\}_{i=1}^{r_n}$ i.i.d. $\mathrm{Bin}(b_n, 1/r_n)$. 
Therefore, 
\begin{align}\label{eq:evar2_I}
\Ex\left[\sum_{i=1}^{r_n}X_iY_i \right] = \frac{a_n}{r_n}, \quad \vr\left[\sum_{i=1}^{r_n}X_iY_i \right] = (1+o(1)) r_n\left(\frac{a_n}{r_n^2} + \frac{a_n^2}{r_n^3} \right) = (1+o(1)) \frac{a_n^2}{r_n^2},
\end{align}
and 
\begin{align}\label{eq:evar2_II}
\Ex\left[\sum_{i=1}^{r_n}{Z_i\choose 2} \right] =\frac{1}{r_n} {b_n \choose 2} = (1+o(1))  \frac{b_n^2}{2r_n}, \quad \vr\left[\sum_{i=1}^{r_n}{Z_i\choose 2}\right] = (1+o(1)) \frac{b_n^3}{r_n^2}.
\end{align}
Using \eqref{eq:evar2_I} and \eqref{eq:evar2_II} along with \eqref{eq:abr} in 
\eqref{eq:THGn_example} gives 
\[\Ex[T(K_2,G_n)] = (1+o(1)) \frac{ b_n^2}{2r_n} \text{ and }  \vr[T(K_2,G_n)] = (1+o(1)) \frac{a_n^2}{r_n^2}=o(\Ex[T(K_2,G_n)])^2,\] 
which shows $\hat N(K_2, G_n)$ is consistent for $N(K_2, G_n)$. However, in this case the asymptotic distribution of $\hat N(K_2, G_n)$ is non-normal. In particular, we will show that 
\begin{align}\label{eq:ZHGn_example}
Z(K_2, G_n)\stackrel{D}{\rightarrow} \mathrm{Pois}(1)-1. 
\end{align} 
Indeed, note that 
\begin{align*}
& \frac{T(K_2, G_n)-\Ex[T(K_2, G_n)]}{\sqrt{\vr[T(K_2, G_n)]}} \\ 
& =\frac{1}{\sqrt{\vr[T(K_2, G_n)]}} \left[\sum_{i=1}^{r_n}(X_iY_i-\Ex[X_i Y_i] )+\sum_{i=1}^{r_n}\left({Z_i\choose 2}-\Ex{Z_i\choose 2}\right) \right] \\
& = \frac{1}{\sqrt{\vr[T(K_2, G_n)]}}\sum_{i=1}^{r_n}X_iY_i-\frac{\sum_{i=1}^{r_n} \Ex[X_i Y_i] }{\sqrt{\vr[T(K_2, G_n)]}} +O_P\left(\frac{\sqrt{\vr[\sum_{i=1}^{r_n}{Z_i\choose 2}]}}{\sqrt{\vr[T(K_2, G_n)]}}\right),
\end{align*}
where the second and third terms converge to $-1$ and $0$ respectively using \eqref{eq:evar2_I} and \eqref{eq:evar2_II}, respectively. Therefore, to complete the proof of \eqref{eq:ZHGn_example}, it suffices to show that \[\frac{1}{\sqrt{\vr[T(K_2, G_n)]}}\sum_{i=1}^{r_n}X_iY_i\stackrel{D}{\rightarrow} \mathrm{Pois}(1).\] This follows by noting that 
\[\frac{1}{\sqrt{\vr[T(K_2, G_n)]}}\sum_{i=1}^{r_n}X_iY_i \stackrel{D}=\frac{r_n}{a_n} \mathrm{Bin}\left(a_n \sum_{i=1}^{r_n} X_i, \frac{1}{r_n} \right) + o_P(1) \stackrel{D}{\rightarrow} \mathrm{Pois}(1), \]
as $a_n/r_n\rightarrow\infty$. 
\end{example}

The next example illustrates the necessity of assuming $p_n$ to be bounded away from 1 (in particular $p_n \in (0, \frac{1}{20}]$) for the limiting normality and the fourth-moment phenomenon of the HT estimator. In particular, this example constructs a sequence of graphs $\{G_n\}_{n \geq 1}$ for which if $p_n$ is chosen to be large enough, then even though $\Ex[Z(K_2, G_n)^4] \rightarrow 3$, $Z(K_2, G_n)$ does not converge to $N(0, 1)$.

\begin{example}\label{example:ab} (Why an upper bound on the sampling ratio is necessary?)  Let $H=K_2$ be the edge and $G_n$ be the disjoint union of an $a_n$-star and $b_n$ disjoint edges with $a_n\ll b_n\ll a_n^2$. Then as before \eqref{eq:ve} holds. 
Fix a sampling ratio $p_n=p$ free of $n$, where the exact value of $p$ will be chosen later. Then we have 
\begin{align}\label{eq:bin_rep}
T(K_2, G_n)=X_nY_n+Z_n,
\end{align}
where $X_n\sim \operatorname{Ber}(p)$, $Y_n\sim \operatorname{Bin}(a_n, p)$ and $Z_n \sim \operatorname{Bin}(b_n, p^2)$ are independent. Note that 
\begin{align}\label{eq:evar}
\vr[X_nY_n] & =a_np^2(1-p )+ a_n^2 p^3(1-p ) = (1+o(1)) p^3(1-p)a_n^2, ~  \vr[Z_n] = b_np^2(1-p^2).
\end{align}
Since $Z_n\sim \mathrm{Bin}(b_n,p^2)$, 
\begin{align}\label{eq:Zn_example}
\frac{Z_n-\Ex[Z_n]}{\sqrt{\mathrm{Var}[Z_n]}}\stackrel{D}{\rightarrow}N(0,1) \text{ which implies, } \frac{Z_n-\Ex[Z_n]}{\sqrt{\mathrm{Var}[T(K_2, G_n)]}}\stackrel{P}{\rightarrow}0,
\end{align}
where the second conclusion uses $\vr[T(K_2, G_n)]=\vr[X_nY_n]+\vr[Z_n] \gg \vr[Z_n]$ (by \eqref{eq:evar}). Thus, on the set $\{X_n=0\}$ (which happens with positive probability $p$) we have
\[Z(K_2, G_n) = \frac{T(K_2, G_n)-\Ex[T(K_2, G_n)]}{\sqrt{\mathrm{Var}[T(K_2, G_n]}}=-\frac{\Ex[X_n Y_n]}{\sqrt{\mathrm{Var}[T(K_2, G_n)]}}+\frac{Z_n-\Ex[Z_n]}{\sqrt{\vr[T(K_2, G_n)]}}\stackrel{P}{\rightarrow}-\sqrt{\frac{p}{1-p}},\]
as the first term converges to $-\sqrt{\frac{p}{1-p}}$ by \eqref{eq:evar}, and the second term converges in probability to $0$ by \eqref{eq:Zn_example}. This shows that any limiting distribution for $Z(K_2,G_n)$ has a point mass, and hence cannot be $N(0,1)$.

To demonstrate that the fourth moment phenomenon indeed fails in this case, we will now show that $\Ex[Z(K_2, G_n)^4]\rightarrow 3$ for a proper of choice of $p$. Towards this, note that \begin{align}\label{e:rel}
\Ex[T(H, G_n) & -\Ex[T(H, G_n)]]^4-3\vr[T(H, G_n)]^2 \nonumber \\ 
& = \Ex[X_nY_n-\Ex[X_nY_n] ]^4-3\vr[X_nY_n]^2+\Ex[Z_n - \Ex[Z_n] ]^4-3\vr[Z_n]^2.
\end{align}
Now, a simple calculation shows, 
\begin{align*}
\Ex[X_nY_n-\Ex[X_nY_n]]^4-3\vr[X_nY_n]^2= O(a_n^3) + a_n^4 g(p), 
\end{align*}
where $g(p)=p^4(1-p)(1-3p+3p^3).$ Note that the function $1-3p+3p^3$ has two roots $p_1 \approx 0.39 $ and $p_2 \approx 0.74$ inside $[0,1]$. Hence, choosing $p=p_1$ gives, 
$$|\Ex[X_nY_n-\Ex[X_nY_n]]^4-3\vr[X_nY_n]^2| =O( a_n^3).$$
Moreover, \[\Ex[Z_n-\Ex[Z_n] ]^4-3\vr[Z_n]^2=b_n p(1-p)(1-6p+6p^2)=O(b_n).\] 
Hence, \eqref{eq:evar} and \eqref{e:rel} give 
$$\Ex[Z(H, G_n)]^4-3 = \frac{\Ex[T(H, G_n) -\Ex[T(H, G_n)]]^4}{\vr[T(H, G_n)]^2 } - 3 = O \left(\frac{a_n^3+b_n}{a_n^4} \right) \rightarrow 0.$$  This shows that some upper bound on the sampling ratio $p$ is necessary to obtain the fourth-moment phenomenon of $Z(H, G_n)$.  
\end{example}

Finally, we construct a sequence of graphs $\{G_n\}_{n \geq 1}$ for which $Z(K_2, G_n) \stackrel{D} \rightarrow N(0, 1)$, but $\Ex[Z(K_2, G_n)] \nrightarrow 3$, that is, the (untruncated) fourth-moment condition is not necessary for normality. This illustrates the  need to consider the truncated fourth-moment condition as in Theorem \ref{thm:normality}, which gives a necessary and sufficient condition for the limiting normal distribution of the HT estimator.

\begin{example}\label{example:4_normal}(Fourth moment phenomenon is not necessary) Let $H=K_2$ be the edge, and $G_n$ be the disjoint union of an $a_n$-star and $b_n$ disjoint edges, with $a_n\ll b_n\ll a_n^2$. Suppose that the sampling ratio $p_n$ satisfies
 \begin{align}\label{eq:pn_condition_II}
 \frac{1}{a_n}\ll p_n \ll \frac{b_n}{a_n^2}.
 \end{align}
Then as in Example \ref{example:ab}, $$T(K_2, G_n)=X_nY_n+Z_n,$$
where $X_n\sim \operatorname{Ber}(p_n)$, $Y_n\sim \operatorname{Bin}(a_n, p_n)$, and $Z_n \sim \operatorname{Bin}(b_n, p_n^2)$ are independent. Then by calculations similar to \eqref{eq:evar} with $p$ replaced by $p_n$, we have  
 \begin{align*}
\lim_{n\rightarrow\infty}\frac{\vr[X_nY_n]}{\vr[T(H,G_n)]}=0 \quad \text{ and } \quad \lim_{n\rightarrow\infty}\frac{\vr[Z_n]}{\vr[T(H,G_n)]} = 1.
\end{align*} 
Using Slutksy's theorem, this gives
\begin{align*}
Z(H, G_n) = \frac{X_nY_n - \Ex[X_nY_n] }{\sqrt{\vr[T(H, G_n)]}} + \sqrt{\frac{\vr[T(H, G_n)]}{\vr[Z_n]}}\cdot\frac{Z_n-\Ex[Z_n]}{\sqrt{\vr[Z_n]}}  & \stackrel{D} \to N(0,1).
\end{align*}
To show that the converse of the (untruncated) fourth moment phenomenon fails, we now show that 
$\lim_{n\rightarrow\infty}\Ex[Z(K_2,G_n)^4]=\infty.$ Recalling \eqref{e:rel}, it suffices to show that
\begin{align}\label{eq:Zn_4example}
\lim_{n\rightarrow\infty}\frac{\Ex[X_nY_n-\Ex[X_nY_n] ]^4}{\vr[T(H,G_n)]^2}=\infty.
\end{align}
To this end, using \eqref{eq:pn_condition_II} note that
\[\Ex[X_n^4 Y_n^4] = \Ex[X_n^4] \Ex[Y_n^4] = (1+o(1)) a_n^4p_n^5\gg b_n^2p_n^4 \text{ and } \Ex[X_nY_n]^4=a_n^4p_n^8 \ll  b_n^2p_n^4.\] 
The result in \eqref{eq:Zn_4example} then follows by using \eqref{eq:evar}. 
\end{example}

	\end{document}